\documentclass[12pt]{article}

\usepackage{amsmath,amssymb,amsthm}
\usepackage[colorlinks=true,allcolors=black,bookmarksopen,bookmarksdepth=3]{hyperref}
\usepackage[margin=1in]{geometry}
\usepackage{tikz-cd}
\usepackage{tikz}

\usepackage{enumitem}
\setlist{topsep=0em,partopsep=0em,parsep=0em,itemsep=0em}

\newtheorem{theorem}{Theorem}[section]
\newtheorem{lemma}[theorem]{Lemma}
\newtheorem{proposition}[theorem]{Proposition}
\newtheorem{corollary}[theorem]{Corollary}
\newtheorem{conjecture}[theorem]{Conjecture}

\theoremstyle{definition}
\newtheorem{definition}[theorem]{Definition}

\theoremstyle{remark}
\newtheorem{remark}[theorem]{Remark}

\newtheorem{question}[theorem]{Question}

\numberwithin{equation}{section}

\newcommand\cE{\mathcal E}

\newcommand\cJ{\mathcal J}

\newcommand\cL{\mathcal L}
\newcommand\cM{\mathcal M}

\newcommand\cMbarprime{{\vphantom{\mathcal M}\smash{\overline{\mathcal M}}}'}
\newcommand\cO{\mathcal O}
\newcommand\cR{\mathcal R}
\newcommand\cU{\mathcal U}
\newcommand\cUbar{\overline{\mathcal U}}

\newcommand\cZ{\mathcal Z}
\renewcommand\AA{\mathbb A}
\newcommand\CC{\mathbb C}
\newcommand\DD{\mathbb D}
\newcommand\FF{\mathbb F}

\newcommand\LL{\mathbb L}
\newcommand\QQ{\mathbb Q}
\newcommand\RR{\mathbb R}
\newcommand\ZZ{\mathbb Z}
\newcommand\bm{\mathbf m}
\newcommand\bM{\mathbf M}

\newcommand\Bl{\mathrm{Bl}}
\newcommand\An{\operatorname{An}}
\newcommand\ord{\operatorname{ord}}
\newcommand\Nbd{\operatorname{Nbd}}

\newcommand\ch{\operatorname{ch}}
\newcommand\Spec{\operatorname{Spec}}

\newcommand\Sym{\operatorname{Sym}}
\newcommand\supp{\operatorname{supp}}
\newcommand\colim{\operatornamewithlimits{colim}}
\newcommand\ev{\mathrm{ev}}
\newcommand\abs[1]{\lvert#1\rvert}

\newcommand\codim{\operatorname{codim}}
\newcommand\Proj{\operatorname{Proj}}
\newcommand\tighttimes{{\times}}

\newcommand\aff{\mathrm{aff}}
\newcommand\inj{\mathrm{inj}}

\newcommand\fin{\mathrm{fin}}
\newcommand\naive{\mathrm{naive}}
\newcommand\simple{\mathrm{simple}}

\newcommand\tors{\mathrm{tors}}

\newcommand\semiFano{\mathrm{sF}}
\newcommand\semiregular{\mathrm{sr}}
\newcommand\smooth{\mathrm{sm}}
\newcommand\ACpx{\mathrm{ACpx}}
\newcommand\Cpx{\mathrm{Cpx}}
\newcommand\Univ{\mathrm{Univ}}
\newcommand\prev{\mathrm{prev}}
\newcommand\PT{\mathrm{PT}}
\newcommand\GW{\mathrm{GW}}

\newcommand\vir{\mathrm{vir}}
\newcommand\pt{\mathrm{pt}}
\newcommand\relinfty{{\mathrm{rel}\infty}}

\makeatletter
\def\varcolim@#1#2{%
  \vtop{\m@th\ialign{##\cr
    \hfil$#1\operator@font colim$\hfil\cr
    \noalign{\nointerlineskip\kern1.5\ex@}#2\cr
    \noalign{\nointerlineskip\kern-\ex@}\cr}}%
}
\def\dircolim{%
  \mathop{\mathpalette\varcolim@{\rightarrowfill@\textstyle}}\nmlimits@
}
\makeatother
\newcommand\invlim\varprojlim

\newcounter{narrows}
\newcounter{nnarrows}
\makeatletter
\newcommand\arrowstack[1]{%
\setcounter{narrows}{0}%
\@for\aaa:=#1\do{\stepcounter{narrows}}%
\setcounter{nnarrows}{0}%
\renewcommand\arraystretch{0}%
\setlength\tabcolsep{0pt}%
\mathrel{\begin{tabular}{c}%
\@for\arrowtype:=#1\do{%
\if\arrowtype r$\rightarrow$\fi%
\if\arrowtype l$\leftarrow$\fi%
\if\arrowtype R$\longrightarrow$\fi%
\if\arrowtype L$\longleftarrow$\fi%
\stepcounter{nnarrows}%
\ifnum\value{narrows}=\value{nnarrows}\else\\[0pt]\fi}%
\end{tabular}}}
\makeatother

\begin{document}

\title{Universally counting curves in Calabi--Yau threefolds}

\author{John Pardon\thanks{This research was conducted during the period the author was partially supported by a Packard Fellowship and by the National Science Foundation under the Alan T.\ Waterman Award, Grant No.\ 1747553.}}

\date{August 26, 2025}

\maketitle

\begin{abstract}
We show that curve enumeration invariants of complex threefolds with nef anti-canonical bundle are determined by their values on local curves.
This implies the MNOP conjecture of Maulik, Nekrasov, Okounkov, and Pandharipande relating Gromov--Witten and Donaldson--Pandharipande--Thomas invariants, for all complex threefolds with nef anti-canonical bundle (in particular, all Calabi--Yau threefolds) and primary insertions (no descendents), given its known validity for local curves due to Bryan, Okounkov, and Pandharipande.
The main new technical ingredient in our work is a generic transversality result for holomorphic curves in complex manifolds.
Due to the rigidity of complex structures, this result is necessarily weaker than the corresponding generic transversality property for holomorphic curves in almost complex manifolds.
Despite this weaker nature, it is enough to obtain our main result by following the proof of the Gopakumar--Vafa integrality conjecture by Ionel and Parker.
\end{abstract}

\section{Introduction}

There are many ways of enumerating curves in complex threefolds \cite{pandharipandethomassurvey}.
These invariants turn out to satisfy some surprising relations which appear to have no straightforward explanation.
In fact, according to Pandharipande--Thomas \cite{pandharipandethomasstablepairs}, the multitude of existing computations suggest that all reasonable curve enumeration theories for complex threefolds are equivalent, despite arising from quite varied geometric origins.

A folk conjecture offers an explanation of this phenomenon: a complex threefold should be `enumeratively equivalent' to a linear combination of \emph{local curves} (rank two vector bundles over smooth proper curves).
We provide a precise formulation and proof of this conjecture for complex threefolds with \emph{nef anti-canonical bundle}.
That is, we define a certain \emph{Grothendieck group of 1-cycles} in complex threefolds (with nef anti-canonical bundle) (Definition \ref{grothintrodef}), and we show that this group is \emph{freely generated by local curves} (Theorem \ref{maincalculation}).

This result and its proof were inspired by the proof of the Gopakumar--Vafa integrality conjecture by Ionel--Parker \cite{ionelparker}.
They showed that Gromov--Witten invariants of almost complex threefolds are integer linear combinations of Gromov--Witten invariants of local curves, which were known to satisfy Gopakumar--Vafa integrality by work of Bryan--Pandharipande \cite{bryanpandharipande}.
Their argument may be interpreted as a proof that a certain Grothendieck group of 1-cycles in \emph{almost} complex threefolds with nef anti-canonical bundle is freely generated by local curves (after completing by genus, later removed by Doan--Ionel--Walpuski \cite{doanionelwalpuski}).
The core technical fact underlying the work of Ionel--Parker is \emph{generic transversality}: for generic almost complex structures, all simple pseudo-holomorphic maps from compact Riemann surfaces are unobstructed.
\emph{The main technical ingredient in our work is a generic transversality result in complex analytic geometry} (necessarily weaker than its analogue in almost complex geometry), which may be of independent interest (see Section \ref{introtransv}).

According to our main result, an identity between curve enumeration invariants is valid for all complex threefolds with nef anti-canonical bundle iff it is valid for local curves.
For example, we may immediately conclude that the conjecture of Maulik--Nekrasov--Okounkov--Pandharipande \cite{mnopi,mnopii} relating Gromov--Witten and Donaldson--Thomas/Pandharipande--Thomas invariants holds for complex threefolds with nef anti-canonical bundle and primary insertions (Theorem \ref{mainresult}), since it has been known for some time to hold for local curves by the work of Bryan--Pandharipande \cite{bryanpandharipande} and Okounkov--Pandharipande \cite{okounkovpandharipande}.
Of course, this is unlikely to be the only application.
It is natural to ask whether our framework could be applied to the invariants of Maulik--Toda \cite{mauliktoda} or to the $K$-theoretic Gromov--Witten invariants of Givental \cite{giventalquantumktheory} and Y.P. Lee \cite{ypleequantum} (and, specifically, the work of Jockers--Mayr \cite{jockersmayr} and Chou--Lee \cite{choulee}).

The present `use generic transversality to reduce to the case of local curves' strategy has two main drawbacks.
First, it relies fundamentally on the nef anti-canonical bundle hypothesis (via the dimension count in Lemma \ref{semiregulardimension}), while most (all?) (including conjectural) identities between curve enumeration invariants do not require this hypothesis.
Second, it provides no insight as to \emph{why} a given identity between curve enumeration invariants should hold.
Progress on either front would be of exceptional interest; in particular, direct geometric arguments for identities between curve enumeration invariants remain valuable.

\subsection{Universal enumerative invariant}

Let us now sketch the definition of the `universal curve enumeration invariant' and the `Grothendieck group of 1-cycles in complex threefolds' $H^*_c(\cZ/\Cpx_3)$ in which it is valued (a thorough treatment is given in Section \ref{grothendieckdefsec}).
The reader should keep in mind that this discussion is little more than formal nonsense.
The content comes later, in the form of our main result computing the Grothendieck group of semi-Fano 1-cycles $H^*_c(\cZ_\semiFano/\Cpx_3)$.

Given a morphism of (reasonable) topological spaces $\pi:W\to B$, we may consider `cohomology of $B$ with coefficients in fiberwise chains rel infinity on $W\to B$', denoted $H_*^\relinfty(W/B)=H^*(W,(W\to B)^*(B\to *)^!\ZZ)$.
Dually, `homology of $B$ with coefficients in fiberwise compactly supported cochains on $W\to B$' is denoted $H^*_c(W/B)=H^*_c(W,(W\to B)^!(B\to *)^*\ZZ)$.
These sort of `mixed groups' associated to a map $W\to B$, taking homology in one direction and cohomology in the other, are called `bivariant theories', see for example Fulton--MacPherson \cite{fultonmacphersonbivariant} and Section \ref{bivariantthry} below.

Given a (not necessarily proper) family of threefolds $X\to B$, we denote by $\cZ(X/B)\to B$ the space of compact (complex) 1-cycles in the fibers of $X\to B$.
That is, a point of $\cZ(X/B)$ over $b\in B$ is a finite formal non-negative integer linear combination $z=\sum_im_iC_i$ of compact irreducible 1-dimensional subvarieties $C_i\subseteq X_b$.

\begin{definition}[see also Definition \ref{grothgpdef}]\label{grothintrodef}
The \emph{Grothendieck group of 1-cycles in complex threefolds} is the directed colimit
\begin{equation}
H^*_c(\cZ/\Cpx_3):=\dircolim_{X\to B}H^*_c(\cZ(X/B)/B)
\end{equation}
over all (not necessarily proper) families of complex threefolds $X\to B$ over (complexified) finite simplicial complexes $B$.
Dually, a \emph{(naive) curve enumeration theory (for complex threefolds)} is a class in the limit $H_*^\relinfty(\cZ/\Cpx_3)_\naive:=\lim_{X\to B}H_*^\relinfty(\cZ(X/B)/B)$.

The notation $H^*_c(\cZ/\Cpx_3)$ is short for $H^*_c(\cZ(\Univ_3/\Cpx_3)/\Cpx_3)$, as we imagine that it is the bivariant group $H^*_c$ of the relative cycle space $\cZ(\Univ_3/\Cpx_3)\to\Cpx_3$ of the universal family $\Univ_3\to\Cpx_3$ of the moduli stack $\Cpx_3$ of complex threefolds.
\end{definition}

A `curve enumeration theory' in this sense is thus a specification of a `relative virtual fundamental class' in $H_*^\relinfty(\cZ(X/B)/B)$ for every family of complex threefolds $X\to B$ over a finite simplicial complex $B$, compatible with pullback.
Concretely speaking, a cycle $c$ representing a class in $H_*^\relinfty(\cZ(X/B)/B)$ consists of a cycle $c_v$ on $\cZ(X_v)$ for every vertex $v$ in $B$, a chain $c_e$ on $\cZ(X_e/e)$ with boundary $\partial c_e=c_v-c_{v'}$ for every edge $e:v\to v'$ in $B$, etc.
For example, we shall see that the usual Gromov--Witten relative virtual fundamental class is naturally a class in $H_*^\relinfty(\cZ/\Cpx_3;\QQ((u)))$.

An element of the Grothendieck group of 1-cycles may be thought of as a `curve enumeration problem' (or an `insertion') on which any curve enumeration theory may be evaluated to yield a number.
For example, for a smooth complex projective threefold $X$ and any class $\beta\in H_2(X)$, the class of the characteristic function $1_\beta\in H^0_c(\cZ(X))$ in $H^*_c(\cZ/\Cpx_3)$ asks to `count curves in class $\beta$'; more generally, any compactly supported cohomology class in $H^*_c(\cZ(X))$ can be used to constrain the curves we ask to count.
A class $\gamma\in H^*_c(\cZ(X/B)/B)$ asks to count curves in fibers of the family $X\to B$, constrained by $\gamma$.
Note that we consider families $X\to B$ which are not necessarily proper: all the `compactness' necessary to have a well-behaved enumerative problem is encoded by the `compact support' condition in the definition of the bivariant group $H^*_c$.

The choice of `finite simplicial complexes' as the class of base spaces $B$ used to define the Grothendieck group $H^*_c(\cZ/\Cpx_3)$ in Definition \ref{grothintrodef} may appear unnatural.
Do observe, though, that for any family $X\to B$ over (say) a complex analytic space $B$, a choice of real analytic triangulation of $B$ will induce a map $H^*_c(\cZ(X/B)/B)\to H^*_c(\cZ/\Cpx_3)$.
The virtue of simplicial complexes is that they are more amenable to the cut-and-paste operations and local perturbations (bump functions) involved in computing the group $H^*_c(\cZ/\Cpx_3)$ (Theorem \ref{maincalculation}).

The Grothendieck group $H^*_c(\cZ/\Cpx_3)$ is graded by cohomological degree $i$ and by chern number $k$ (the chern number of a 1-cycle in $X$ is its pairing with $c_1(TX)$, which is a locally constant function on the space of 1-cycles $\cZ$).
\begin{equation}\label{grothgrading}
H^*_c(\cZ/\Cpx_3)=\bigoplus_{i,k\in\ZZ}H^i_c(\cZ(-,k)/\Cpx_3)
\end{equation}
Given that the `expected dimension' (index) of moduli spaces of curves (in complex \emph{three}folds) of chern number $k$ is (typically) $2k$ (real dimensions) and a degree $i$ cohomology class defines a codimension $i$ constraint, it is natural to introduce the \emph{virtual dimension} grading on $H^*_c(\cZ/\Cpx_3)$ given by $2k-i$ (which is a `homological grading' and represents the typical virtual dimension of the moduli spaces of curves that the given class in $H^*_c(\cZ/\Cpx_3)$ asks to enumerate).
Dually, there is a bi-grading $H_*^\relinfty(\cZ/\Cpx_3)_\naive=\bigoplus_i\prod_kH_i^\relinfty(\cZ(-,k)/\Cpx_3)_\naive$ by homological degree $i$ and chern number $k$, and it is natural to consider the `virtual codimension' grading $2k-i$.
The standard curve enumeration theories have virtual codimension zero (hence yield numbers when paired against elements of $H^*_c(\cZ/\Cpx_3)$ of virtual dimension zero).

The group $H^*_c(\cZ/\Cpx_3)$ is a bi-algebra: product corresponds to disjoint union of cycles, while coproduct corresponds to sum of cycles.
It also has bi-algebra endomorphisms given by pulling back under the `multiply by $d$' operation on cycles.

One could make all these same definitions for compact $k$-cycles in complex $n$-folds for any values $0\leq k\leq n$.
One of the significant special features of case we consider here, namely $1$-cycles in $3$-folds, is that the standard moduli spaces of curves in threefolds have virtual dimension depending only on the chern number.

The reader may wonder why the colimit in Definition \ref{grothintrodef} yields the `right' result, when it would certainly make more sense to take a homotopy colimit instead.
The answer is that Definition \ref{grothintrodef} actually \emph{is} a homotopy colimit, essentially due to the fact that our class of bases is `homotopical' (this is another reason why simplicial complexes are a nice class of base spaces to work with).

\subsection{Free generation by local curves}

We may now state our main result, which concerns not $H^*_c(\cZ/\Cpx_3)$ but rather $H^*_c(\cZ_\semiFano/\Cpx_3)$, whose definition is identical except that instead of the entire space of 1-cycles $\cZ$, it considers just the open subset of \emph{semi-Fano} 1-cycles $\cZ_\semiFano\subseteq\cZ$, consisting of those 1-cycles $\sum_im_iC_i$ all of whose components $C_i\subseteq X$ pair non-negatively with $c_1(TX)$.

\begin{theorem}[see Theorems \ref{mainsurjectivity} and \ref{maininjectivity}]\label{maincalculation}
The Grothendieck group of semi-Fano 1-cycles in complex threefolds $H^*_c(\cZ_\semiFano/\Cpx_3)$ is supported in virtual dimension $\geq 0$, and in virtual dimension zero it is:
\begin{itemize}
\item freely generated as a $\QQ$-algebra by the equivariant local curve elements $x_{g,m,k}$ with $g\geq 0$, $m\geq 1$, $k\geq 0$, and $(m-1)k=0$, after passing to $\QQ$ coefficients.
\item freely generated as a $\ZZ$-algebra by any choice of geometric local curve elements $y_{g,m,k}$ with $g\geq 0$, $m\geq 1$, $k\geq 0$, and $(m-1)k=0$.
\end{itemize}
\end{theorem}

The equivariant local curve elements $x_{g,m,k}\in H^*_c(\cZ/\Cpx_3;\QQ)$ encode the $S^1$-equivariant enumerative problem of curves of degree $m$ in a local curve of genus $g$ and chern number $k$ (see Definition \ref{loccurveeltsdef} and Proposition \ref{explicitlocalizedpushforward}).
They are defined for all integers $g\geq 0$, $m\geq 0$, and $k\in\ZZ$, and have canonical lifts $x_{g,m,k}\in H^*_c(\cZ_\semiFano/\Cpx_3;\QQ)$ for $k\geq 0$.
The geometric local curve elements $y_{g,m,k}\in H^*_c(\cZ_\semiFano/\Cpx_3)$ are defined (but not uniquely so) for $g\geq 0$, $m\geq 1$, $k\geq 0$, and $(m-1)k=0$, and they encode the enumerative problem of a smooth, isolated, and unobstructed curve of genus $g$ and chern number $k$, taken with multiplicity $m$ (see Definition \ref{geolocalelts}, Lemma \ref{geolocaleltsexist}, and Remark \ref{geolocaleltsnottopological}).

For example, Theorem \ref{maincalculation} says that for any smooth projective complex threefold $X$ with $c_1(TX)=0$ and any homology class $\beta\in H_2(X)$, the element $(X,1_\beta)\in H^*_c(\cZ_\semiFano/\Cpx_3)$ is equal to a unique polynomial in the variables $x_{g,m,0}$ (the variables $x_{g,1,k}$ for $k>0$ do not appear since they have positive chern number while $(X,1_\beta)$ has chern number zero).
It would be interesting to compute these polynomials for various $X$, perhaps by lifting existing computations of curve enumeration invariants.

\begin{remark}
The Grothendieck group of 1-cycles in complex threefolds $H^*_c(\cZ/\Cpx_3)$ is the homology of a naturally defined spectrum (see Remark \ref{shgroth}).
This spectrum thus has $E$-homology groups $E^*_c(\cZ/\Cpx_3)$ for any spectrum $E$.
Recall that a spectrum $E$ is called \emph{connective} when $E(\pt)$ (homology or cohomology, they are the same) is supported in non-negative homological degree (that is $\pi_iE=0$ for $i<0$).
It follows from Theorem \ref{maincalculation} and the Atiyah--Hirzebruch spectral sequence that for any connective spectrum $E$, the group $E^*_c(\cZ_\semiFano/\Cpx_3)$ is supported in non-negative virtual dimension and in virtual dimension zero is freely generated as an $E_0(\pt)$-algebra by any choice of geometric local curve elements (the corresponding statement for equivariant local curve elements is `not interesting' since it requires rationalizing, which destroys any information captured by taking $E$-homology instead of singular homology).
\end{remark}

Theorem \ref{maincalculation} is not the final word on the structure of enumerative invariants of complex threefolds with nef anti-canonical bundle.
Specifically, one could ask for the product expansion of Ionel--Parker \cite{ionelparker} in the complex setting:

\begin{conjecture}\label{exponentialconjecture}
For any complex projective Calabi--Yau threefold $X$, the element $(X,t^{[\cdot]})\in H^0_c(\cZ_\semiFano/\Cpx_3)[[t^{H_2(X)}]]$ is an infinite product $\prod_\beta\prod_{g\geq 0}f_g(t^\beta)^{e_{\beta,g}(X)}$ for unique integer invariants $e_{\beta,g}(X)\in\ZZ$, where $f_g(t)=\sum_{m\geq 0}x_{g,m,0}t^m$.
\end{conjecture}

Such a result is perhaps deducible from the result of \cite{ionelparker} by comparing $H^*_c(\cZ_\semiFano/\Cpx_3)$ and $H^*_c(\cZ_\semiFano/\ACpx_3)$ via Theorem \ref{maincalculation} and an almost complex analogue thereof.
Alternatively, one could ask to define and compute another flavor of Grothendieck group of 1-cycles which keeps all cycles with the same support `together' and in which a decomposition of $X$ into (local curve?) generators naturally encodes the statement of Conjecture \ref{exponentialconjecture}.

While a favorable comparison between $H^*_c(\cZ_\semiFano/\Cpx_3)$ and $H^*_c(\cZ_\semiFano/\ACpx_3)$ seems rather tractable, the following conjecture is rather wild:

\begin{conjecture}
The map $H^*_c(\cZ/\Cpx_3)\to H^*_c(\cZ/\ACpx_3)$ is an isomorphism.
\end{conjecture}

Unfortunately, we see no particular reason why the group $H^*_c(\cZ/\Cpx_3)$ should be understandable.
The most interesting question is probably whether there exists a modification of this group $H^*_c(\cZ/\Cpx_3)$ for which an analogue of Theorem \ref{maincalculation} holds and which can be used to study enumerative invariants of general (not necessarily having nef anti-canonical bundle) complex threefolds.

\subsection{Generic transversality}\label{introtransv}

Now let us have a discussion of the main ideas going into the proof of Theorem \ref{maincalculation} (without which the rest of this Introduction would amount to little more than wishful thinking).

Theorem \ref{maincalculation} is, at its core, a transversality assertion.
We \emph{perturb} the target $X$ containing the 1-cycles we wish to count, and by choosing this perturbation \emph{generically} we reduce our enumerative problem to a finite sum of enumerative problems each supported on a single smooth curve which is unobstructed in some sense.
We then use a filtration (by multiplicity) argument to show that such `geometric local curve' enumerative problems are related (by some `upper triangular' change of basis) to the `equivariant local curve' enumerative problems.
This outline is precisely the strategy taken by Ionel--Parker \cite{ionelparker} to prove the Gopakumar--Vafa integrality conjecture using generic transversality for almost complex structures.
The contribution of this text is to formulate a generic transversality result for complex structures and to show how it may be applied to complex analytic enumerative problems (there is currently no satisfactory theory of sheaf-theoretic enumerative invariants, such as Donaldson--Pandharipande--Thomas invariants, in the almost complex setting).

Generic transversality is a familiar and elementary aspect of almost complex geometry.
Almost complex structures are very flexible (they have a large space of local perturbations), which ensures that for a generic almost complex structure, all (compact) simple pseudo-holomorphic curves are unobstructed.
Complex structures are much more rigid, and a compact complex manifold may fail to have any nontrivial deformations at all!
So, to rescue the statement of generic transversality, we weaken it: we first fix a curve $C\subseteq X$, and then we ask to perturb just a small open neighborhood of $C$ (throwing away the rest of $X$) to make all curves inside this neighborhood unobstructed.
Now a small neighborhood of a curve in any complex manifold has plenty of deformations, from which this \emph{generic transversality for complex structures} claim follows easily.

Now if we try to apply this complex analytic generic tranversality fact to an enumerative problem, it appears (at first glance) much too weak to gain any leverage at all.
Curves in a fixed (say, projective) target $X$ typically move in high-dimensional families throughout all of $X$, even if they lie in a moduli space of `expected dimension' zero (indeed, this is the `problem' that we are trying to use generic transversality to solve!).
A perturbation of just a small neighborhood of a fixed curve $C\subseteq X$ gives us a well-behaved moduli space of `curves near $C$', but to understand the entire enumerative problem, we need to \emph{patch together} these local perturbations somehow.
In a sense, we want to `perturb the target $X$ as a function of the curve $C\subseteq X$ we are counting' (circular though this may sound).

A nice way to make this idea precise is to use what can be called the `graph trick'.
Suppose we are asked to enumerate curves in a family $X\to B$ with an insertion $\gamma\in H^*_K(\cZ(X/B)/B)$ supported on a compact set $K\subseteq\cZ(X/B)$.
By introducing $n$ more degrees of freedom $y\in\RR^n$ (that is, replacing $X\to B$ with its stabilization $X\times\RR^n\to B\times\RR^n$) and imposing the codimension $n$ constraint $y=0$ specifying that these new coordinates must equal zero, we get an equivalent enumerative problem, now supported on $K\times 0\subseteq\cZ(X/B)\times\RR^n$.
So far we have done nothing at all, but now we observe that we could instead impose the enumeratively equivalent constraint $y=f(z)$ for some function $f:\cZ(X/B)\to\RR^n$, and then our new insertion is supported on the graph of $f|_K$ inside $\cZ(X/B)\times\RR^n$, which maps injectively to the base $B\times\RR^n$ provided we choose $f|_K$ to be an embedding.
We conclude: \emph{every curve enumeration problem is equivalent to one whose insertion $\gamma\in H^*_K(\cZ(X/B)/B)$ is supported on a compact set $K\subseteq\cZ(X/B)$ mapping injectively to the base $B$}.
Now for such an insertion $\gamma$, we can (at least plausibly) get some mileage out of (complex analytic) generic transversality, since the constraint $\gamma$ localizes the counting problem around at most a single 1-cycle in every fiber of $X\to B$ (and thus we can trim away all but a small neighborhood of these 1-cycles from $X$, leaving us with a family of neighborhoods of 1-cycles, to which we can now plausbily apply complex analytic generic transversality).
Note how the weaker nature of complex analytic generic transversality compared to almost complex generic transversality has forced us into the setting of \emph{family} invariants (i.e.\ over a nontrivial base space $B$), while in the setting of almost complex generic transversality, Ionel--Parker \cite{ionelparker} only need to consider $0$- and $1$-parameter families of almost complex manifolds.

To apply complex analytic generic transversality to a curve enumeration problem $\gamma\in H^*_K(\cZ(X/B)/B)$ for compact $K\subseteq\cZ(X/B)$ injecting to $B$, there is yet one more issue to resolve.
The statement of complex analytic generic transversality involves deformations of a neighborhood of a curve $C\subseteq X$ by a certain `regluing' operation supported near a (germ of) smooth divisor $D\subseteq X$ intersecting $C$ in a finite set which meets every irreducible component of $C$ (we say such a divisor \emph{controls} $C$).
While the existence of a controlling divisor for a single curve $C\subseteq X$ is trivial, the same cannot be said for a family of curves $K\subseteq\cZ(X/B)$ in a family $X\to B$ (even given injectivity of $K\to B$).
A controlling (relative) divisor certainly exists locally on $B$, but \emph{a priori} the divisors chosen for different open neighborhoods in $B$ could intersect (which prevents us from using them to deform $X$ simultaneously, as we need to do).
A nontrivial inductive argument (induction on the dimension of $K$) is required to see that such divisors together controlling all of $K$ can indeed be chosen to be disjoint (we call this `enough divisors' Proposition \ref{divisorcovering}).
Note that since our divisors only exist locally on $B$, we need a large supply of real analytic bump functions on $B$ to write down the requisite divisorial deformations, and to do this we replace $B$ with a very fine real analytic triangulation thereof.

We have now finally shown that an arbitrary curve enumeration problem is equivalent to one supported inside the open locus of \emph{interior semi-regular 1-cycles} (that is, the set of 1-cycles $z=\sum_im_iC_i$ which are semi-regular, meaning $\bigsqcup_i\tilde C_i\to X$ is unobstructed in a strong sense, see Definition \ref{regulardef}, and for which all nearby 1-cycles $z'$ are also semi-regular).
In fact, this argument shows that map from the Grothendieck group of interior semi-regular 1-cycles to the Grothendieck group of all 1-cycles is an isomorphism (see Theorem \ref{grothendiecktransverse}).

Now we may use this result to study the Grothendieck group of semi-Fano 1-cycles.
The locus of interior semi-regular semi-Fano 1-cycles has dimension bounded by its virtual dimension (see Lemma \ref{semiregulardimension}), from which it follows that the Grothendieck group of such 1-cycles is supported in virtual dimension $\geq 0$ and is generated in virtual dimension zero by Poincaré duals of smooth points, which are the so-called `geometric local curve elements'.
While the geometric local curve elements are in general rather complicated (see Remark \ref{geolocaleltsnottopological}), they become simple (and agree with the equivariant local curve elements) if we pass to the associated graded of the filtration of the cycle space $\cZ$ by multiplicity (Definition \ref{multfilt}, Lemma \ref{eqlocalcurveassocgraded}, and Proposition \ref{reducedtoptype}).
This proves the generation part of Theorem \ref{maincalculation}.

We expect that one can derive the full strength of Theorem \ref{maincalculation} (\emph{free} generation) by continuing this geometric reasoning (see Remark \ref{continue}), but instead we opt for an algebraic argument to deduce free generation (see Section \ref{bialgebraconstraints}).

\subsection{MNOP correspondence}

Now let us explain how Theorem \ref{maincalculation} may be used to verify the primary insertions case of the conjecture of Maulik--Okounkov--Nekrasov--Pandharipande \cite{mnopi,mnopii} for complex threefolds with nef anti-canonical bundle (regarding descendents, see Question \ref{descendentquestion}).

The original `MNOP' conjecture of Maulik--Nekrasov--Okounkov--Pandharipande \cite{mnopi,mnopii} relates Gromov--Witten and Donaldson--Thomas invariants of projective threefolds.
Pandharipande--Thomas \cite{pandharipandethomasstablepairs} introduced a variant of Donaldson--Thomas invariants which is easier to work with and conjectured a relation between the two, which was later proven by Bridgeland \cite{bridgelandhall}.
Given this equivalence, the MNOP conjecture may be stated as a relation between Gromov--Witten and Pandharipande--Thomas invariants, and it is this form of the conjecture that we will consider here.

Gromov--Witten invariants and Donaldson--Pandharipande--Thomas invariants describe curves in fundamentally different ways: Gromov--Witten moduli spaces describe curves in $X$ as images of maps from nodal domain curves to $X$, whereas Donaldson--Pandharipande--Thomas moduli spaces describe curves in $X$ as zero sets of functions on $X$.
These moduli spaces agree on the set of smooth embedded curves $C\subseteq X$, but are rather different compactifications of this space.
The MNOP conjecture is interesting because there is no known or even proposed geometric relation between these two different ways of compactifying.

Let us briefly recall the definition of Gromov--Witten and Pandharipande--Thomas invariants, leaving a more detailed discussion to Section \ref{zgwptsummary}.
Given a complex projective threefold $X$, a homology class $\beta\in H_2(X)$, and cohomology classes $\gamma_1,\ldots,\gamma_r\in H^*(X)$, these invariants have the form
\begin{align}
\GW(X,\beta;\gamma_1,\ldots,\gamma_r)&=\int_{[\cMbarprime(X,\beta)]^\vir}{\textstyle\prod\limits_{i=1}^r}\pi_!\ev^*\gamma_i\cdot u^{-\chi}\in\QQ((u)),\\
\PT(X,\beta;\gamma_1,\ldots,\gamma_r)&=\int_{[P(X,\beta)]^\vir}{\textstyle\prod\limits_{i=1}^r}\pi_!(\ch_2(\FF)\cup\ev^*\gamma_i)\cdot q^n\in\ZZ((q)).
\end{align}
For Gromov--Witten invariants, $\cMbarprime(X,\beta)$ is the moduli space of stable maps from (not necessarily connected) nodal curves $C$ to $X$, in homology class $\beta$, all of whose connected components are non-constant, and $\chi$ denotes the arithmetic Euler characteristic $2\cdot\chi(C,\cO_C)$ of the domain (locally constant, proper superlevel sets).
For Pandharipande--Thomas invariants, $P(X,\beta)$ denotes the moduli space of stable pairs in homology class $\beta$, and $n$ denotes the holomorphic Euler characteristic (locally constant, proper sublevel sets).
The integrands are given by push/pull via the universal families over these moduli spaces.
More generally, we could consider the family invariants
\begin{align}
\GW(X/B;\gamma)&=\int_{[\cMbarprime(X/B)]^\vir}\gamma\cdot u^{-\chi}\in\QQ((u)),\\
\PT(X/B;\gamma)&=\int_{[P(X/B)]^\vir}\gamma\cdot q^n\in\ZZ((q)),
\end{align}
for any relative dimension three smooth morphism of varieties $X\to B$ and any compactly supported cohomology class $\gamma\in H^*_c(\cZ(X/B)/B)$.
These reduce to the invariants of $(X,\beta;\gamma_1,\ldots,\gamma_r)$ defined above in the case $B=*$ and $\gamma=1_\beta\prod_i\pi_!\ev^*\gamma_i\in H^*_c(\cZ(X))$ (see Section \ref{zgwptsummary} for more details).

Now let us say that a pair of formal Laurent series $\GW\in\QQ((u))$ and $\PT\in\ZZ((q))$ satisfies the \emph{MNOP correspondence} with weight $k$ when $\PT$ is a rational function of $q$ and the evaluation of $(-q)^{-k/2}\PT$ at $-q=e^{iu}$ equals $(-iu)^k\GW$.

\begin{conjecture}[\cite{mnopi,mnopii,pandharipandethomasstablepairs}]\label{mainconjecture}
For any projective threefold $X$, any homology class $\beta\in H_2(X)$, and any tuple of cohomology classes $\gamma_1,\ldots,\gamma_r\in H^*(X)$, the invariants
\begin{equation}
\GW(X,\beta;\gamma_1,\ldots,\gamma_r)\text{ and }\PT(X,\beta;\gamma_1,\ldots,\gamma_r)
\end{equation}
satisfy the MNOP correspondence with weight $k=\langle c_1(TX),\beta\rangle$.
More generally, so do $\GW(X/B;\gamma)$ and $\PT(X/B;\gamma)$ for any relative dimension three smooth morphism of varieties $X\to B$ and any $\gamma\in H^*_c(\cZ(X/B)/B)$ of chern number $k$.
\end{conjecture}

Conjecture \ref{mainconjecture} is known in many cases, essentially by computing both sides of the equality.
The case of (equivariant invariants of) local curves holds by deep calculations of Bryan--Pandharipande \cite{bryanpandharipande} (of Gromov--Witten invariants) and Okounkov--Pandharipande \cite{okounkovpandharipande} (of Donaldson--Thomas invariants and, by \cite[Section 5]{mptmodular}, Pandharipande--Thomas invariants).
Work of Maulik--Oblomkov--Okounkov--Pandharipande \cite{mooptoric} established the conjecture for toric varieties by direct computation of both sides.
Work of Pandharipande--Pixton \cite{ppmany} showed the result for many threefolds (e.g.\ complete intersections in products of projective spaces) by degeneration to the toric case.

To bring Theorem \ref{maincalculation} to bear on Conjecture \ref{mainconjecture}, we note that Gromov--Witten invariants and Pandharipande--Thomas invariants define ring homomorphisms out of the Grothendieck group
\begin{align}
\label{grothgw}\GW:H^*_c(\cZ/\Cpx_3)&\to\QQ((u))\\
\label{grothpt}\PT:H^*_c(\cZ/\Cpx_3)&\to\ZZ((q))
\end{align}
(see Section \ref{zgwptsummary}).
Conjecture \ref{mainconjecture} amounts to the assertion that these homomorphisms satisfy the MNOP correspondence with weight $k$ when evaluated on certain elements of the Grothendieck group of 1-cycles $H^*_c(\cZ/\Cpx_3)$ of chern number $k$.

\begin{theorem}\label{mainresult}
For any element $\gamma\in H^*_c(\cZ_\semiFano/\Cpx_3)$ of chern number $k$, the power series $\GW(\gamma)$ and $\PT(\gamma)$ satisfy the MNOP correspondence with weight $k$.
In particular, Conjecture \ref{mainconjecture} holds if $X$ (or, more generally, every fiber of $X\to B$) is semi-Fano (meaning the evaluation of $c_1$ on every curve is $\geq 0$).
\end{theorem}

\begin{proof}
The results of \cite{bryanpandharipande,okounkovpandharipande} imply that $\GW$ and $\PT$ satisfy the MNOP correspondence with weight $k$ when evaluated on any equivariant local curve element $x_{g,m,k}$ (see Corollary \ref{mnoplocalgrothendieck}).
By Theorem \ref{maincalculation} (or just the surjectivity part, Theorem \ref{mainsurjectivity}), this implies they satisfy the MNOP correspondence (weighted by chern number) on all of $H^*_c(\cZ_\semiFano/\Cpx_3)$ (note that, for degree reasons, $\GW$ and $\PT$ are only nonzero on the virtual dimension zero part of this group).
\end{proof}

This approach to Conjecture \ref{mainconjecture} is similar in spirit to \cite{ppmany} in that in essence we deform to a simpler situation (specifically, a disjoint union of local curves) where the result is already known (and the strength of Theorem \ref{maincalculation} allows us to obtain a stronger result from a weaker input).
Note that our result applies to family enumeration problems, while these are mostly inaccessible by other degeneration/deformation arguments.

\begin{remark}
The definition of the Gromov--Witten and Pandharipande--Thomas invariants on the Grothendieck group $H^*_c(\cZ/\Cpx_3)$ (that is, the ring homomorphisms \eqref{grothgw}--\eqref{grothpt}) as used in the proof of Theorem \ref{mainresult} requires a definition of the moduli spaces of stable maps $\cMbarprime(X/B)$ and stable pairs $P(X/B)$ for arbitrary separated submersions of reduced complex analytic spaces $X\to B$ of relative dimension three (see Sections \ref{vfcsubsec}--\ref{zgwptsummary}).
While the standard constructions of these moduli spaces assume that $X\to B$ is a smooth separated morphism of schemes over $\CC$, we believe that there is no essential difficulty in translating everything to the complex analytic setting by substituting the theory of Douady spaces \cite{douady} in place of the theory of Hilbert schemes.
Unfortunately, we do not know a reference for this.
\end{remark}

The case of descendent insertions is conspicuously missing from this discussion.
We do not know whether Theorem \ref{mainresult} says anything about descendent insertions:

\begin{question}\label{descendentquestion}
How many descendent insertions arise via pullback under the maps $\cMbarprime(X)\to\cZ(X)$ and $P(X)\to\cZ(X)$ (as the primary insertions do)?
Does $\cU(X/B)$ carry any non-trivial natural (i.e.\ compatible with pullback) cohomology classes $\xi$, and, if so, how are the resulting insertions $\pi_!(\xi\cup\ev^*\gamma)$ related to descendents?
Is there an enhancement of $\cZ(X)$ which receives maps from $\cMbarprime(X)$ and $P(X)$ and for which the answer to these questions is better?
\end{question}

\subsection{Acknowledgements}

I gratefully acknowledge conversations related to the subject of this paper with Shaoyun Bai, Jim Bryan, Mike Miller Eismeier, Rahul Pandharipande, Dhruv Ranganathan, Will Sawin, Mohan Swaminathan, Felix Thimm, Richard Thomas, and Tony Yue Yu.
The comments from the two anonymous referees have greatly improved the exposition.
The author received support from the Packard Foundation (Fellowship for Science and Engineering) and the National Science Foundation (Alan T.\ Waterman Award, Grant No.\ 1747553).

\section{Background}

\subsection{Spaces of 1-cycles}\label{spaceofcycles}

A (compact holomorphic) \emph{1-cycle} $z$ on a complex analytic manifold $X$ is a formal non-negative integer linear combination of irreducible compact 1-dimensional subvarieties $C\subseteq X$.
Such a cycle will usually be written as a finite sum $z=\sum_im_iC_i$ where it is implicitly assumed that the $C_i\subseteq X$ are distinct irreducible curves and all $m_i>0$.
The set of such 1-cycles is denoted $\cZ(X)$ (more systematically, this would be denoted $\cZ_1(X)$, but we will not consider $r$-cycles $\cZ_r(X)$ for any $r$ other than $1$ in this text, so we drop the subscript from the notation).

The \emph{chern number} of a cycle refers to its pairing with $c_1(TX)$, so the chern number of $z=\sum_im_iC_i$ is $\langle c_1(TX),z\rangle=\sum_im_i\langle c_1(TX),C_i\rangle$.
A cycle $z=\sum_im_iC_i$ is called \emph{semi-Fano} when all $C_i$ have non-negative chern number (it bears emphasis that this is strictly stronger than $z$ itself having non-negative chern number).
We denote by $\cZ(X,k)\subseteq\cZ(X)$ the set of cycles with chern number $k$, and we denote by $\cZ(X)_\semiFano\subseteq\cZ(X)$ the set of semi-Fano cycles.
We also denote by $\cZ(X,\beta)\subseteq\cZ(X)$ the set of cycles in homology class $\beta\in H_2(X)$ (which, we should warn, somewhat conflicts with the notation $\cZ(X,k)$ for cycles of chern number $k$).

The set $\cZ(X)$ can be given the structure of a separated reduced complex analytic space due to work of Barlet \cite{barlet} (by `complex analytic space' we mean locally an analytic subspace of an open set in $\CC^N$ for some $N<\infty$); this is a complex analytic analogue of the theory of Chow varieties.
By definition, an analytic map $A\to\cZ(X)$ from a reduced complex analytic space $A$ is a family of 1-cycles $\{z_a\in\cZ(X)\}_{a\in A}$ which satisfies the following analyticity condition \cite[Chapitre 1, \S1, Définition fondamentale]{barlet}.
First, we ask that $a\mapsto z_a$ be continuous in the Hausdorff topology on compact subsets of $X$ (concretely, this means that every $a\in A$ has a small neighborhood over which the cycles $z_{a'}$ are all contained in any given neighborhood of $z_a$ and intersect any given neighborhood of any given point of $z_a$).
Second, we ask that for any chart $\DD\tighttimes\DD^2\subseteq X$ (where $\DD\subseteq\CC$ denotes the closed unit disk), the map
\begin{align}
A\times\DD^\circ&\to\bigsqcup_k\Sym^k\CC\\
(a,q)&\mapsto f(z_a\cap(\{q\}\tighttimes\DD^2))
\end{align}
(intersection taken with multiplicity) where $f:\DD\times\DD^2\to\CC$ is any analytic function, should be analytic over the open set of $(a,p)$ for which $z_a\cap(\{q\}\tighttimes\partial(\DD^2))=\varnothing$, where analytic with target $\Sym^k\CC$ just means its composition with any symmetric polynomial $\CC^k\to\CC$ is analytic.
For any analytic family of cycles $\{z_a\in\cZ(X)\}_{a\in A}$, the `total space' $\bigcup_a\{a\}\tighttimes z_a\subseteq A\tighttimes Z$ is a closed analytic subvariety (and, conversely, when $A$ is normal, specifying an analytic family of cycles $\{z_a\in\cZ(X)\}_{a\in A}$ is equivalent specifying its total space inside $A\times Z$ along with multiplicities for every irreducible component thereof, see \cite[Chapitre 1, \S2, Théorème 1]{barlet} for a precise statement and proof).
In particular, there is a `universal family' $\cU(X)\subseteq X\times\cZ(X)$.

The homology class function $\cZ(X)\to H_2(X)$ is locally constant; that is, the locus $\cZ(X,\beta)\subseteq\cZ(X)$ of cycles in homology class $\beta\in H_2(X)$ is open (hence also closed).
In partcular, the subset $\cZ(X,k)\subseteq\cZ(X)$ of cycles with chern number $k$ is open (hence also closed).
The subset $\cZ(X)_\semiFano\subseteq\cZ(X)$ is also open.

This discussion generalizes readily to the relative setting.
Given a holomorphic submersion $X\to B$, we define $\cZ(X/B)=\bigcup_b\cZ(X_b)$ to be the set of cycles in fibers of $X\to B$.
It is an open subset of $\cZ(X)$, so the basic properties of $\cZ(X)$ pass easily to $\cZ(X/B)$.

\begin{definition}[Multiplicity stratification]\label{multfilt}
Let $\bM=\bigsqcup_{n\geq 0}\ZZ_{\geq 1}^n/S_n$ be the set of finite multi-sets of positive integers.
Partially order $\bM$ by declaring that $\bm\geq\bm'$ whenever we may express every $m_i\in\bm$ as a non-negative integer linear combination $m_i=\sum_jn_{ij}m_j'$ of the $m_j'\in\bm'$ in such a way that every $m_j'$ is used at least once (in other words, $\sum_in_{ij}>0$ for every $j$).
It is evident that if $\bm\geq\bm'$ then $\sum_im_i\geq\sum_jm_j'$; moreover, if $\bm\geq\bm'$ and $\sum_im_i=\sum_jm_j'$, then $\bm'$ is obtained from $\bm$ by replacing each $m_i$ with a partition thereof.
It is now not hard to show that the relation $\geq$ is anti-symmetric, hence is indeed a partial order as claimed.

There is a map $\cZ\to\bM$ assigning to each cycle $z=\sum_im_iC_i$ the multi-set $\bm$ of multiplicities $m_i$.
It is not hard to check that this map is upper semi-continuous (that is, the loci $\cZ_{\leq\bm}\subseteq\cZ$ are open): if $z=\sum_im_iC_i$ is a 1-cycle and $z'=\sum_jm_j'C_j'$ is sufficiently close to $z$, then looking at $z'$ near a smooth point of $C_i$ yields a decomposition $m_i=\sum_jn_{ij}m_j'$, and every $C_j'$ is close to some $C_i$ hence has some positive $n_{ij}$.
Morever, if $z'$ is sufficiently close to $z=\sum_im_iC_i$ and has the same multiplicity, then $z'=\sum_im_iC_i'$ for some curves $C_i'$ close to $C_i$.
\end{definition}

\subsection{Bivariant theories}\label{bivariantthry}

A \emph{bivariant theory} is an invariant of maps $W\to B$.
Bivariant theories come in two dual flavors: those which `take cohomology of the base with coefficients in fiberwise chains', and those which `take homology of the base with coefficients in fiberwise cochains'.
A reference for bivariant theories is Fulton--MacPherson \cite{fultonmacphersonbivariant}.
We take the following as our definition:

\begin{definition}\label{bivariantmaindef}
We associate to a map $\pi:W\to B$ the groups
\begin{align}
H^*_c(W/B)&:=H^*_c(W,\pi^*\omega_B)=a_!\pi_!\pi^*a^!\ZZ\\
H_{-*}^\relinfty(W/B)&:=H^*(W,\pi^!\underline\ZZ_B)=a_*\pi_*\pi^!a^*\ZZ
\end{align}
where $a:B\to *$ (the sheaf functors $(f^*,f_*)$ and $(f_!,f^!)$ are all `derived'.)
Of course, one could consider any other coefficient group in place of $\ZZ$.

We also, define $H^*_K(W/B)=H^*_K(W,\pi^*\omega_B)$ (cohomology with supports in $K$) for any compact $K\subseteq W$.
\end{definition}

This definition makes sense in any context where we have the derived functors $(f^*,f_*)$ and $(f_!,f^!)$ on sheaves.
We will consider here the case of sheaves of complexes of abelian groups on nice (say, locally triangulable) topological spaces (except in Appendix \ref{vfc} where we consider such sheaves on complex analytic stacks).

Intuitively speaking, we think of $H_*^\relinfty(W/B)$ as `cohomology of $B$ with coefficients in fiberwise chains rel infinity on $W\to B$', and we think of $H^*_c(W/B)$ as `homology of $B$ with coefficients in fiberwise compactly supported cochains on $W\to B$'.
Note that the validity of these interpretations (specifically, the interpretation of upper shriek as homology) relies upon the map $W\to B$ and the space $B$ (respectively) being sufficiently nice (e.g.\ finite-dimensional and separated).

We warn the reader that these groups are, in general, supported in arbitrarily positive and arbitrarily negative degrees (for the simple reason that they combine homology and cohomology, which are unbounded in opposite directions).

When either $B$ is a manifold or $W\to B$ is a `manifold bundle' (i.e.\ locally a product with $\RR^n$ for some $n$), then the bivariant groups $H_*^\relinfty(W/B)$ and $H^*_c(W/B)$ may be identified with ordinary homology and cohomology, via a Poincaré duality type isomorphism.
Precisely, if $B$ is an oriented manifold of dimension $n$, then we have canonical isomorphisms $H^*_c(W/B)=H^{*+n}_c(W)$ and $H_*^\relinfty(W/B)=H_{*+n}^\relinfty(W)$.
If $W\to B$ is a relatively oriented manifold bundle of relative dimension $n$, then we have canonical isomorphisms $H^*_c(W/B)=H_{n-*}(W)$ and $H_*^\relinfty(W/B)=H^{n-*}(W)$.

The bivariant groups $H^*_c$ and $H_*^\relinfty$ have the following functoriality (references include \cite{khanvirtual,portayu}):

\begin{definition}\label{bivariantfunctoriality}
Fix $W\to B$ and $W'\to B'$.
Given a map $\alpha:B\to B'$ and a proper map $q:W\times_BB'\to W'$ over $B'$, there are induced maps $\alpha_*q^*:H^*_c(W'/B')\to H^*_c(W/B)$ and $q_*\alpha^*:H_*^\relinfty(W/B)\to H_*^\relinfty(W'/B')$.
\begin{equation}
\begin{tikzcd}
W\times_BB'\ar[dr,"\pi_{B'}"']\ar[r,"q"',"\mathrm{proper}"]\ar[rr,bend left,"\beta"]&W'\ar[d,"\pi'"]&W\ar[d,"\pi"]\\
{}&B'\ar[r,"\alpha"]&B
\end{tikzcd}
\quad\leadsto\quad
\newlength\Hclength
\settowidth\Hclength{$H^*_c$}
\newlength\Hrelinftylength
\settowidth\Hrelinftylength{$H_*^\relinfty$}
\begin{tikzcd}[cramped,row sep=0ex]
\hspace\Hrelinftylength\hspace{-\Hclength}H^*_c(W'/B')\ar[r,"\alpha_*q^*"]&\hspace\Hrelinftylength\hspace{-\Hclength}H^*_c(W/B)\\
H_*^\relinfty(W'/B')&\ar[l,"q_*\alpha^*"]H_*^\relinfty(W/B)
\end{tikzcd}
\end{equation}
The map $H^*_c(W'/B')\to H^*_c(W/B)$ is defined as follows.
It suffices to define a map $\alpha_!\pi'_!(\pi')^*\alpha^!\to\pi_!\pi^*$.
We have the unit transformation $1\to q_*q^*$, and $q_*=q_!$ since $q$ is proper, which gives a map $\pi'_!(\pi')^*\to\pi'_!q_!q^*(\pi')^*=(\pi_{B'})_!\pi_{B'}^*$, so it is enough to define a map $\alpha_!(\pi_{B'})_!\pi_{B'}^*\alpha^!\to\pi_!\pi^*$.
For this, it is enough to define a map $\beta_!\pi_{B'}^*\alpha^!\to\pi^*$, which by the adjunction $(\alpha_!,\alpha^!)$ is equivalent to specifying a map $\beta_!(\pi_{B'})^*\to\pi^*\alpha_!$.
Such a natural transformation is given by the inverse to the base change morphism $\pi^*\alpha_!\xrightarrow\sim\beta_!\pi_{B'}^*$, which is an isomorphism by the Proper Base Change Theorem.
The map $H_*^\relinfty(W/B)\to H_*^\relinfty(W'/B')$ is defined analogously: we want a map $\pi_*\pi^!\to\alpha_*\pi'_*(\pi')^!\alpha^*$, which comes from combining the counit $q_!q^!\to 1$, the isomorphism $q_*=q_!$ from properness of $q$, and the natural transformation $(\pi_{B'})_!\beta^*\to\alpha^*\pi_!$ inverse to the proper base change morphism.

This recipe is compatible with composition for triples $W\to B$, $W'\to B'$, $W''\to B''$, since the proper base change morphism is compatible with composition of pullback squares.
\end{definition}

There are canonical (commutative and associative) Künneth morphisms
\begin{align}
H^*_c(W/B)\otimes H^*_c(W'/B')&\to H^*_c(W\times W'/B\times B')\\
H_*^\relinfty(W/B)\otimes H_*^\relinfty(W'/B')&\to H_*^\relinfty(W\times W'/B\times B')
\end{align}
which are isomorphisms modulo torsion (better, are isomorphisms if we use the derived tensor product).
These may be defined using the symmetric monoidal properties of the sheaf functors $(f^*,f_*)$ and $(f_!,f^!)$ and proper base change.

\subsection{Simplicial complexes}

Recall that a \emph{simplicial complex} $T$ consists of a set $V(T)$ of `vertices' and a set $\Sigma(T)\subseteq 2^{V(T)}_\fin\setminus\{\varnothing\}$ of `simplices' satisfying $A\subseteq B\in\Sigma(T)\implies A\in\Sigma(T)$.
The \emph{geometric realization} of a simplicial complex $T$ is the subspace $T_\RR\subseteq\RR^{V(T)}$ consisting of those tuples $x:V(T)\to\RR_{\geq 0}$ satisfying $\sum_vx(v)=1$ and for which $\supp x\subseteq V(T)$ lies in $\Sigma(T)$.

A \emph{locally ordered} simplicial complex is a simplicial complex $T$ together with, for every simplex $\sigma\in\Sigma(T)$, a total ordering of its vertices, which is compatible with inclusions of simplices.
The product of locally ordered simplicial complexes $T$ and $T'$ is another locally ordered simplicial complex $T\times T'$ with vertex set $V(T\times T')=V(T)\times V(T')$, and a finite subset $S\subseteq V(T)\times V(T')$ is a simplex of $T\times T'$ when its projections to $V(T)$ and $V(T')$ are simplices of $T$ and $T'$, and $S$ is a totally ordered subset of their product (it is then equipped with the restriction of the ordering on said product).
The natural map $(T\times T')_\RR\to T_\RR\times T'_\RR$ is a homeomorphism (at least) when $T$ and $T'$ are both finite.

For any simplicial complex $T$, its \emph{barycentric subdivision} $bT$ has vertex set $V(bT)=\Sigma(T)$ (the set of simplices of $T$), and a finite subset of $V(bT)$ spans a simplex of $bT$ iff it is totally ordered (nested) as a subset of $\Sigma(T)$ (this ordering means $bT$ is in fact a locally ordered simplicial complex).
There is a natural homeomorphism $(bT)_\RR=T_\RR$ (at least) for finite $T$.
Iterative barycentric subdivisions become \emph{arbitrarily fine} when $T$ is finite: for any open cover of $T_\RR$, there is some $n<\infty$ for which every simplex of the $n$th iterated barycentric subdivision $b^nT$ is contained in some element of the open cover.

There is a natural `bordism' or `concordance' between $T$ and $bT$.
It is the simplicial complex with vertex set $V(T)\sqcup V(bT)=V(T)\sqcup\Sigma(T)$ and in which a simplex is a pair of simplices $A\in\Sigma(T)$ and $B\in\Sigma(bT)$ for which every element of $B$ (a vertex of $bT$, that is a simplex of $T$) contains $A$.
The geometric realization of this bordism $bT$ is naturally identified with $[0,1]\times\abs T$ (at least for $T$ finite).
A local ordering of $T$ induces one on this bordism, by taking the orderings on $T$ and $bT$ and declaring that all vertices of $T$ lie below vertices of $bT$.

A \emph{map} of simplicial complexes $T\to T'$ is a map on vertex sets $V(T)\hookrightarrow V(T')$ which sends simplices to simplices.

\subsection{Complex geometry over simplicial complexes}\label{complexoversimplex}

We import simplicial complexes into complex analytic geometry as follows.
For a finite simplicial complex $T$, we define $T_\CC\subseteq\CC^{V(T)}$ to consist of those tuples $x:V(T)\to\CC$ satisfying $\sum_vx(v)=1$ and for which $\supp x\subseteq V(T)$ lies in $\Sigma(T)$.
We regard $T_\CC\subseteq\CC^{V(T)}$ a complex analytic subspace (defined by the ideal of all analytic functions vanishing on it) and refer to it as the \emph{complexified geometric realization} (or just \emph{complexification}) of $T$.
There is an evident inclusion $T_\RR\subseteq T_\CC$, and for many later purposes it is the germ of $T_\CC$ around $T_\RR$ that is most relevant.

When working with complexified simplicial complexes $T_\CC$, it will be convenient to argue over each simplex individually (for example, inductively).
Here is the foundational fact which allows us to do this:

\begin{lemma}\label{functionextension}
A (locally defined) analytic function on the union of hyperplanes $\{x_1\cdots x_n=0\}\times\CC^m\subseteq\CC^n\times\CC^m$ (analytic in the sense that its restriction to each hyperplane $\{x_i=0\}\times\CC^m\subseteq\CC^n\times\CC^m$ is analytic) is (locally) the restriction of a (locally defined) analytic function on $\CC^n\times\CC^m$.
\end{lemma}

\begin{proof}
Given $f$ defined on the union of hyperplanes, an extension is given by the usual inclusion-exclusion formula $\bar f(x_1,\ldots,x_n,y)=\sum_{\varnothing\ne S\subseteq\{1,\ldots,n\}}(-1)^{\abs S-1}f(\{x_i\}_{i\notin S},\{0\}_{i\in S},y)$ for $x_i\in\CC$ and $y\in\CC^m$.
\end{proof}

\begin{corollary}\label{gluinghyperplanes}
The closed analytic subspace $\{x_1\cdots x_n=0\}\times\CC^m\subseteq\CC^n\times\CC^m$ is the colimit of the coordinate subspaces $\CC^S\times 0^{\{1,\ldots,n\}\setminus S}\times\CC^m$ over proper subsets $S\subsetneqq\{1,\ldots,n\}$.
\end{corollary}

\begin{proof}
The universal property of being a colimit is a question about morphisms to any complex analytic space.
Since every complex analytic space is a (finite) limit of open subsets of $\CC$, limits commute with limits, and morphisms of complex analytic spaces are of local nature (are sheaves), it suffices to check the colimit condition on morphisms to $\CC$.
That is, we should check that the sheaf of analytic functions on $\{x_1\cdots x_n=0\}\times\CC^m$ is the limit of the sheaves of compatible (that is, agreeing via restriction) families of analytic functions on the coordinate subspaces $\CC^S\times\CC^{\{1,\ldots,n\}\setminus S}\times\CC^m$.
Injectivity is immediate, and surjectivity is Lemma \ref{functionextension}.
\end{proof}

\begin{corollary}\label{simplexcolimits}
$T_\CC$ is the colimit of its simplices $\sigma_\CC$ (over the diagram indexed by $\sigma\in\Sigma(T)$).
\end{corollary}

\begin{proof}
The assertion is local on the underlying topological space of $T_\CC$, where it reduces to Corollary \ref{gluinghyperplanes}.
\end{proof}

\begin{definition}[Family of complex manifolds]\label{familysimplicialfinite}
Recall that a \emph{submersion} (or \emph{smooth morphism}) of complex analytic spaces is a morphism $X\to B$ which is source-locally modelled on (open subsets of) $A\times\CC^N\to A$ for complex analytic spaces $A$.
A \emph{family (of complex manifolds)} parameterized by (or `over') $B$ is a separated submersion $X\to B$ (this term is somewhat dubious since we have not imposed any properness condition).
We shall also want to consider families of complex manifolds parmeterized by simplicial complexes $B$, by which we mean a family over the complexification $B_\CC$.
\end{definition}

Note that a family of complex manifolds over an open subset $B^-\subseteq B$ is also a family over $B$.
Given a family $X\to B$ over a complex analytic space $B$, we may wish to pull it back to a finite simplicial complex $T$ along a simplex-wise real analytic map to $B$.
Note that, according to Corollary \ref{simplexcolimits}, a simplex-wise real analytic map $T\to B$ is the same thing as a complex analytic map $T_\CC\to B$ defined in a neighborhood of $T_\RR$.
The pullback $X\times_BT\to T$ now depends on a choice of neighborhood of $T_\RR\subseteq T_\CC$ over which the map to $B$ is defined and analytic, but when we appeal to this construction, it will be clear that this choice does not matter.

Now let us argue that a family over a simplicial complex is the same as a compatible collection of families over each of its simplices:

\begin{lemma}\label{familysimplicialtotalspace}
Let $T$ be a finite simplicial complex.
The total space of a family of complex manifolds over $T$ is the colimit of its restrictions to each simplex.
The category of families of complex manifolds over $T$ is equivalent (via the canonical functor) to the limit of the categories of families of complex manifolds over each of its simplices.
\end{lemma}

\begin{proof}
Let $X\to T$ be a family.
That its total space is the colimit of its restrictions to simplices follows from Corollary \ref{gluinghyperplanes} (applied in a local trivialization of $X\to T$).

To check that the functor is fully faithful, we should compare morphisms over $T_\CC$ with compatible systems of morphisms over its simplices.
That these give the same result follows from applying the universal property of colimits to the previous assertion (the total space is the colimit of its restrictions to simplices).

Now let us check that the functor is essentially surjective.
This means that we should glue together a system of families $X_\sigma\to\sigma$ over each simplex $\sigma\subseteq T$ into a family $X\to T$.
By the neighborhood comparison result above, we may assume that we have a single neighborhood of $T_\RR\subseteq T_\CC^\aff$ and that the families $X_\sigma\to\sigma$, the restriction isomorphisms $X_\sigma\xrightarrow\sim X_\tau|_\sigma$ for $\sigma\subseteq\tau$, and the cocycle condition for $\sigma\subseteq\tau\subseteq\upsilon$ exist over the intersection of $\sigma_\CC^\aff$ with this neighborhood.
Now we claim that the colimit of the families $X_\sigma\to\sigma$ over these particular neighborhoods $\sigma_\CC$ provides the desired family $X\to T_\CC$.
This question is local on the topological colimit.
Locally near any point of the topological colimit, we can find compatible local trivializations of the families $X_\sigma\to\sigma_\CC$ by induction on simplices and applying Lemma \ref{functionextension}.
The colimit then exists and has the desired form by Corollary \ref{gluinghyperplanes}.
\end{proof}

Given a family of complex manifolds $X\to B$ over a finite simplicial complex $B$, there is an induced family $bX\to bB$ over the barycentric subdivision $bB$ of $B$, obtained by pulling back $X\to B$ under the canonical map $bB\to B$ (on geometric realizations, which, since it is linear on each simplex, extends to their complexifications).
There is moreover a concordance between $X\to B$ and $bX\to bB$ obtained by pulling back $X\to B$ under the map to $B$ from the concordance between $B$ and $bB$.

\section{Grothendieck group of 1-cycles}\label{grothendieckdefsec}

The goal of this section is to define the Grothendieck group of 1-cycles in complex threefolds $H^*_c(\cZ/\Cpx_3)$ and to establish some of its basic properties.

\subsection{Main definition}\label{grothtechdefsec}

Morally speaking, we would like to define
\begin{equation}\label{grothpredef}
H^*_c(\cZ/\Cpx_3)=H^*_c(\cZ(\Univ_3/\Cpx_3)/\Cpx_3)
\end{equation}
as the bivariant group $H^*_c$ (Definition \ref{bivariantmaindef}) of the relative cycle space $\cZ(\Univ_3/\Cpx_3)\to\Cpx_3$ (see Section \ref{spaceofcycles}) of the universal family $\Univ_3\to\Cpx_3$ over the moduli stack $\Cpx_3$ of complex threefolds (defined by the property that a map $B\to\Cpx_3$ from any complex analytic space $B$ is a smooth separated morphism of relative dimension three $X\to B$).
However, we wish to avoid the technicalities involved in making sense of the bivariant theory $H^*_c$ for complex analytic stacks (for example, the infinite-dimensionality of $\Cpx_3$ means the description of its homology via the $a^!$ functor needs adjustment), so we will instead take a (directed) colimit of the groups $H^*_c(\cZ(X/B)/B)$ over families of threefolds $X\to B$.
For reasons already mentioned in the Introduction, we also choose to replace what would perhaps be the `default' class of base spaces $B$, namely complex analytic spaces, with finite simplicial complexes.
Recall from Section \ref{complexoversimplex} the notion of a family of complex threefolds $X\to B$ over a simplicial complex $B$.
For any such family, there is a relative cycle space $\cZ(X/B)\to B$ as defined in Section \ref{spaceofcycles} (that is, take the relative cycle space of the complexification $X_\CC\to B_\CC$ and then pull it back to the geometric realization $B_\RR$).

\begin{definition}[Grothendieck group of 1-cycles in complex threefolds; compare Definition \ref{grothintrodef}]\label{grothgpdef}
The group $H^*_c(\cZ/\Cpx_3)$ defined as the directed colimit
\begin{equation}
H^*_c(\cZ/\Cpx_3):=\dircolim_{X\to B}H^*_c(\cZ(X/B)/B)
\end{equation}
over the category whose objects are (not necessarily proper) families of complex threefolds $X\to B$ over finite simplicial complexes $B$ (in the sense of Definition \ref{familysimplicialfinite}) and whose morphisms are commuting squares
\begin{equation}
\begin{tikzcd}
X'\ar[r]\ar[d]&X\ar[d]\\
B'\ar[r]&B
\end{tikzcd}
\end{equation}
where $B'\to B$ is an \emph{injective} map of simplicial complexes and for which the induced map $X'\to X\times_BB'$ is an \emph{open embedding}.
Dually we define $H_*^\relinfty(\cZ/\Cpx_3)_\naive$ as the inverse limit of $H_*^\relinfty(\cZ(X/B)/B)$ over (the opposite of) this same indexing category (due to the non-exactness of inverse limits, this group carries the $_\naive$ subscript and is slightly more technical to work with).
\end{definition}

Before proceeding further, we must rectify a small lie in this definition.
The category of families of complex threefolds over finite simplicial complexes is actually \emph{not} directed.
But, the diagrams we shall consider over it all factor through a directed quotient of it, as we now show in Lemma \ref{familyfiltered}.
When we write directed colimit over families of threefolds over finite simplicial complexes, we will always mean the colimit over this directed quotient.
Directedness will be of crucial importance later (directed colimits have much better properties than general colimits).

\begin{lemma}\label{familyfiltered}
Consider the category of families of threefolds $X\to B$ over finite simplicial complexes $B$ as in Definition \ref{grothgpdef}.
Identify two morphisms $f,g:B'\to B$ when they extend to a morphism $B'\times[0,1]\to B$, where the product $B'\times[0,1]$ is taken with respect to some local ordering on $B'$ (this relation is respected by composition, hence we still have a category).
This new category is directed.
\end{lemma}

\begin{proof}
Recall that a category is called directed (aka filtered) when (1) it is non-empty, (2) for every pair of objects $X$ and $Y$, there exists a pair of maps $X\to Z\leftarrow Y$ to a common object $Z$, and (3) for every pair of morphisms $f,g:X\arrowstack{r,r}Y$, there exists a morphism $Y\to Z$ for which the two compositions $X\arrowstack{r,r}Y\to Z$ coincide.

Condition (1) is immediate (take $B=\varnothing$).

Condition (2) is achieved by taking the disjoint union of families $X\to B$ and $X'\to B'$, embedding both into $X\sqcup X'\to B\sqcup B'$.

Condition (3) is achieved as follows.
Suppose given two injective maps $f,g:B'\arrowstack{r,r}B$ (covered by maps between the families $X'\to B'$ and $X\to B$).
Naturally, we just want to consider the inclusion $B\to B^+$ where $B^+$ is obtained from $B$ by gluing $B'\times[0,1]$ (with respect to an arbitrary local ordering on $B'$) to it along $f$ and $g$ along $B'\times\{0\}$ and $B'\times\{1\}$, respectively, where we equip $B'\times[0,1]$ with the family given by the union of $f^*X\times[0,\frac 12)$, $g^*X\times(\frac 12,1]$, and $X'\times[0,1]$ (glued together in the evident way).
We cannot quite glue in $B'\times[0,1]$ to $B$ in this way in the world of simplicial complexes (the issue is that a given set of vertices may not bound more than one simplex), but it works if we instead stack together three copies of $B'\times[0,1]$ and glue the two extreme ends using $f$ and $g$.
\end{proof}

\begin{lemma}\label{subdivisiongrothendieck}
For any family of threefolds $X\to B$ over a finite simplicial complex $B$, the following diagram commutes:
\begin{equation}
\begin{tikzcd}[row sep=tiny]
H^*_c(\cZ(X/B)/B)\ar[rd]\ar[dd,equals]\\
&H^*_c(\cZ/\Cpx_3)\\
H^*_c(\cZ(bX/bB)/bB)\ar[ru]
\end{tikzcd}
\end{equation}
where $bX\to bB$ denotes the induced family over the barycentric subdivision (see Section \ref{complexoversimplex}).
\end{lemma}

\begin{proof}
Consider the concordance between $X\to B$ and $bX\to bB$, noting that its relative cycle space is $\cZ(X/B)\times[0,1]$.
\end{proof}

Now let us remark on the relation between the bivariant groups $H^*_c(\cZ(X/B)/B)$ for families of threefolds $X\to B$ over complex analytic spaces $B$ and the Grothendieck group of 1-cycles $H^*_c(\cZ/\Cpx_3)$.
Suppose we are given a \emph{real analytic triangulation} of $B$, by which we mean a finite simplicial complex $T$ and a homeomorphism $T_\RR\to B$ (where $T_\RR$ denotes the geometric realization of $T$) which is real analytic on each simplex (a map $\sigma\to B$ is called real analytic when it is the restriction of a complex analytic map $\sigma_\CC\to B$).
Since $T\to B$ is real analytic, we can form the pullback family $X\times_BT\to T$ and the map $H^*_c(\cZ(X/B)\tighttimes_BT/T)=H^*_c(\cZ(X\tighttimes_BT/T)/T)\to H^*_c(\cZ/\Cpx_3)$.
Since $T\to B$ is a homeomorphism, we have $H^*_c(\cZ(X/B)/B)=H^*_c(\cZ(X/B)\tighttimes_BT/T)$, so we obtain by composition a map $H^*_c(\cZ(X/B)/B)\to H^*_c(\cZ/\Cpx_3)$, depending on the choice of real analytic triangulation of $B$.
It is then straightforward to check that if two real analytic triangulations of $B$ are (real analytically) \emph{concordant}, in the sense that they extent to a real analytic triangulation of $B\times[0,1]$, then they induce the same map $H^*_c(\cZ(X/B)/B)\to H^*_c(\cZ/\Cpx_3)$.
The condition that $T_\RR\to B$ be a homeomorphism may be relaxed to being a closed embedding (this way we can deal with non-compact $B$), provided we consider $H^*_c$ of the restriction of $\cZ(X/B)\to B$ to the interior of the image of $T_\RR\to B$.

When $B$ is smooth and compact Hausdorff, it has a canonical concordance class of real analytic triangulation, represented by any real analytic triangulation for which every simplex is immersed (that is, the differential of each map $\sigma\to B$ is injective at every point of the closed simplex $\sigma$).
More generally, if $B$ is just smooth and Hausdorff, there is a canonical concordance class of real analytic triangulations of compact subsets of $B$ containing in their interior a fixed compact $K\subseteq B$.
The case that $B$ is singular is, of course, much harder, and beyond the scope of our current treatment.
We will often represent elements of $H^*_c(\cZ/\Cpx_3)$ using families $X\to B$ over complex manifolds $B$, implicitly appealing to these triangulations.

\begin{lemma}\label{bivariantdoublecomplex}
Let $X\to B$ be a family of threefolds over a finite simplicial complex $B$.
We have
\begin{equation}
H^*_c(\cZ(X/B)/B)=H^*_c\Bigl(\cZ(X/B),\textstyle\bigoplus\limits_{\sigma\subseteq B}\ZZ_{\cZ(X\tighttimes_B\sigma/\sigma)}[\dim\sigma]\Bigr)
\end{equation}
where $\ZZ_{\cZ(X\tighttimes_B\sigma/\sigma)}$ denotes the constant sheaf $i_*\ZZ$ on the closed subset $\cZ(X\tighttimes_B\sigma/\sigma)\subseteq\cZ(X/B)$ and the direct sum of all of these is equipped with the differential given by the sum of all restrictions along codimension one simplex inclusions (with the usual orientation signs).
Thus $H^*_c(\cZ(X/B)/B)$ is the cohomology of the double complex $\bigoplus_{\sigma\subseteq B}C^{*+\dim\sigma}_c(\cZ(X\tighttimes_B\sigma/\sigma))$ whose differentials are the internal differential of $C^*_c$ and the sum of all restrictions along codimension one simplex inclusions (with the usual orientation signs).
This is a `second quadrant' double complex, illustrated as follows:
\begin{equation}
\begin{tikzcd}[row sep=scriptsize,column sep=tiny]
&\vdots&\vdots&\vdots\\
\cdots\ar[r]&\ar[u]\smash{\textstyle\bigoplus\limits_{\Delta^2\subseteq B}}C^2_c(\cZ(X/\Delta^2))\ar[r]&\ar[u]\smash{\textstyle\bigoplus\limits_{\Delta^1\subseteq B}}C^2_c(\cZ(X/\Delta^1))\ar[r]&\ar[u]\smash{\textstyle\bigoplus\limits_{\Delta^0\subseteq B}{}}C^2_c(\cZ(X/\Delta^0))\\
\cdots\ar[r]&\ar[u]\smash{\textstyle\bigoplus\limits_{\Delta^2\subseteq B}}C^1_c(\cZ(X/\Delta^2))\ar[r]&\ar[u]\smash{\textstyle\bigoplus\limits_{\Delta^1\subseteq B}}C^1_c(\cZ(X/\Delta^1))\ar[r]&\ar[u]\smash{\textstyle\bigoplus\limits_{\Delta^0\subseteq B}{}}C^1_c(\cZ(X/\Delta^0))\\
\cdots\ar[r]&\ar[u]\textstyle\bigoplus\limits_{\Delta^2\subseteq B}C^0_c(\cZ(X/\Delta^2))\ar[r]&\ar[u]\textstyle\bigoplus\limits_{\Delta^1\subseteq B}C^0_c(\cZ(X/\Delta^1))\ar[r]&\ar[u]\textstyle\bigoplus\limits_{\Delta^0\subseteq B}C^0_c(\cZ(X/\Delta^0))
\end{tikzcd}
\end{equation}
Its `total complex' means we take the direct sum over the anti-diagonals.
\end{lemma}

\begin{proof}
The dualizing complex $\omega_B=(B\to *)^!\ZZ$ is given by $\bigoplus_{\sigma\subseteq B}\ZZ_\sigma[\dim\sigma]$, where $\ZZ_\sigma$ denotes the constant sheaf $i_*\ZZ$ on the closed simplex $\sigma\subseteq B$, equipped with the differential given by the sum of all restrictions along codimension one simplex inclusions, with the usual orientation signs.
Now pull back to $\cZ(X/B)$ and take compactly supported cohomology.
\end{proof}

\begin{lemma}\label{bivariantsmoothbasecycles}
Let $X\to B$ be a family of threefolds over an oriented real analytic manifold $B$ of dimension $b$.
The bivariant group $H^*_c(\cZ(X/B)/B)$ is the shifted cohomology $H^{*+b}_c(\cZ(X/B))$.
\end{lemma}

\begin{proof}
The dualizing complex of $B$ is $\ZZ[b]$.
\end{proof}

\begin{remark}\label{doublecomplexdef}
It would be equivalent to define $H^*_c(\cZ/\Cpx_3)$ as the cohomology of the double complex $\bigoplus_{p\geq 0}\bigoplus_{X\to\Delta^p}C^{*+p}_c(\cZ(X/\Delta^p))$ with differentials given by the internal differential and the sum over all codimension one simplex inclusions, with the usual orientation signs (compare the double complex in Lemma \ref{bivariantdoublecomplex}).
To make sense of this direct sum over `all families of threefolds over $\Delta^p$' (which is not a set), we may regard it as a simplicial groupoid and `resolve' it by a simplicial set mapping to it via a trivial Kan fibration.
This was our original definition of $H^*_c(\cZ/\Cpx_3)$ (the difference between it and Definition \ref{grothgpdef} is purely technical).
\end{remark}

\begin{remark}\label{shgroth}
The group $H^*_c(\cZ/\Cpx_3)$ is naturally the homology of a spectrum (object in the stable homotopy category).
One way to see this is to port the bivariant theory $H^*_c$ in Definition \ref{bivariantmaindef} to the stable homotopy category (consider sheaves of spectra rather than sheaves of complexes of $\ZZ$-modules).
Explicitly, the model in Remark \ref{doublecomplexdef} can be lifted to the stable homotopy category as
\begin{equation}
\colim\Bigl(\cdots\arrowstack{r,r,r,r}\coprod_{X\to\Delta^2}D(\cZ(X/\Delta^2)/\infty)\arrowstack{r,r,r}\coprod_{X\to\Delta^1}D(\cZ(X/\Delta^1)/\infty)\arrowstack{r,r}\coprod_XD(\cZ(X)/\infty)\Bigr)
\end{equation}
where by $D$ we mean Spanier--Whitehead dual (more precisely, $A/\infty$ denotes the inverse system $\{A/(A\setminus K)\}_{K\subseteq A\text{ compact}}$, and by $D(A/\infty)$ we mean the colimit of the directed system obtained by applying the contravariant Spanier--Whitehead duality functor $D$).
\end{remark}

\subsection{Algebraic structure}\label{grothalg}

We now endow the Grothendieck group $H^*_c(\cZ/\Cpx_3)$ with the structure of a bi-algebra.
The product corresponds to `disjoint union of cycles', while the coproduct (which exists modulo torsion) corresponds to `addition of cycles'.
There are also `division by $d$' operations on $H^*_c(\cZ/\Cpx_3)$ for integers $d\geq 1$, which are bi-algebra endomorphisms.

\begin{definition}[Product $\mu$ on $H_c^*(\cZ/\Cpx_3)$]\label{productdef}
Given two families $X\to B$ and $X'\to B'$, we may consider the product base space $B\times B'$ and the `disjoint union family' $X\tighttimes B'\sqcup B\tighttimes X'\to B\times B'$ over it.
The relative cycle space of the disjoint family is the product of the relative cycle spaces of the two input families: $\cZ((X\tighttimes B'\sqcup B\tighttimes X')/(B\times B'))=\cZ(X/B)\times\cZ(X'/B')$.
The Künneth morphism for the bivariant theory $H^*_c$ thus gives a map
\begin{multline}
H^*_c(\cZ(X/B)/B)\otimes H^*_c(\cZ(X'/B')/B')\\\to H^*_c(\cZ((X\tighttimes B'\sqcup B\tighttimes X')/(B\tighttimes B'))/(B\tighttimes B')).
\end{multline}
Now the target maps canonically to $H^*_c(\cZ/\Cpx_3)$ by triangulating $B\times B'$ (the resulting map is independent of the choice of triangulation by a concordance argument).
Taking the directed colimit over all families $X\to B$ and $X'\to B'$ thus yields a map
\begin{equation}
\mu:H^*_c(\cZ/\Cpx_3)^{\otimes 2}\to H^*_c(\cZ/\Cpx_3).
\end{equation}
This product is evidently associative and (graded) commutative with unit $\eta=1\in H^0_c(*/*)=H^0_c(\cZ(\varnothing/*)/*)\to H^0_c(\cZ/\Cpx_3)$.
\end{definition}

\begin{definition}[Coproduct $\Delta$ on $H_c^*(\cZ/\Cpx_3)/\tors$]\label{coproductdef}
The addition map $\Sigma:\cZ(X/B)\times_B\cZ(X/B)\to\cZ(X/B)$ over $B$ is proper, hence by the functoriality of the bivariant theory $H^*_c$ determines a pullback map
\begin{equation}
H^*_c(\cZ(X/B)/B)\to H^*_c((\cZ(X/B)\times_B\cZ(X/B))/B).
\end{equation}
Now $\cZ(X/B)\times_B\cZ(X/B)\to B$ is the pullback along the diagonal of the map $\cZ(X/B)\times\cZ(X/B)\to B\times B$, which gives a pushforward map
\begin{equation}
H^*_c((\cZ(X/B)\times_B\cZ(X/B))/B)\to H^*_c((\cZ(X/B)\times\cZ(X/B))/(B\times B)).
\end{equation}
Finally, the Künneth isomorphism for the bivariant theory $H^*_c$ identifies this target with $H^*_c(\cZ(X/B)/B)^{\otimes 2}$, at least modulo torsion.
Composing everything, we obtain a map $H^*_c(\cZ(X/B)/B)/\tors\to(H^*_c(\cZ(X/B)/B)/\tors)^{\otimes 2}$.
Taking the directed colimit over all families $X\to B$, we obtain a map
\begin{equation}
\Delta:H^*_c(\cZ/\Cpx_3)/\tors\to(H^*_c(\cZ/\Cpx_3)/\tors)^{\otimes 2}.
\end{equation}
A diagram chase shows that this coproduct is coassociative and (graded) cocommutative with counit $\varepsilon:H_c^*(\cZ/\Cpx_3)\to\ZZ$ given by `evaluate at the empty cycle' (that is, act on $H^*_c(\cZ(X/B)/B)$ by pulling back under the zero section $B\to\cZ(X/B)$ to get to $H^*_c(B/B)=H^*_c(B,\omega_B)=H_*(B)$, which we map to $\ZZ$ by pairing with $1\in H^0(B)$).
\end{definition}

\begin{lemma}
The operations $(\eta,\mu,\varepsilon,\Delta)$ define the structure of a bi-algebra on $H_c^*(\cZ/\Cpx_3)/\tors$.
This means:
\begin{itemize}
\item$(R,\eta,\mu)$ is an algebra (satisfies unitality and associativity).
\item$(R,\varepsilon,\Delta)$ is a co-algebra (satisfies co-unitality and co-associativity).
\item The maps $\eta$ and $\Delta$ are algebra maps (equivalently, the maps $\varepsilon$ and $\mu$ are co-algebra maps).
\end{itemize}
\end{lemma}

\begin{proof}
Diagram chase.
\end{proof}

\begin{definition}[Multiplicative curve enumeration theory]\label{weaklymultiplicative}
A curve enumeration theory $V\in H_*^\relinfty(\cZ/\Cpx_3)_\naive$ shall be called \emph{(weakly) multiplicative} when it satisfies $V(\mu(x,y))=V(x)V(y)$ for all $x,y\in H_c^*(\cZ/\Cpx_3)$.
\end{definition}

\begin{definition}[Division]\label{rhodivision}
For any $d\geq 1$, the `multiply by $d$' map $\cZ(X/B)\to\cZ(X/B)$ (which is proper) determines a pullback map $H^*_c(\cZ(X/B)/B)\to H^*_c(\cZ(X/B)/B)$.
Taking the colimit over all $X\to B$, we obtain a map $\rho_d:H^*_c(\cZ/\Cpx_3)\to H^*_c(\cZ/\Cpx_3)$.
\end{definition}

The maps $\rho_d$ are bi-algebra morphisms by inspection.

\subsection{Restricting the class of 1-cycles}\label{modifiedgroups}

We will also want to form Grothendieck groups of 1-cycles satisfying certain properties (e.g.\ semi-Fano 1-cycles $\cZ_\semiFano$ or 1-cycles satisfying a multiplicity condition $\cZ_{\leq\bm}$).
Here is a general framework for defining such groups:

\begin{definition}[Modified Grothendieck group of 1-cycles $H^*_c(\cZ_\alpha/\Cpx_3)$]\label{grothgpdefmodifiedeasy}
Let $\cZ_\alpha\subseteq\cZ$ be the specification of a locally closed subset $\cZ_\alpha(X/B)\subseteq\cZ(X/B)$ for every family of threefolds $X\to B$ over a finite simplicial complex $B$, satisfying the following properties:
\begin{itemize}
\item$\cZ_\alpha\subseteq\cZ$ is compatible with pullback, in the sense that $\cZ_\alpha(X/B)\times_BB'=\cZ_\alpha(X\times_BB'/B')$ as subsets of $\cZ(X/B)\times_BB'=\cZ(X\times_BB'/B')$ for all maps of simplicial complexes $B'\to B$.
\item$\cZ_\alpha\subseteq\cZ$ is local, in the sense that for any open set $X^-\subseteq X\to B$, we have $\cZ_\alpha(X^-/B)=\cZ(X^-/B)\cap\cZ_\alpha(X/B)$.
\end{itemize}
We may then define $H^*_c(\cZ_\alpha/\Cpx_3):=\dircolim_{X\to B}H^*_c(\cZ_\alpha(X/B)/B)$.
\end{definition}

Recall that for a locally closed subset $A\subseteq Z$, we define $\ZZ_A=i_*j_!\ZZ$ for $j:A\hookrightarrow\overline A$ and $i:\overline A\to Z$ (or, in fact, for any factorization of $A\to Z$ as the composition of an open embedding followed by a closed embedding).
Now Lemma \ref{bivariantdoublecomplex} generalizes to tell us that
\begin{equation}\label{bivariantdoublecomplexalpha}
H^*_c(\cZ_\alpha(X/B)/B)=H^*_c\Bigl(\cZ(X/B),\textstyle\bigoplus\limits_{\sigma\subseteq B}\ZZ_{\cZ_\alpha(X\tighttimes_B\sigma/\sigma)}[\dim\sigma]\Bigr).
\end{equation}
If $A,B\subseteq Z$ are locally closed subsets and $A\cap B\subseteq A$ is closed and $A\cap B\subseteq B$ is open, then there is an induced map $\ZZ_A\to\ZZ_B$ (call this a \emph{good} pair of locally closed subsets).
We conclude that if $\cZ_\alpha\cap\cZ_\beta\subseteq\cZ_\alpha$ is closed and $\cZ_\alpha\cap\cZ_\beta\subseteq\cZ_\beta$ is open, then we get a map $H^*_c(\cZ_\alpha/\Cpx_3)\to H^*_c(\cZ_\beta/\Cpx_3)$.
If $\cZ_\beta$ is the disjoint union of an open subset $\cZ_\alpha\subseteq\cZ_\beta$ and its closed complement $\cZ_\gamma\subseteq\cZ_\beta$, then there is a long exact sequence
\begin{equation}\label{modifiedgrothles}
\cdots\to H^*_c(\cZ_\alpha/\Cpx_3)\to H^*_c(\cZ_\beta/\Cpx_3)\to H^*_c(\cZ_\gamma/\Cpx_3)\to\cdots
\end{equation}
since directed colimts are exact.

For example, we could consider the open locus of semi-Fano 1-cycles $\cZ_\semiFano\subseteq\cZ$, and Definition \ref{grothgpdefmodifiedeasy} produces a group $H^*_c(\cZ_\semiFano/\Cpx_3)$ with a map $H^*_c(\cZ_\semiFano/\Cpx_3)\to H^*_c(\cZ/\Cpx_3)$ (pushforward of compactly supported cohomology under open embeddings).
We could also consider the loci $\cZ_{\leq\bm}$ from Definition \ref{multfilt} and produce a long exact sequence
\begin{equation}
\cdots\to H^*_c(\cZ_{<\bm}/\Cpx_3)\to H^*_c(\cZ_{\leq\bm}/\Cpx_3)\to H^*_c(\cZ_{=\bm}/\Cpx_3)\to\cdots
\end{equation}
for any $\bm$.

The bi-algebra operations on $H^*_c(\cZ/\Cpx_3)$ extend to the variants $H^*_c(\cZ_\alpha/\Cpx_3)$ under natural hypotheses.
The product construction in Definition \ref{productdef} yields a map $H^*_c(\cZ_\alpha/\Cpx_3)\otimes H^*_c(\cZ_\beta/\Cpx_3)\to H^*_c(\cZ_\gamma/\Cpx_3)$ provided that for a disjoint union of families $X\sqcup X'\to B$, the pair of locally closed subsets $\cZ(X/B)_\alpha\times_B\cZ_\beta(X'/B)$ and $\cZ_\gamma((X{\sqcup}X')/B)$ of $\cZ((X{\sqcup}X')/B)=\cZ(X/B)\times_B\cZ(X'/B)$ form a good pair (in practice, the former will just be an open subset of the latter).
For example, $H^*_c(\cZ_\semiFano/\Cpx_3)$ has a product, and there are product operations $H^*_c(\cZ_{\leq\bm}/\Cpx_3)\otimes H^*_c(\cZ_{\leq\bm'}/\Cpx_3)\to H^*_c(\cZ_{\leq(\bm,\bm')}/\Cpx_3)$.

The coproduct construction in Definition \ref{coproductdef} yields a map $H^*_c(\cZ_\alpha/\Cpx_3)\to(H^*_c(\cZ_\alpha/\Cpx_3)/\tors)^{\otimes 2}$ provided that the sum operation sends $\cZ_\alpha\times\cZ_\alpha\to\cZ_\alpha$ and as such is proper (for example, it suffices to know that $z+z'$ is an $\alpha$-cycle iff $z$ and $z'$ are both $\alpha$-cycles).
For example, this applies to $\cZ_\semiFano$.

The operations $\rho_d$ from Definition \ref{rhodivision} are defined on $H^*_c(\cZ_\alpha/\Cpx_3)$ provided `mutiply by $d$' is proper $\cZ_\alpha\to\cZ_\alpha$ (which holds, for example, when $d\cdot z$ is an $\alpha$-cycle iff $z$ is an $\alpha$-cycle); this also holds for $\cZ_\semiFano$.

\subsection{Virtual fundamental classes}\label{vfcsubsec}

Now let us recall the standard construction of curve enumeration theories, that is classes in $H_*^\relinfty(\cZ/\Cpx_3)_\naive$, based on the theory of virtual fundamental classes and perfect obstruction theories that we review in Appendix \ref{vfc}.

Curve enumeration theories usually arise via proper pushforward along a map $\cE\to\cZ$, where $\cE$ is another functor associating to each family of threefolds $X\to B$ over a complex analytic base $B$ a complex analytic space (or Deligne--Mumford complex analytic stack) $\cE(X/B)\to B$, compatible with pullback.
Given such an $\cE$, we may define $H_*^\relinfty(\cE/\Cpx_3)_\naive$ exactly as in Definition \ref{grothgpdef} just with $\cE$ in place of $\cZ$.
Given a natural transformation $\cE\to\cZ$ which is proper (meaning $\cE(X/B)\to\cZ(X/B)$ is proper for every family $X\to B$), we get a proper pushforward map $H_*^\relinfty(\cE/\Cpx_3)_\naive\to H_*^\relinfty(\cZ/\Cpx_3)_\naive$.
So, it suffices to define classes in $H_*^\relinfty(\cE/\Cpx_3)_\naive$ for such $\cE\to\cZ$.

If $\cE$ is smooth, in the sense that $\cE(X/B)\to B$ is smooth (i.e.\ submersive) for every $X\to B$, then there is a canonical relative (i.e.\ vertical) fundamental class in $H_*^\relinfty(\cE(X/B)/B)$, lying in degree given by the relative dimension of $\cE(X/B)\to B$, compatible with pullback.
More generally, we may ask that $\cE$ have a (relative) \emph{perfect obstruction theory} in the sense of Behrend--Fantechi \cite{intrinsicnormalcone} (see Definition \ref{potnondef}), meaning that $\cE(X/B)\to B$ has a perfect obstruction theory for every $X\to B$, compatible with pullback.
A perfect obstruction theory on $\cE$ induces a (relative) \emph{virtual fundamental class} in $H_*^\relinfty(\cE(X/B)/B)$ (see Definition \ref{vfcdef}) for every $X\to B$, compatible with pullback (by Lemma \ref{vfcpullbackpot}).

Now consider the case of a family of threefolds $X\to B$ over a finite simplicial complex $B$ (instead of a complex analytic space).
We have the complexified family $X_\CC\to B_\CC$ for some open set $B_\CC\subseteq B_\CC^\aff$ containing the geometric realization $B_\RR$ (as discussed in Section \ref{complexoversimplex}), hence its virtual fundamental class in $H_*^\relinfty(\cE(X_\CC/B_\CC)/B_\CC)$ defined above.
We may then define the virtual fundamental class in $H_*^\relinfty(\cE(X/B)/B)$ (though we write $B$ we really mean $B_\RR$) to be its pullback.
This class is compatible with pullback under maps of finite simplicial complexes $B'\to B$, so we obtain a `virtual fundamental' class $[\cE]^\vir\in H_*^\relinfty(\cE/\Cpx_3)_\naive$.

The virtual fundamental class $[\cE]^\vir\in H_*^\relinfty(\cE/\Cpx_3)_\naive$ lies in homological degree $i$ equal to (twice) the (complex) index of the perfect obstruction theory on $\cE$.
For families of \emph{three}folds, this typically equals $2k$ (twice the chern number $k$), and hence $[\cE]^\vir$ has `virtual codimension' $2k-i=0$.

\begin{remark}\label{vfccoherent}
We saw in Lemma \ref{familysimplicialtotalspace} that a family of threefolds $X\to B$ over a finite simplicial complex $B$ (in the sense of having a family of threefolds over the complexification $B_\CC$) is the same thing as a system of families over the complexifications $\sigma_\CC$ of each simplex $\sigma\subseteq B$, with restriction isomorphisms for simplex inclusions.
It is worthwhile to point out that the above discussion really relies on having families in the former sense, not the latter.
To extract a relative virtual fundamental class in $H_*^\relinfty(\cE(X/B)/B)$ out of a family in the latter sense (without appealing to Lemma \ref{familysimplicialtotalspace}) would require a discussion of homotopically coherent perfect obstruction theories (or, more likely, quasi-smooth derived structures) and their homotopically coherent virtual fundamental cycles (such a theory surely exists, but is much more technical than the classical theory outlined in Appendix \ref{vfc}, which is all we are relying on here).
Philosophically speaking, this is an instance of the general principle that a relative construction for arbitrary families is often a close enough substitute for a homotopy coherent construction.
\end{remark}

Now suppose $\cE$ is \emph{multiplicative} in the sense that for any pair of families of threefolds $X\to B\leftarrow X'$ over a complex analytic base $B$, there is a functorial isomorphism $\cE((X\sqcup X')/B)=\cE(X/B)\times_B\cE(X'/B)$, compatible with perfect obstruction theories.
It then follows from compatibility of virtual fundamental classes with (fiber) product (Lemmas \ref{vfcpullbackpot} and \ref{vfcproductpot}) that the Künneth morphism
\begin{multline}
H_*^\relinfty(\cE(X/B)/B)\otimes H_*^\relinfty(\cE(X'/B')/B')\\\to H_*^\relinfty(\cE((X\tighttimes B'\sqcup B\tighttimes X')/(B\times B'))/(B\times B'))
\end{multline}
sends the product of the virtual fundamental classes of $\cE(X/B)\to B$ and $\cE(X'/B')\to B'$ to that of $\cE((X\tighttimes B'\sqcup B\tighttimes X')/(B\times B'))\to(B\times B')$.
It follows that $[\cE]^\vir$ is multiplicative in the sense of Definition \ref{weaklymultiplicative}.

\subsection{Enumerative invariants}\label{zgwptsummary}

We now define the Gromov--Witten and Pandharipande--Thomas virtual fundamental classes in $H_*^\relinfty(\cZ/\Cpx_3;\QQ)((u))$ and $H_*^\relinfty(\cZ/\Cpx_3)((q))$.
We denote by
\begin{align}
\label{gwptmapsI}\GW:H^*_c(\cZ/\Cpx_3)&\to\QQ((u))\\
\label{gwptmapsII}\PT:H^*_c(\cZ/\Cpx_3)&\to\ZZ((q))
\end{align}
the resulting homomorphisms, which are in fact ring homomorphisms since these virtual fundamental classes are multiplicative.

Gromov--Witten and Pandharipande--Thomas invariants are defined using moduli spaces $\cMbarprime(X/B)$ and $P(X/B)$ (respectively) over $B$ associated to any family of threefolds $X\to B$ over a complex analytic space $B$.
The moduli space $\cMbarprime(X/B)$ is a Deligne--Mumford analytic stack representing stable maps from compact (not necessarily connected) nodal curves to fibers of $X\to B$, all of whose connected components are non-constant.
The analytic space $P(X/B)$ parameterizes stable pairs on fibers of $X\to B$ (a stable pair is a coherent sheaf $F$ of proper support with the property that the support of every subsheaf is pure dimension one, along with a section $s$ whose cokernel has relative dimension zero) \cite{pandharipandethomasstablepairs,lepotier,douady}.
There are locally constant maps
\begin{alignat}{2}
\chi&:{}&\cMbarprime(X/B)&\to\ZZ\\
n&:{}&P(X/B)&\to\ZZ
\end{alignat}
given by domain arithmetic Euler characteristic and holomorphic Euler characteristic, respectively.

Both $\cMbarprime(X/B)\to B$ and $P(X/B)\to B$ carry a natural (relative) perfect obstruction theory, compatible with pullback.
As reviewed in Section \ref{vfcsubsec}, there are hence induced virtual fundamental classes
\begin{alignat}{2}
[\cMbarprime/\Cpx_3]^\vir&=\prod_{X\to\Delta^k}[\cMbarprime(X/\Delta^k)]^\vir&&\in H_*^\relinfty(\cMbarprime/\Cpx_3;\QQ),\\
[P/\Cpx_3]^\vir&=\prod_{X\to\Delta^k}[P(X/\Delta^k)]^\vir&&\in H_*^\relinfty(P/\Cpx_3).
\end{alignat}
Now the maps $\cMbarprime\to\cZ$ and $P\to\cZ$ are proper when restricted to the sets on which $-\chi$ and $n$ are bounded above by a given $N<\infty$, respectively.
Pushing forward $u^{-\chi}\cdot[\cMbarprime/\Cpx_3]^\vir$ and $q^n\cdot[P/\Cpx_3]^\vir$ thus defines classes
\begin{align}
\GW&\in H_*^\relinfty(\cZ/\Cpx_3;\QQ)((u)),\\
\PT&\in H_*^\relinfty(\cZ/\Cpx_3)((q)),
\end{align}
which have virtual codimension zero since the virtual fundamental classes of $\cMbarprime$ and $P$ lie in homological degree twice the chern number.
This defines the group homomorphisms \eqref{gwptmapsI}--\eqref{gwptmapsII}.

The moduli spaces $\cMbarprime$ and $P$ are `multiplicative' in the sense that $\cMbarprime((X\sqcup Y)/B)=\cMbarprime(X/B)\times_B\cMbarprime(Y/B)$ compatibly with perfect obstruction theories (and the same for $P$).
As reviewed in Section \ref{vfcsubsec}, it follows that the induced virtual fundamental classes $[\cMbarprime]^\vir$ and $[P]^\vir$ are also multiplicative in the sense that \eqref{gwptmapsI}--\eqref{gwptmapsII} are ring homomorphisms.

Classical Gromov--Witten and Pandharipande--Thomas theory is interested in evaluating $\GW$ and $\PT$ on elements of $H^*_c(\cZ/\Cpx_3)$ coming from projective threefolds.
When $X$ is projective, the space of cycles $\cZ(X,\beta)$ in homology class $\beta$ is compact, hence its characteristic function defines a class $(X,\beta)\in H^0_c(\cZ(-,\langle c_1(TX),\beta\rangle)/\Cpx_3)$, which has virtual dimension $2\langle c_1(TX),\beta\rangle$.
Thus when $\langle c_1(TX),\beta\rangle=0$, we may evaluate the homomorphisms $\GW$ and $\PT$ on this element to obtain invariants
\begin{align}
\GW(X,\beta)&=\int_{[\cMbarprime(X,\beta)]^\vir}u^{-\chi}\in\QQ((u)),\\
\PT(X,\beta)&=\int_{[P(X,\beta)]^\vir}q^n\in\ZZ((q)).
\end{align}
More generally, given cohomology classes $\gamma_1,\ldots,\gamma_r\in H^*(X)$ (called `insertions'), we may consider the class
\begin{equation}
(X,\beta;\gamma_1,\ldots,\gamma_r)\in H^{(\abs{\gamma_1}-2)+\cdots+(\abs{\gamma_r}-2)}_c(\cZ(-,\langle c_1(TX),\beta\rangle)/\Cpx_3)
\end{equation}
given by the cohomology class $1_\beta\prod_{i=1}^r\pi_!i^*\gamma_i$ on $\cZ(X)$, namely the result of push/pull via the universal family.
\begin{equation}
\begin{tikzcd}
\cU(X)\ar[r,"i"]\ar[d,"\pi"]&X\\
\cZ(X)
\end{tikzcd}
\end{equation}
Evaluating $\GW$ and $\PT$ on this class produces Gromov--Witten invariants and Pandharipande--Thomas invariants of $X$ in homology class $\beta$ with insertions $\gamma_1,\ldots,\gamma_r$
\begin{align}
\GW(X,\beta;\gamma_1,\ldots,\gamma_r)&=\int_{[\cMbarprime(X,\beta)]^\vir}{\textstyle\prod\limits_{i=1}^r}\pi_!\ev^*\gamma_i\cdot u^{-\chi}\in\QQ((u))\\
\PT(X,\beta;\gamma_1,\ldots,\gamma_r)&=\int_{[P(X,\beta)]^\vir}{\textstyle\prod\limits_{i=1}^r}\pi_!(\ch_2(\FF)\cup\ev^*\gamma_i)\cdot q^n\in\ZZ((q))
\end{align}
where the integrand involves push/pull for the universal families
\begin{align}
\begin{tikzcd}[ampersand replacement=\&]
\cUbar'(X)\ar[r,"\ev"]\ar[d,swap,"\pi"]\&X\\
\cMbarprime(X)
\end{tikzcd}
&&&
\begin{tikzcd}[ampersand replacement=\&]
P(X)\times X\ar[r,"\ev=p_X"]\ar[d,swap,"\pi"]\&X\\
P(X)
\end{tikzcd}
\end{align}
and $\FF$ denotes the universal stable pair on $P(X)\times X$ (note that the second chern character $\ch_2(\FF)$ is simply the fundamental cycle of the support of $\FF$, a codimension four cohomology class on $P(X,\beta)\times X$).
These invariants vanish for dimension reasons except when the virtual dimension $2\langle c_1(TX),\beta\rangle-\sum_i(\abs{\gamma_i}-2)$ of the element in $H^*(\cZ/\Cpx_3)$ is zero.

\section{Local curves}\label{localcurvesec}

In the study of enumerative invariants of complex threefolds, the term \emph{local curve} refers to (the total space of) a rank two vector bundle $E$ over a smooth proper (usually connected) curve $C$.
Given a local curve $E\to C$, one is then interested in enumerating curves \emph{supported on the zero section} $C\subseteq E$; unfortunately, this has no meaning \emph{a priori} since $\ZZ_{\geq 0}\cdot[C]\subseteq\cZ(E)$ is usually not open.
The goal of this section is to recall how to make sense of the enumerative theory of local curves by working equivariantly, and to recall that equivariant invariants are a special case of family invariants (via the Borel construction) hence are realized naturally within the framework of the Grothendieck group $H^*_c(\cZ/\Cpx_3)$.

\begin{remark}\label{localdeform}
It is not hard to show that local curves are classified up to deformation by the pair of integers $g=g(C)\geq 0$ and $c=c_1(E)\in\ZZ$.
The chern number of the zero section is given by $k=c_1(TE)=c_1(E)+c_1(TC)=2-2g+c$ and is a more convenient index than $c$.
We write $E_{g,k}$ for the (unique up to deformation) local curve of genus $g\geq 0$ and chern number $k$.
\end{remark}

\subsection{Equivariant homology}

The flavor of equivariant homology relevant for our present discussion is called \emph{co-Borel equivariant homology}, which measures `homotopically $S^1$-invariant cycles' on an $S^1$-space.
We will describe this homology theory in terms of the bivariant theories from Section \ref{bivariantthry}.

Given an $S^1$-space $X$, denote by $X_{\CC P^N}\to\CC P^N$ (a locally trivial fibration with fiber $X$) the quotient $(X\times S^{2N+1})/S^1\to S^{2N+1}/S^1$, where $S^{2N+1}\subseteq\CC^{N+1}$ is the unit sphere acted on by the unit circle $S^1\subseteq\CC$ by multiplication.
Note that if the $S^1$-action on $X$ is the restriction to the unit circle $S^1\subseteq\CC^\times$ of a $\CC^\times$-action on $X$, then we also have $X_{\CC P^N}=(X\times(\CC^{N+1}-0))/\CC^\times$ (since $\CC^\times=S^1\times\RR_{>0}$ and $X\times S^{2N+1}\subseteq X\times(\CC^{N+1}-0)$ is a slice for the $\RR_{>0}$-action).

The families $X_{\CC P^N}\to\CC P^N$ are related by pullback diagrams
\begin{equation}\label{borelcompatible}
\begin{tikzcd}
X_{\CC P^N}\ar[r]\ar[d]&X_{\CC P^{N+1}}\ar[d]\\
\CC P^N\ar[r]&\CC P^{N+1}
\end{tikzcd}
\end{equation}
where $\CC P^N\subseteq\CC P^{N+1}$ is the standard hyperplane inclusion.
The directed system $\{X_{\CC P^N}\to\CC P^N\}_N$ will be called the \emph{Borel construction} of the $S^1$-space $X$.
Usually the term `Borel construction' refers to the single fibration $X_{\CC P^\infty}\to\CC P^\infty$, but for our present work it is more convenient to work with the sequence of finite-dimensional approximations $X_{\CC P^N}\to\CC P^N$.

\begin{definition}[co-Borel equivariant homology and cohomology]\label{coboreldef}
Let $X$ be an $S^1$-space with reasonable topology (say Hausdorff, paracompact, and locally homeomorphic to a finite CW-complex of uniformly bounded dimension).
The co-Borel $S^1$-equivariant homology rel infinity of $X$ is given by the inverse limit
\begin{equation}\label{coborelinvlim}
H^{cS^1,\relinfty}_*(X)=\invlim_NH_*^\relinfty(X_{\CC P^N}/\CC P^N).
\end{equation}
Dually, the co-Borel $S^1$-equivariant compactly supported cohomology of $X$ is the directed colimit
\begin{equation}\label{coboreldirlim}
H_{cS^1,c}^*(X)=\dircolim_NH^*_c(X_{\CC P^N}/\CC P^N).
\end{equation}
Since $\CC P^N$ are compact oriented manifolds, the bivariant groups appearing in \eqref{coborelinvlim}--\eqref{coboreldirlim} may be expressed in terms of ordinary homology and cohomology.
\begin{align}
H_*^\relinfty(X_{\CC P^N}/\CC P^N)&=H_{*+2N}^\relinfty(X_{\CC P^N})\\
H^*_c(X_{\CC P^N}/\CC P^N)&=H^{*+2N}_c(X_{\CC P^N})
\end{align}
In terms of these descriptions, the transition maps in the inverse limit \eqref{coborelinvlim} and directed colimit \eqref{coboreldirlim} are the `wrong way maps'
\begin{align}
H_{*+2N}^\relinfty(X_{\CC P^N})&\leftarrow H_{*+2N+2}^\relinfty(X_{\CC P^{N+1}})\\
H^{*+2N}_c(X_{\CC P^N})&\to H^{*+2N+2}_c(X_{\CC P^{N+1}})
\end{align}
(`intersection with the hyperplane $\CC P^N\subseteq\CC P^{N+1}$') associated to the diagram \eqref{borelcompatible}.
There is a long exact sequence
\begin{equation}
\cdots\to H_{*+2N+2}^\relinfty(X_{\CC P^{N+1}})\to H_{*+2N}^\relinfty(X_{\CC P^N})\to H_{*+2N+1}^\relinfty(X)\to\cdots
\end{equation}
(and dually for $H^*_c$) which means the transition maps are isomorphisms once $*+2N+1>\dim X$.
This eventual constancy implies the inverse limit \eqref{coborelinvlim} is well behaved (so, for example, the long exact sequence of the pair exists and is exact for $H^{cS^1,\relinfty}$).
\end{definition}

\begin{remark}
Given a reasonable definition of the bivariant theory for maps such as $X_{\CC P^\infty}\to\CC P^\infty$, one should be able to eliminate the inverse limit and directed colimit above and simply write $H^{cS^1,\relinfty}_*(X)=H_*^\relinfty(X_{\CC P^\infty}/\CC P^\infty)$ and $H^*_{cS^1,c}(X)=H^*_c(X_{\CC P^\infty}/\CC P^\infty)$.
We choose to use the definition in terms of (co)limits because it is technically simpler (the main technical complication of infinite-dimensional bases such as $\CC P^\infty$ is that homology and upper shriek are no longer related in the same way, since this relation involves a shift by the dimension).
\end{remark}

It is evident that $H^{cS^1,\relinfty}_*(X)$ and $\smash{H_{cS^1,c}^*(X)}$ can be non-zero only for $*\leq\dim X$, but there is typically no lower bound on $*$ for which they are non-zero.
For example, $H^{cS^1}_*(\pt)=\invlim_NH_{*+2N}(\CC P^N)=\invlim_NH^{-*}(\CC P^N)=H^{-*}(\CC P^\infty)=\ZZ[t]$ (free polynomial algebra) where $t$ is the class of a hyperplane and lies in homological degree $-2$ (cohomological degree $2$).
Intersection of cycles gives $H^{cS^1\!}_*(\pt)$ the structure of a ring and gives each $H^{cS^1,\relinfty}_*(X)$ the structure of a module over it.

\begin{definition}[Tate $S^1$-equivariant homology]\label{tatedefn}
The Tate $S^1$-equivariant homology is the localization of co-Borel equivariant homology at $t\in H^{cS^1}_{-2}(\pt)=H^2(\CC P^\infty)$, namely it is the directed colimit
\begin{equation}
H^{tS^1,\relinfty}_*(X)=\dircolim_iH^{cS^1,\relinfty}_{*-2i}(X)
\end{equation}
where the transition maps are multiplication by $t$ (compare Greenlees--May \cite[Corollary 16.3]{greenleesmay}).
\end{definition}

The key property of Tate $S^1$-equivariant homology is that it vanishes for (almost) free $S^1$-spaces (with rational coefficients), hence by the long exact sequence and excision, depends rationally only on the fixed set.
This is known as the \emph{equivariant localization theorem}, which originates in the work of Smith \cite{smith,smithII,smithIII}, was reformulated cohomologically by Borel \cite{borel}, and was then formalized in its present form by Atiyah--Segal \cite{atiyahsegal,segal} and Quillen \cite{quillen}.

\begin{proposition}\label{equivariantlocalization}
The map $H^{tS^1,\relinfty}_*(X^{S^1})\to H^{tS^1,\relinfty}_*(X)$ is an isomorphism over $\QQ$.
\end{proposition}

\begin{proof}
We assume that our spaces have a reasonable $S^1$-equivariant cell decomposition (which holds in the cases we care about by real analyticity).
Precisely speaking, this means that $X$ is glued out of cells of the form $(S^1/\Gamma)\times(D^k,\partial D^k)$ for subgroups $\Gamma\subseteq S^1$, where $S^1$ acts by multiplication on the first factor (and trivially on the second factor).
Given such a cell decomposition of $X$, to show that $H^{tS^1,\relinfty}_*(X,X^{S^1})=0$, it suffices (by the long exact sequence and excision) to show that $H^{tS^1}_*((S^1/\Gamma)\times(D^k,\partial D^k))=0$ for $\Gamma\subsetneqq S^1$ a \emph{proper} subgroup.
We have $H^{tS^1\!}_*((S^1/\Gamma)\times(D^k,\partial D^k))=H^{tS^1\!}_{*-k}(S^1/\Gamma)$, so we are reduced to showing that $H^{tS^1\!}_*(S^1/\Gamma)=0$ for $\Gamma\subsetneqq S^1$.
Since $\Gamma$ is finite, there is a `transfer' map $H^{tS^1\!}_*(S^1/\Gamma)\to H^{tS^1\!}_*(S^1)$ whose composition with the pushforward map $H^{tS^1\!}_*(S^1)\to H^{tS^1\!}_*(S^1/\Gamma)$ is multiplication by $\#\Gamma$ on $H^{tS^1\!}_*(S^1/\Gamma)$.
It thus suffices to show that $H^{tS^1\!}_*(S^1)=0$, which follows from calulating $H^{cS^1\!}_*(S^1)=\ZZ$.
\end{proof}

The significance of equivariant localization is the following.
Given a class in $H_*^{cS^1,\relinfty}(X)$, we may push forward to $H_*^{cS^1\!}(\pt)$ provided $X$ is compact.
However, if we are satisfied with pushing forward to the Tate group $H_*^{tS^1\!}(\pt)$ (over the rationals), then equivariant localization provides such a pushforward map when just the fixed set $X^{S^1}$ is compact.
\begin{equation}\label{tatepushforwardbetter}
\begin{tikzcd}
H^{cS^1,\relinfty}_*(X^{S^1})\ar[r]\ar[d]\ar[dd,rounded corners,"X^{S^1}\text{ compact}"',to path={(\tikztostart)--([xshift=-10ex]\tikztostart.center)--([xshift=-10ex]\tikztotarget.center)\tikztonodes--(\tikztotarget)}]&H^{tS^1,\relinfty}_*(X^{S^1};\QQ)\ar[d,"\sim"]\ar[dd,rounded corners,"X^{S^1}\text{ compact}",to path={(\tikztostart)--([xshift=10ex]\tikztostart.center)--([xshift=10ex]\tikztotarget.center)\tikztonodes--(\tikztotarget)}]\\
H^{cS^1,\relinfty}_*(X)\ar[r]\ar[d,"X\text{ compact}"']&H^{tS^1,\relinfty}_*(X;\QQ)\ar[d,"X\text{ compact}"]\\
H^{cS^1}_*(\pt)\ar[r]&H^{tS^1}_*(\pt;\QQ)\\
\end{tikzcd}
\end{equation}

\subsection{Equivariant enumerative invariants}

We have already seen how a curve enumeration theory, namely a class in $H^\relinfty_*(\cZ/\Cpx_3)$, specializes to a `virtual fundamental class' in $H_*^\relinfty(\cZ(X))$ for any complex threefold $X$.
More generally, a curve enumeration theory determines a `relative' virtual fundamental class in $H_*^\relinfty(\cZ(X/B))$ for any families of complex threefolds $X\to B$ over a simplicial complex $B$.
For more general bases $B$, we may define the relative virtual fundamental class by choosing a concordance class of triangulations of $B$.

Now here we consider a complex threefold $X$ with an $S^1$-action, and we seek an \emph{equivariant virtual fundamental class} in $H_*^{cS^1,\relinfty}(\cZ(X))$.
By the definition of co-Borel equivariant homology, such a class is simply a collection of compatible relative virtual fundamental classes on the Borel construction of $\cZ(X)$ (equivariant enumerative invariants are, by definition, family invariants of the Borel construction).
\begin{align}
H_*^{cS^1,\relinfty}(\cZ(X))&=\invlim_NH_*^\relinfty(\cZ(X)_{\CC P^N}/\CC P^N)\\
&=\invlim_NH_*^\relinfty(\cZ(X_{\CC P^N}/\CC P^N)/\CC P^N)
\end{align}
Such a class is therefore provided by any curve enumeration theory (all real analytic triangulations of $\CC P^N$ are concordant).

Now a key feature of equivariant enumerative invariants is that they \emph{localize} to the fixed locus, after inverting the equivariant parameter (that is, in Tate equivariant homology, see Definition \ref{tatedefn}) and passing to rational coefficients.
That is, we may consider the maps $H_*^{cS^1,\relinfty}(\cZ(X))\to H_*^{tS^1,\relinfty}(\cZ(X);\QQ)\xleftarrow\sim H_*^{tS^1,\relinfty}(\cZ(X)^{S^1};\QQ)$, where the second map is an isomorphism by Proposition \ref{equivariantlocalization}, to obtain a \emph{localized} equivariant virtual fundamental class in $H_*^{tS^1,\relinfty}(\cZ(X)^{S^1};\QQ)$.

Now let us specialize to our case of interest, namely a local curve $E=E_{g,k}\to C_g$ equipped with its fiberwise scaling action by $\CC^\times$ (rather, the restriction of this action to the unit circle $S^1\subseteq\CC^\times$).
Any curve enumeration theory defines a localized equivariant virtual fundamental class in $H_*^{tS^1,\relinfty}(\cZ(E)^{S^1};\QQ)$.
Now $\cZ(E)^{S^1}$ is particularly simple: the only $S^1$-fixed cycles on a local curve are the multiples of the zero section $m[C]$ for $m\geq 0$.
Thus a localized equivariant virtual fundamental class for a local curve $E_{g,k}$ is simply a collection of classes in $H_*^{tS^1,\relinfty}(\pt)=\QQ[t,t^{-1}]$ indexed by integers $m\geq 0$.
Since the local curve depends only on $g$ and $k$ up to deformation, we obtain elements of $\QQ[t,t^{-1}]$ indexed by integers $g\geq 0$, $m\geq 0$, and $k\in\ZZ$ (associated to any choice of curve enumeration theory); these are, roughly speaking, `$S^1$-equivariant counts of curves of degree $m$ on $E_{g,k}$'.
These `curve counts' lie in $\QQ\cdot t^{-mk}$ for degree reasons (the virtual fundamental class lies in homological degree given by twice the chern number, which in this case equals $2mk$, and $t$ has homological degree $-2$).
In particular, the Gromov--Witten and Pandharipande--Thomas curve enumeration theories (see Section \ref{zgwptsummary}) give rise to invariants.
\begin{align}
\GW_{S^1}(E_{g,k},m)&\in\QQ((u))\cdot t^{-mk}\\
\PT_{S^1}(E_{g,k},m)&\in\QQ((q))\cdot t^{-mk}
\end{align}
for integers $g\geq 0$, $m\geq 0$, and $k\in\ZZ$.

\begin{theorem}[\cite{bryanpandharipande,okounkovpandharipande}]\label{boplocal}
The power series $\GW_{S^1}(E_{g,k},m)$ and $\PT_{S^1}(E_{g,k},m)$ satisfy the MNOP correspondence with weight $k$.
\end{theorem}

We explain the citation: Bryan--Pandharipande \cite{bryanpandharipande} compute the $S^1$-equivariant Gromov--Witten invariants of $E_{g,k}$, while Okounkov--Pandharipande \cite{okounkovpandharipande} compute the $S^1$-equivariant Donaldson--Thomas invariants of $E_{g,k}$, and it is explained in \cite[Section 5]{mptmodular} how to walk through the arguments of \cite{okounkovpandharipande} to see that they apply equally well to Pandharipande--Thomas invariants.

\subsection{Local curve elements of the Grothendieck group}\label{localcurveeltsection}

We now show how to express, explicitly, the $S^1$-localized equivariant enumerative invariants of local curves $E_{g,m,k}\to C_g$ (defined just above) as the (non-equivariant) enumerative invariants of certain elements $x_{g,m,k}\in H^*_c(\cZ/\Cpx_3;\QQ)$ of virtual dimension zero which we call \emph{equivariant local curve elements}.
This amounts to unwinding the definition of equivariant invariants in terms of the Borel construction and, more significantly, taking a geometric perspective on equivariant localization.

Given a local curve $E\to C$, we consider maps of the form
\begin{equation}\label{equivariantcutdownmap}
\cZ(E,m)\xrightarrow{\cap E_p}\Sym^mE_p\xrightarrow{\Sym^m\lambda}\Sym^m\CC\xrightarrow{\beta_r}\CC
\end{equation}
for a point $p\in C$, a linear map $\lambda:E_p\to\CC$, and a homogeneous symmetric polynomial $\beta_r:\Sym^m\CC\to\CC$ of degree $r\geq 1$.
This map is $\CC^\times$-equivariant for the weight $r$ action on the target $\CC$.
It thus determines a section $f$ of $\cL^{\otimes r}$ over $\cZ(E,m)_{\CC P^N}$ for all $N<\infty$, where $\cL$ denotes (the pullback of) the tautological line bundle on $\CC P^N$.

\begin{definition}[Equivariant local curve elements]\label{loccurveeltsdef}
Fix $g\geq 0$, $m\geq 0$, and $k\in\ZZ$.

For any tuple of sections $f_1,\ldots,f_n$ (as defined just above) of weights $r_1,\ldots,r_n>0$, whose joint zero set $f_1^{-1}(0)\cap\cdots\cap f_n^{-1}(0)\subseteq\cZ(E,m)$ is compact, and any $N<\infty$, we consider the expression
\begin{multline}\label{localcurveexpression}
_{N;f_1,\ldots,f_n}x_{g,m,k}=\prod_{i=1}^nr_i^{-1}f_i^*c_1(\cL^{\otimes r_i})\\\in H^{2n}_c(\cZ(E,m)_{\CC P^N};\QQ)=H^{2n-2N}_c(\cZ(E,m)_{\CC P^N}/\CC P^N;\QQ)\\\to H_c^{2n-2N}(\cZ/\Cpx_3;\QQ)
\end{multline}
where $\cZ(E,m)\subseteq\cZ(E)$ denotes the cycles of degree $m$.
This is a lift of the $n$th power of the hyperplane class to compactly supported cohomology.

Now the pushforward map $H^*_c(\cZ(E)_{\CC P^N}/\CC P^N)\to H^*_c(\cZ(E)_{\CC P^{N+1}}/\CC P^{N+1})$ associated to the diagram \eqref{borelcompatible} sends $_{N;f_1,\ldots,f_n}x_{g,m,k}$ to $_{N+1;f_1,\ldots,f_n,f_{n+1}}x_{g,m,k}$ for \emph{any} choice of $f_{n+1}$ (compare the discussion in Definition \ref{coboreldef}).
It follows that the image of $_{N;f_1,\ldots,f_n}x_{g,m,k}$ in $H^{2n-2N}_c(\cZ/\Cpx_3;\QQ)$ depends only on $(g,m,k)$ and the difference $n-N$.
We denote this image by $_\ell x_{g,m,k}\in H^*_c(\cZ/\Cpx_3;\QQ)$ for $\ell=n-N-mk$.
Note that these $_\ell x_{g,m,k}$ are indeed defined for all $g,m,k,\ell$: by taking $N$ sufficiently large, we may ensure that there exist $n=\ell+mk+N$ sections $f_1,\ldots,f_n$ whose joint zero set is compact.

This class $_\ell x_{g,m,k}\in H^*_c(\cZ/\Cpx_3;\QQ)$ lies in cohomological degree $2mk+2\ell$ and has chern number $2mk$, hence has virtual dimension $-2\ell$.
When $k\geq 0$, the local curve $E$ is semi-Fano, so the elements $_\ell x_{g,m,k}$ naturally lift to $H^*_c(\cZ_\semiFano/\Cpx_3;\QQ)$ for $k\geq 0$.
The elements $_\ell x_{g,m,k}$ also naturally lift to $H^*_c(\cZ_{\leq(m)}/\Cpx_3)$ (and to $H^*_c(\cZ_{\semiFano,\leq(m)}/\Cpx_3;\QQ)$ when $k\geq 0$) since all points of $\cZ(E,m)$ have multiplicity $\leq(m)$.

The \emph{equivariant local curve element} is the value $x_{g,m,k}={}_0x_{g,m,k}$ at $\ell=0$.
\end{definition}

The reader may wonder why we bother defining $_\ell x_{g,m,k}$ for general $\ell$, rather than restricting to $\ell=0$ from the beginning.
One answer is that the formula for the coproduct of $x_{g,m,k}$ (Lemma \ref{coproductlocalcurve}) involves nonzero $\ell$.

\begin{proposition}\label{explicitlocalizedpushforward}
For any curve enumeration theory (class in $H^\relinfty_*(\cZ/\Cpx_3)$ of virtual codimension zero), its induced $S^1$-equivariant curve count of degree $m$ cycles in a local curve $E_{g,k}$ is given by its evaluation on $x_{g,m,k}\in H^*_c(\cZ/\Cpx_3;\QQ)$ times $t^{-mk}$.
\end{proposition}

\begin{proof}
This amounts to unwinding definitions.

Recall that the $S^1$-equivariant count of degree $m$ curves in a local curve $E=E_{g,k}$ is defined as follows.
We take the equivariant virtual fundamental class in $H_*^{cS^1,\relinfty}(\cZ(E,m))$, we multiply it by a sufficiently high power of the hyperplane class $t^n$ and rationalize so that it lifts to $H_*^{cS^1}(\cZ(E,m);\QQ)$ (as guaranteed by Proposition \ref{equivariantlocalization}), we push forward to a point to obtain a class in $H_*^{cS^1}(\pt;\QQ)=\QQ[t]$, and we multiply by $t^{-n}$.

We may realize this procedure concretely as follows.
Realize the relevant co-Borel equivariant homology groups via the inverse limit \eqref{coborelinvlim}.
To lift $t^n$ times a class in $H_*^{cS^1,\relinfty}(\cZ(E,m))=\invlim_NH_{*+2N}^\relinfty(\cZ(E,m)_{\CC P^N})$ to $H_*^{cS^1}(\cZ(E,m))=\invlim_NH_{*+2N}(\cZ(E,m)_{\CC P^N})$ (rationally), we may fix some $f_1,\ldots,f_n$ with compact common zero locus as in Definition \ref{loccurveeltsdef} and multiply our class by the expression $_{N;f_1,\ldots,f_n}x_{g,m,k}\in H^{2n}_c(\cZ(E,m)_{\CC P^N};\QQ)$ \eqref{localcurveexpression} for every $N$.
Now we push forward to $H_*^{cS^1}(\pt)=\invlim_NH_{*+2N}(\CC P^N)$ and ask for the coefficient in front of $t^{n-mk}$.
For $N=n-mk$, the coefficient in front of $t^{n-mk}=t^N$ is the component lying in $H_0(\CC P^N)$, which is now just the evaluation of the equivariant virtual fundamental class on $_{N;f_1,\ldots,f_n}x_{g,m,k}$, which is $x_{g,m,k}$ by our choice of $N$.
\end{proof}

\begin{corollary}\label{mnoplocalgrothendieck}
The power series $\GW(x_{g,m,k})$ and $\PT(x_{g,m,k})$ satisfy the MNOP correspondence with weight $k$.
\end{corollary}

\begin{proof}
Combine Theorem \ref{boplocal} with Proposition \ref{explicitlocalizedpushforward}.
\end{proof}

Now let us establish a few elementary properties of the local curve elements.

\begin{lemma}
We have $x_{g,0,k}=1$ and $_\ell x_{g,0,k}=0$ for $\ell>0$.
\end{lemma}

\begin{proof}
In the definition of $_\ell x_{g,m,k}$, consider taking $N=0$, which means $n=mk+\ell$.
We are interested in the case that $m=0$ and $\ell\geq 0$, so $n\geq 0$ as required.
Since $m=0$, the space of cycles $\cZ(E_{g,k},m)$ of degree $m$ is a single point (the empty cycle) hence compact, so the requirement in Definition \ref{loccurveeltsdef} that $f_1,\ldots,f_n$ have compact joint zero set is vacuous.
Thus $x_{g,m,k}$ is represented by the class of the hyperplane to the power $\ell\geq 0$ in $H^{2\ell}(*;\QQ)$.
This evidently vanishes for $\ell>0$.
For $\ell=0$, it is the constant function $1$ on the empty cycle, which is the definition of the unit $1\in H^*_c(\cZ/\Cpx_3)$.
\end{proof}

\begin{lemma}\label{eqloclvanishing}
We have $_\ell x_{g,m,k}=0$ for $\ell$ sufficiently large (as a function of $(g,m,k)$).
\end{lemma}

\begin{proof}
We give two related, but not quite equivalent, arguments.

First, note that $_{N;f_1,\ldots,f_n}x_{g,m,k}\in H^{2mk+2\ell}_c(\cZ(E,m)_{\CC P^N}/\CC P^N;\QQ)$ which vanishes for degree reasons once $2mk+2\ell>2\dim_\CC\cZ(E,m)$.
Thus we have $_\ell x_{g,m,k}=0$ for $\ell>\dim_\CC\cZ(E_{g,k},m)-mk$.

Second, suppose there exist $f_1,\ldots,f_n$ on $\cZ(E_{g,k},m)$ with compact joint zero set.
Taking $\ell=n-mk$ gives $N=0$, so $_{0;f_1,\ldots,f_n}x_{g,m,k}\in H^{2n}_c(\cZ(E,m)_{\CC P^0};\QQ)$.
Now adding a single additional $f_{n+1}$ (that is, increasing $\ell$) multiplies this class by the hyperplane class, which lives in $H^2(\CC P^0)=0$, thus yielding zero.
We conclude that $_\ell x_{g,m,k}=0$ for $\ell>n_{\min}(g,m,k)-mk$, where $n_{\min}(g,m,k)$ is the minimum number of sections $f_1,\ldots,f_n$ on $\cZ(E_{g,k},m)$ with compact joint zero set.
\end{proof}

\begin{lemma}\label{rholocalcurve}
We have $\rho_d({}_\ell x_{g,m,k})={}_{\ell+(1-1/d)mk}x_{g,m/d,k}$ if $d|m$ and is zero otherwise.
\end{lemma}

\begin{proof}
We may calculate $\rho_d({}_\ell x_{g,m,k})$ by pulling back the cohomology class $_{N;f_1,\ldots,f_n}x_{g,m,k}\in H^{2n}_c(\cZ(E,m)_{\CC P^N};\QQ)$ \eqref{localcurveexpression} under the `multiply by $d$' map $\cZ(E)\to\cZ(E)$ (which is $\CC^\times$-equivariant, hence induces a map on the Borel construction).
If $m$ is not divisible by $d$, then the image of this map $d$ is disjoint from $\cZ(E,m)$, hence the pullback is zero.
If $m$ is divisible by $d$, then the pullback of a weight $r$ map $f:\cZ(E,m)\to\CC^\times$ under $d:\cZ(E,m/d)\to\cZ(E,m)$ is a weight $r$ map $\cZ(E,m/d)\to\CC^\times$.
Thus the pullback of $_{N;f_1,\ldots,f_n}x_{g,m,k}$ is $_{N;f_1,\ldots,f_n}x_{g,m/d,k}$, which maps to $_{\ell'}x_{g,m/d,k}$ where $\ell'=n-N-mk/d$.
Since $\ell=n-N-mk$, we have $\ell'=\ell+(1-1/d)mk$.
\end{proof}

\begin{lemma}\label{coproductlocalcurve}
The coproduct $\Delta$ (Definition \ref{coproductdef}) is given on equivariant local curve elements by $\Delta({}_\ell x_{g,m,k})=\sum_{m=m'+m''}\sum_{\ell=\ell'+\ell'}{}_{\ell'}x_{g,m',k}\otimes{}_{\ell''}x_{g,m'',k}$ (this sum being finite by Lemma \ref{eqloclvanishing}).
\end{lemma}

\begin{proof}
Realize $_\ell x_{g,m,k}$ by an expression $_{N;f_1,\ldots,f_n}x_{g,m,k}$ \eqref{localcurveexpression} for some collection of sections $f_1,\ldots,f_n$ whose joint zero set is compact.
The coproduct $\Delta({}_{N;f_1,\ldots,f_n}x_{g,m,k})$ (Definition \ref{coproductdef}) is defined via the disjoint union family
\begin{equation}\label{disjointunionoflocalcurves}
(E_{\CC P^N}\times\CC P^N)\sqcup(\CC P^N\times E_{\CC P^N})\to\CC P^N\times\CC P^N,
\end{equation}
whose relative cycle space is the product of relative cycle spaces
\begin{equation}\label{erelcycleprod}
\cZ(E)_{\CC P^N}\times\cZ(E)_{\CC P^N}=\cZ(E_{\CC P^N}/\CC P^N)\times\cZ(E_{\CC P^N}/\CC P^N)\to\CC P^N\times\CC P^N.
\end{equation}
Over the diagonal $\Delta(\CC P^N)\subseteq\CC P^N\times\CC P^N$, there is an addition map $\Sigma$ from this relative cycle space to the relative cycle space of $E_{\CC P^N}\to\CC P^N$.
The coproduct $\Delta({}_{N;f_1,\ldots,f_n}x_{g,m,k})$ is represented by the disjoint union family \eqref{disjointunionoflocalcurves} equipped with the compactly supported cohomology class $\Delta_!\Sigma^*({}_{N;f_1,\ldots,f_n}x_{g,m,k})$ on $\cZ(E)_{\CC P^N}\times\cZ(E)_{\CC P^N}$, which lies in degree $2N+2n$.
This class is certainly supported inside $\bigsqcup_{m=m'+m''}\cZ(E,m')_{\CC P^N}\times\cZ(E,m'')_{\CC P^N}$, so it suffices to fix a decomposition $m=m'+m''$ and show that the restriction of $\Delta_!\Sigma^*({}_{N;f_1,\ldots,f_n}x_{g,m,k})$ to $\cZ(E,m')_{\CC P^N}\times\cZ(E,m'')_{\CC P^N}$ represents $\sum_{\ell=\ell'+\ell'}{}_{\ell'}x_{g,m',k}\otimes{}_{\ell''}x_{g,m'',k}$.

To compute $\Delta_!\Sigma^*({}_{N;f_1,\ldots,f_n}x_{g,m,k})$ restricted to $\cZ(E,m')_{\CC P^N}\times\cZ(E,m'')_{\CC P^N}$, we consider the more general expression
\begin{multline}\label{moregeneralexpressioncoproductlocalcurve}
\Delta_!\Sigma^*({}_{N;f_1,\ldots,f_n}x_{g,m,k})\cup p_1^*({}_{N;f_1',\ldots,f_{n'}'}x_{g,m',k})\cup p_2^*({}_{N;f_1'',\ldots,f_{n''}''}x_{g,m'',k})\\\in H^{2N+2n+2n'+2n''}_c(\cZ(E)_{\CC P^N}\times\cZ(E)_{\CC P^N})
\end{multline}
where $p_i$ denotes the projection to the $i$th factor.
The pushforward of this compactly supported cohomology class to $H^{2(N+1)+2n+2n'+2n''+2}_c(\cZ(E)_{\CC P^{N+1}}\times\cZ(E)_{\CC P^{N+1}})$ represents the same element of the Grothendieck group of 1-cycles, and may be described in multiple ways.
Namely, this pushforward has the same form as \eqref{moregeneralexpressioncoproductlocalcurve}, and is obtained by incrementing $N$ and adding exactly one $f$, $f'$, or $f''$ (this is because a hyperplane in the diagonal $\CC P^{N+1}\subseteq\CC P^{N+1}\times\CC P^{N+1}$ may be described as the pullback of a hyperplane in either factor, or symbolically because $(\Delta_!\alpha)\cup\beta=\Delta_!(\alpha\cup\Delta^*\beta)$ and $\Delta^*p_i^*H=H$).
Now the expression \eqref{moregeneralexpressioncoproductlocalcurve} is a valid compactly supported cohomology class as long as either $f_1,\ldots,f_n$ have compact joint zero set or both $f_1',\ldots,f_{n'}'$ and $f_1'',\ldots,f_{n''}''$ have compact joint zero set.
We may thus add functions $f'$ and $f''$ and then remove functions $f$ to conclude that $\Delta_!\Sigma^*({}_{N;f_1,\ldots,f_n}x_{g,m,k,\ell})$ represents the same element in the Grothendieck group of 1-cycles as
\begin{multline}
[\Delta]\cup p_1^*({}_{\bar N;f_1',\ldots,f_{n'}'}x_{g,m',k})\cup p_2^*({}_{\bar N;f_1'',\ldots,f_{n''}''}x_{g,m'',k})\\\in H^{2\bar N+2n'+2n''}_c(\cZ(E)_{\CC P^{\bar N}}\times\cZ(E)_{\CC P^{\bar N}})
\end{multline}
where $\bar N-N=n'+n''-n$.
Now expand $[\Delta]=\sum_{\begin{smallmatrix}a'+a''=\bar N\\a',a''\geq 0\hfill\end{smallmatrix}}H^{a'}\otimes H^{a''}$, and note that we can achieve multiplication by $H^{a'}\otimes H^{a''}$ by adding $a'$ functions $f'$ and adding $a''$ functions $f''$.
We therefore obtain
\begin{multline}
\smash{\sum_{\begin{smallmatrix}a'+a''=\bar N\\a',a''\geq 0\hfill\end{smallmatrix}}}p_1^*({}_{\bar N,f_1',\ldots,f_{n'+a'}'}x_{g,m',k})\cup p_2^*({}_{\bar N;f_1'',\ldots,f_{n''+a''}''}x_{g,m'',k})\\\in H^{2\bar N+2n'+2n''}_c(\cZ(E)_{\CC P^{\bar N}}\times\cZ(E)_{\CC P^{\bar N}}),
\end{multline}
which in $H^*_c(\cZ/\Cpx_3)$ is the sum of $_{\ell'}x_{g,m',k}\otimes{}_{\ell''}x_{g,m'',k}$ in the Grothendieck group over some set of pairs $(\ell',\ell'')=(n'+a'-\bar N-m'k,n''+a''-\bar N-m''k)$.
To compute the relevant pairs $(\ell',\ell'')$, note that $\ell'+\ell''=(n'+n'')+(a'+a'')-2\bar N-k(m'+m'')=(n+\bar N-N)+\bar N-2\bar N-mk=n-N-mk=\ell$.
Now we also have the inequalities $a',a''\geq 0$, which translate to $\ell'\geq n'-\bar N-m'k=n-n''-N-m'k$ (and similarly for $\ell''$).
In this lower bound on $\ell'$, the quantities $n$, $N$, and $m'k$ are fixed, and we may choose $n''$ as large as we like.
Thus by taking $n',n''\to\infty$ we obtain the desired sum of $_{\ell'}x_{g,m',k}\otimes{}_{\ell''}x_{g,m'',k}$ over all decompositions $\ell=\ell'+\ell''$ (noting that all but finitely many of these terms vanish by Lemma \ref{eqloclvanishing}).
\end{proof}

\begin{lemma}\label{eqlocalcurveassocgraded}
Suppose $m\geq 1$ and $k\geq 0$, and consider the map $H^*_c(\cZ_{\semiFano,\leq(m)}/\Cpx_3;\QQ)\to H^*_c(\cZ_{\semiFano,=(m)}/\Cpx_3;\QQ)$.
This map annihilates $x_{g,m,k}$ when $(m-1)k>0$, and when $(m-1)k=0$ it sends $x_{g,m,k}$ to the Poincaré dual of the point $0\in H^0(C,E)=\cZ(E,m)_{\semiFano,=(m)}$ for any local curve $E\to C$ of type $(g,m,k)$ with $H^1(C,E)=0$ (which exists).
\end{lemma}

\begin{proof}
We can take $E$ to be the direct sum of a pair of generic line bundles of degrees $g-1$ and $g-1+k$.
Now let us realize the definition of $x_{g,m,k}$ in \eqref{localcurveexpression} explicitly.
The action of $\CC^\times$ on $\cZ(E,m)_{=(m)}=H^0(C,E)$ is the standard multiplication action from its structure as a complex vector space.
Thus $(\cZ(E,m)_{=(m)})_{\CC P^N}\to\CC P^N$ is the total space of $H^0(C,E)\otimes\cL$, and elements of $H^0(C,E)^*$ give allowable functions $f:(\cZ(E,m)_{=(m)})_{\CC P^N}\to\cL$.
Now let $f\in H^0(C,E)^*$ run over a basis (note $h^0(C,E)=k$) and note that the product of pullbacks $f^*c_1(\cL)$ is the Poincaré dual of the zero section inside the total space of $(\cZ(E,m)_{=(m)})_{\CC P^N}\to\CC P^N$.
Now $x_{g,m,k}$ is defined by taking a product of $n=mk+N$ total classes $f^*c_1(\cL)/r$, so there are $N+(m-1)k$ more to go.
The product we have so far is already compactly supported, so we can just multiply by the hyperplane class to the power $N+(m-1)k$.
Multiplying by $N$ copies of the hyperplane yields the Poincaré dual of a single point.
If $(m-1)k=0$, this is the desired answer.
If $(m-1)k>0$, then we need to multiply by more hyperplanes, which yields zero.
\end{proof}

\section{Generic transversality}\label{gentransec}

It is a standard result that for generic \emph{almost} complex structures on a target manifold, all simple pseudo-holomorphic maps from closed Riemann surfaces are unobstructed (see \cite{gromov,mcduffsalamon,wendl} for precise statements).
We now derive analogous results for complex structures (see Lemmas \ref{dregularsingle}, \ref{generictransversality}, and \ref{generictransversalityfamily}).
Since complex structures are much more rigid (for example, they have no nontrivial perturbations supported inside a small ball), these results are necessarily weaker than those in the almost complex setting: they only apply to a small neighborhood of a given compact 1-cycle.

We then introduce `enough divisors' (Proposition \ref{divisorcovering}) and the `graph trick' (Lemma \ref{graphtrick}), which allow us to bring generic transversality to bear on curve enumeration problems.
We use these to show that Grothendieck groups of 1-cycles are unchanged by imposing a transversality condition on the 1-cycles in question (Theorem \ref{grothendiecktransverse}).

\subsection{Regularity}\label{secregularity}

The deformation theory of a map $u:C\to X$ from a smooth proper curve $C$ to a smooth complex analytic manifold $X$ is controlled by $H^*(C,u^*TX)$.
The deformation theory of $C$ itself is controlled by $H^*(C,TC[1])$.
The deformation theory of the pair $(C,u)$ is controlled by $H^*(C,[TC[1]\to u^*TX])$ (the map $TC\to u^*TX$ being $du$).
A deformation problem (in any flavor) is said to be \emph{unobstructed} when $H^{\geq 1}=0$.

Given a finite set of points $S\subseteq C$, the deformation theory of the pair $(C,S)$ is controlled by $H^*(C,TC(-S)[1])$, and the deformation theory of a triple $(C,S,u)$ is controlled by $H^*(C,[TC(-S)[1]\to u^*TX])$.
We may also track of the behavior of the map $u$ at the marked points $S$.
If $du$ vanishes to order $k$ at $p\in S$, then there is a map from the deformation complex of $(C,S,u)$ to $T_{u(p)}X\otimes[\cO_C(-p)/\cO_C(-(k+1)p)]$ measuring the first order variation in the derivatives of order $<k$ of $du$ at $p$.
The cocone of this map then controls the deformation theory of $(C,S,u)$ subject to the condition that $du$ vanishes to order $k$ at $p$.
Similarly, if $u$ maps a subset $S_0\subseteq S$ to a single point of the target $X$, then there is an induced map from the deformation complex of $(C,S,u)$ to $(T_{u(S_0)}X)^{S_0}/T_{u(S_0)}X$, and the cocone of this map controls the deformation theory of $(C,S,u)$ subject to the condition that all points of $S_0$ should map to the same point of $X$.

Given a submersion of complex manifolds $X\to B$, we may also consider the deformation theory of pairs $(b,u:C\to X_b)$, which is an extension of $T_bB$ and the deformation theory of $u$.
Note that this differs from the deformation theory of maps from $C$ to the total space $X$ (which we will never consider).
If $X'\to B'$ is a pullback of a submersion $X\to B$, then a pair $(b',u':C\to X_{b'}')$ in $X'\to B'$ which is unobstructed remains unobstructed when pushed forward to $X\to B$ (and if $B'\to B$ is submersive then conversely if the pushforward is unobstructed then the original pair is unobstructed).

\begin{definition}[Regular]\label{regulardef}
Let $u:C\to X$ be a holomorphic map from a compact smooth curve $C$.
A point $x\in C$ will be called \emph{special} (for $u$) when $du(x)=0$ or $\#u^{-1}(u(x))>1$.
The set $S\subseteq C$ of special points is finite provided $u$ is simple (nowhere a multiple cover), which we now assume.
We now consider the deformation theory of the triple $(C,S,u)$ subject to the constraint that the points $S$ remain special with the same discrete data, meaning that all conditions $u(x)=u(x')$ and $(D^ru)(x)=0$ which hold for $u$ are preserved.
We say that the map $u$ is \emph{regular} when this deformation problem is unobstructed.

A dimension count shows that the addition of the points $S$ and their constraints (`remaining special with the same discrete data') to the deformation problem of $(C,u)$ adds to the (complex) index the quantity
\begin{equation}
\abs S-\dim_\CC X\cdot\Bigl(\abs S-\abs{u(S)}+\sum_{p\in S}\ord_p(du)\Bigr).
\end{equation}
When $\dim_\CC X\geq 3$, this quantity is $<0$ unless $S=\varnothing$.

Regularity is also defined for curves in fibers of a family $X\to B$, meaning the deformation problem includes variation in the base parameter.
If $X'\to B'$ is a pullback of $X\to B$, then regularity in $X'\to B'$ implies regularity of the pushforward to $X\to B$ (and conversely when $B'\to B$ is a submersion).
\end{definition}

Regularity is quite a strong condition on a map $u:C\to X$ (probably stronger than necessary).
Regularity of $u$, that is unobstructedness of the constrained deformation problem in Definition \ref{regulardef}, implies unobstructedness of any deformation problem for $(C,u)$ with fewer constraints (in particular, with no constraints), by inspecting the relevant exact sequence relating the two.
The most serious use of the strong unobstructedness properties encoded by the notion of regularity will come in Proposition \ref{vdimpositivebetter}, where we need to know that regular maps deform to smooth embeddings.

\subsection{Deforming by regluing near a divisor}

We will describe complex structures and families thereof by gluing.
To this end, let us introduce some notation.
For complex manifolds $U$ and $V$, denote by $\An(U,V)$ the space of analytic maps $U\to V$ with relatively compact image.
If $V$ admits an open embedding into some $\CC^n$ (which will always be the case for us), then $\An(U,V)$ is a complex analytic Banach manifold, locally modelled on the space of $n$-tuples of bounded holomorphic functions on $U$.
Given a complex manifold $U$, let $\cR(U)=\An(U^-,U)$ (the space of `regluings'), where $U^-\subseteq U$ denotes a(n unspecified) large relatively compact open subset.
More formally, $\cR(U)$ is a pro-object, namely the inverse system of all neighborhoods of the identity $1_U\in\An(U^-,U)$ over all relatively compact open sets $U^-\subseteq U$.
In all cases of interest to us, $U$ will admit an open embedding into some $\CC^n$, implying that $\cR(U)$ is a (pro) complex analytic Banach manifold.

\begin{definition}\label{generaldeform}
Given a complex manifold $X$ with an open cover $X=A\cup B$, we may deform $X$ by modifying the identification between open sets $A\supseteq A\cap B\subseteq B$.
More formally, we consider the family $\tilde X\to\cR(A\cap B)$ defined by taking the trivial families $A$ and $B$ over $\cR(A\cap B)$ and gluing via the base parameter $A\times\cR(A\cap B)\ni(a,\gamma)\sim(\gamma(a),\gamma)\in B\times\cR(A\cap B)$.
To make this construction precise, and to ensure the result is Hausdorff, we may fix relatively compact sets $A^-\subseteq A$ and $B^-\subseteq B$ and glue $(A^-\sqcup B^-)\times\cR(A\cap B)$ to obtain a proper map $\tilde X^-\to\cR(A\cap B)$.
\end{definition}

We will in fact only need a special case of the above construction, namely when the regluing takes place in a small neighborhood of a \emph{divisor} (a closed complex submanifold of codimension one).

\begin{definition}[Deforming complex structure near a divisor]\label{divisordeform}
Let $X$ be a complex manifold, and let $D\subseteq X$ be a smooth divisor.
Regarding $X$ as the gluing of $X\setminus D$ and $\Nbd D$ over their common intersection, Definition \ref{generaldeform} provides a family $\tilde X\to\cR(\Nbd D\setminus D)$.
This family is smoothly trivial (analytic perturbations of the identity map on $\Nbd D\setminus D$ extend smoothly to $\Nbd D$), so a choice of smooth trivialization determines a family of complex structures on $X$ parameterized by $\cR(\Nbd D\setminus D)$.
We will also denote this base space by $\cJ_D(X)$ (complex structures on $X$ obtained by regluing near $D$).
Of course, this isn't really a space but rather a family of spaces depending on choices of neighborhoods, etc.
Sometimes we will need to fix a specific one, but we will do this at the relevant time.

The same construction applies to families $X\to B$ of complex manifolds.
Given a \emph{relative divisor} $D\subseteq X\to B$, meaning a divisor inside the total space which is submersive over $B$, we may consider the set $\cJ_D(X/B)=\cR_B(\Nbd D\setminus D)=\bigcup_{b\in B}\cR(\Nbd D_b\setminus D_b)\to B$, a holomorphic section $\alpha$ of which determines a `vertical' (i.e.\ over $B$) regluing $X_\alpha\to B$ of $X$.
\end{definition}

The tangent space to $\cR(\Nbd D\setminus D)$ at the identity is the space of germs of holomorphic vector fields on $\Nbd D$ possibly singular along $D$.
We denote this space by $H^0(D,TX(\infty D))$ (implicitly restricting the sheaf of holomorphic sections of $TX$ over $X$ to the divisor $D$).
Such a vector field thus gives a first order deformation of the complex structure on $X$ modulo gauge, that is an element of $H^1(X,TX)$.
Concretely, this map $H^0(D,TX(\infty D))\to H^1(X,TX)$ sends a holomorphic vector field $v$ to (the Dolbeault cohomology class represented by) $\bar\partial((1-\varphi)\cdot v)$ for a smooth function $\varphi:X\to[0,1]$ supported inside an open set $U\subseteq X$ containing $D$ such that $v$ is defined on $U\setminus D$ and $\varphi\equiv 1$ in a neighborhood of $D$.
Note that the choice of $\varphi$ evidently does not matter since $\bar\partial((1-\varphi)\cdot v)-\bar\partial((1-\varphi')\cdot v)=\bar\partial((\varphi'-\varphi)\cdot v)$ is exact in the Dolbeault complex since $(\varphi'-\varphi)\cdot v$ is a smooth vector field on $X$ (in contrast to $\varphi\cdot v$, which has singularities along $D$, or $(1-\varphi)\cdot v$, which is defined only on $U$).
In terms of distributions, the map $H^0(D,TX(\infty D))\to H^1(X,TX)$ is simply $v\mapsto\bar\partial v$, where $\bar\partial v$ is meant in the distributional sense and is supported on $D$ since $v$ is otherwise holomorphic (indeed, $\bar\partial v-\bar\partial((1-\varphi)\cdot v)=\bar\partial(\varphi\cdot v)$ is exact in the distributional Dolbeault complex).

\subsection{Effect of divisorial deformations on curves}

We now identify the sort of maps which can be made transverse by deforming the complex structure near a divisor.

\begin{definition}[$D$-controlled]
Let $D\subseteq X$ be a divisor.
A map $u:C\to X$ from a smooth proper curve $C$ will be called \emph{$D$-controlled} when $u^{-1}(D)\subseteq C$ is discrete and intersects every component of $C$.
A cycle $z=\sum_im_iC_i$ in $X$ will be called \emph{$D$-controlled} when $\bigsqcup_i\tilde C_i\to X$ is $D$-controlled.
\end{definition}

It is elementary to observe that being $D$-controlled is an open condition on maps $u:C\to X$.
Now let us argue the same is true for cycles:

\begin{lemma}\label{controlledopen}
The set of $D$-controlled cycles in $\cZ(X/B)$ is open for any relative divisor $D\subseteq X\to B$.
\end{lemma}

\begin{proof}
Suppose $z=\sum_im_iC_i\in\cZ(X/B)$ is $D$-controlled.
Since $C_i$ intersects $D$ geometrically, the algebraic intersection number $C_i\cdot D$ is positive.
If $z'=\sum_im_i'C_i'$ is close to $z$, then every $C_i'$ is homologous to a positive linear combination of some $C_i$'s, hence also has positive algebraic intersection with $D$, thus \emph{a fortiori} intersects it geometrically.
\end{proof}

For any family $X\to B$, the deformation complex of a map $u:C\to X_b$ maps to the deformation complex of the pair $(b,u:C\to X_b)$, with cokernel $T_bB$.
This induces a map from $T_bB$ to the obstruction space of the map $u$, whose cokernel is the obstruction space of the pair $(b,u)$.
Explicitly, this map is the Kodaira--Spencer map $T_bB\to H^1(X_b,TX_b)$ followed by restriction (e.g.\ of Dolbeault representatives) from $H^1(X_b,TX_b)$ to $H^1(C,TX_b)$.

We now come to the key technical result underlying generic transversality, which says that the space of first order deformations (of a complex structure) associated to a divisor $D$ by Definition \ref{divisordeform} surjects onto the obstruction space of any $D$-controlled simple map $u$ (via the map defined in the paragraph just above).

\begin{lemma}[Enough first order deformations]\label{dregularsingle}
Let $u:C\to X$ be a simple map from a smooth proper curve $C$ to a complex manifold $X$, and let $D\subseteq X$ be a divisor.
If $u$ is $D$-controlled, then the map
\begin{equation}\label{dregularsingleeqn}
H^0(\Nbd(D\cap u(C)),TX(\infty D))\to H^1(C,TX)
\end{equation}
is surjective for every sufficiently small neighborhood of $D\cap u(C)\subseteq D$ inside $X$.
In fact, it is surjective onto the obstruction space for the problem of deforming the map $u:C\to X$ subject to any finite number of point constraints (such as those appearing in the notion of `regularity' in Definition \ref{regulardef}).
\end{lemma}

\begin{proof}
Recall from above that the map in question sends a vector field $v$ to $\bar\partial((1-\varphi)\cdot v)$ (note that the `primitive' $(1-\varphi)\cdot v$ is not defined globally on $X$, so does not trivialize this element in cohomology), for any choice of cutoff function $\varphi$ (supported near $D$ and identically equal to $1$ in a neighborhood of it).
We will take $\varphi$ to be (a smoothing of) the characteristic function of a small tubular $\Nbd(D)$ of $D$.
Fix a local projection $\pi:X\to\CC$ (defined near $u(C)\cap D$) with $D=\pi^{-1}(0)$, and let $\Nbd(D)=\pi^{-1}(\DD_\varepsilon)$ be the inverse image of the $\varepsilon$-disk $\DD_\varepsilon\subseteq\CC$.
Now $\pi\circ u:C\to\CC$ is a ramified cover near the origin, so the inverse image $u^{-1}(\partial\Nbd(D))$ is a union of circles, one going around each point of $u^{-1}(D)$.
We will show that by varying $v$, we can make $u^*(\bar\partial((1-\varphi)\cdot v))$ approximate the delta function at any point of this union of circles $u^{-1}(\partial\Nbd(D))$.
This implies the desired surjectivity of \eqref{dregularsingleeqn} since every nonzero element of $H^0(C,K_C\otimes T^*X)=H^1(C,TX)^*$ has nonzero restriction to $u^{-1}(\partial\Nbd(D))$ by holomorphicity and unique continuation (recall that $u^{-1}(\partial\Nbd(D))$ meets all components of $C$).
Note that adding finitely many point constraints to the deformation problem of $u:C\to X$ (which leads to considering sections of $K_C\otimes T^*X$ which may have poles at said points) does not affect this argument, since we can always choose $u^{-1}(\partial\Nbd(D))$ to be disjoint from these finitely many special points (even if they happen to coincide with $u^{-1}(D)$).

It remains to prove that we can make $\bar\partial((1-\varphi)\cdot v)$ approximate a delta function at any point of $u^{-1}(\partial\Nbd(D))$.
Fix local coordinates $X=\CC_z\times\CC^2_{x,y}$ near an (isolated, by hypothesis) intersection point $u(C)\cap D$ in which $\pi$ is the $z$-coordinate, so $D=\{z=0\}=0\times\CC^2_{x,y}$ and $\Nbd(D)=D^2_\varepsilon\times\CC^2_{x,y}$.
Choose $\varphi$ to be a smoothing of the characteristic function $H(\varepsilon-z\bar z)$ of the $\varepsilon$-disk in $\CC_z$, so that $\bar\partial(1-\varphi)$ is a smoothing of $\delta(\varepsilon-z\bar z)zd\bar z$.
Writing $v$ in Laurent series expansion $v=\sum_kf_k(x,y)z^k\partial_z$, we calculate that $\bar\partial((1-\varphi)\cdot v)=\bar\partial(1-\varphi)\cdot v$ is a smoothing of $\delta(\varepsilon-z\bar z)\sum_kf_k(x,y)z^{k+1}d\bar z\partial_z$.
Now the factor $\sum_kf_k(x,y)z^{k+1}d\bar z\partial_z$ can approximate any continuous function on $\partial D^2_\varepsilon\times\CC^2_{x,y}=\partial\Nbd(D)$ which is holomorphic on fibers $e^{i\theta}\times\CC^2_{x,y}$ (use approximation by Fourier polynomials in the $\partial D^2_\varepsilon$ direction).
The pullbacks of such functions to $u^{-1}(\partial\Nbd(D))$ are dense in continuous functions since $u$ is simple and $\pi$ is a ramified covering.
\end{proof}

\subsection{Generic transversality}

We now use the existence of enough first order deformations (Lemma \ref{dregularsingle}) to prove `generic transversality', namely that generic divisorial deformations are regular in the following sense:

\begin{definition}[$D$-regular]\label{dregular}
We say that a complex manifold $X$ is \emph{$D$-regular} (for a given divisor $D\subseteq X$) when every $D$-controlled simple map is regular (Definition \ref{regulardef}).
More generally, we make the same definition for any family $X\to B$ and any relative divisor $D$.
\end{definition}

For all our eventual applications, we will need the families version of generic transversality given in Lemma \ref{generictransversalityfamily} below.
But for sake of exposition, we begin by treating the case of a single complex manifold:

\begin{lemma}[Generic transversality]\label{generictransversality}
Fix a complex manifold $X$, a divisor $D\subseteq X$, and a finite set $A\subseteq D$.
After `trimming' $X$ to remove a closed subset contained in $D\setminus A$, there exists a Banach manifold $\cJ_D(X)$ of $D$-deformations of $X$, a generic point of which is $D$-regular.
\end{lemma}

The effect of trimming $X$ is to `localize' the problem near the finite set $A$: it means we only divisorially deform $X$ in a neighborhood of $A$ and that we only care about curves whose intersection with $D$ is contained in a neighborhood of $A$ (in particular, the statement becomes vacuous when $A=\varnothing$).

\begin{proof}
This is a typical argument based on Smale's infinite-dimensional Sard theorem \cite{smalesard}.
The idea is that Lemma \ref{dregularsingle} asserts $D$-regularity of the universal family over $\cJ_D(X)$, and the infinite-dimensional Sard theorem converts this to $D$-regularity of fibers over generic points of $\cJ_D(X)$.

We begin by fixing a precise space $\cJ_D(X)$ to consider.
Let $\DD\subseteq\CC$ denote the unit disk.
Fix coordinates $\DD\times\DD^{n-1}$ on $X$ near each point $a\in A$ with $a=(0,0)$ and $D=0\times\DD^{n-1}$, and \emph{remove from $X$ the part of $D$ outside the interiors of the charts $0\times\DD^{n-1}$}.
We let $\cJ_D(X)$ consist of holomorphic sections $f$ of the tangent bundle of $(\DD\setminus 0)\times\DD^{n-1}$ with $\|f\|_2<\varepsilon$ for some $\varepsilon>0$, where the $L^2$-norm is weighted near $0\times\DD^{n-1}$ so that meromorphic sections have finite norm (this space is most naturally identified with the Lie algebra of $\cR((\DD\setminus 0)\times\DD^{n-1})$, and is subsequently mapped to it via the exponential map).
By smearing the Cauchy Integral Formula and appealing to Cauchy--Schwarz, we see that $\|f\|_\infty$ over any compact subset of the interior of $(\DD\setminus 0)\times\DD^{n-1}$ is bounded linearly in $\varepsilon$.
Thus for sufficiently small $\varepsilon>0$, the reglued family (Definition \ref{divisordeform}) is defined over $\cJ_D(X)$.
Using the $L^2$-norm here guarantees that the space $\cJ_D(X)$ is separable.

Now consider a compact smooth (not necessarily connected!) surface $C$ and a smooth family of almost complex structures on $C$ parameterized by a finite-dimensional smooth manifold $\cJ(C)$.
Now $W^{k,2}(C,X)$ is a smooth Banach manifold for any integer $k\geq 2$ (note that $W^{k,2}\subseteq C^0$ for such $k$), whose product with $\cJ(C)\times\cJ_D(X)$ carries the smooth Banach bundle
\begin{equation}
\begin{tikzcd}
\cJ(C)\times W^{k,2}(C,X)\times_{W^{k-1,2}(C,X)}W^{k-1,2}(C,\overline{TC}\otimes_\CC TX)\times\cJ_D(X)\ar[d]\\
\cJ(C)\times W^{k,2}(C,X)\times\cJ_D(X)\ar[u,bend left,"\bar\partial"]
\end{tikzcd}
\end{equation}
with a section $\bar\partial$ measuring the failure of the map $C\to X$ to be holomorphic.
Now the linearization (derivative) of $\bar\partial$ at a triple $(u:C\to X,j,J)$ is a map
\begin{equation}
W^{k,2}(C,u^*TX)\oplus T_j\cJ(C)\oplus T_J\cJ_D(X)\to W^{k-1,2}(C,\overline{TC}\otimes_\CC u^*TX)
\end{equation}
whose restriction to the first direct summand is the deformation complex of the map $u$.
Lemma \ref{dregularsingle} guarantees that the restriction to $T_J\cJ_D(X)$ surjects onto the obstruction space $H^1(C,TX)$ if $u$ is $D$-controlled.
Thus $\bar\partial$ is transverse to zero at every $D$-controlled holomorphic triple $(u,j,J)$ with $u$ simple.

Now restrict to the clopen subset of $W^{k,2}(C,X)$ consisting of those maps whose restriction to every component of $C$ has positive algebraic intersection with $D$ (thus a holomorphic map is $D$-controlled iff it lies in this set).
Over this clopen set, the section $\bar\partial$ is transverse to zero at simple, hence its zero set $\bar\partial^{-1}(0)_\simple$ (the open simple locus) is a smooth Banach manifold, and the projection map
\begin{equation}
\bar\partial^{-1}(0)_\simple\to\cJ_D(X)
\end{equation}
is Fredholm by ellipticity of the deformation complex of the map $u$.
Now Sard--Smale \cite{smalesard} implies that the fibers of this map over generic elements of $\cJ_D(X)$ are regular.
We can cover all curves using countably many pairs $(C,\cJ(C))$, so we conclude that for generic elements of $\cJ_D(X)$, all $D$-controlled simple maps are unobstructed.

Now regularity is stronger than unobstructedness, since it involves a deformation problem with point constraints.
To prove regularity of $D$-controlled simple maps with respect to generic elements of $\cJ_D(X)$, it suffices to apply the above argument to triples $(C,\cJ(C),\gamma)$ where $\gamma$ is a finite set of point constraints (again, countably many such triples suffice to cover all possible situations).
\end{proof}

\begin{lemma}[Generic transversality in a family]\label{generictransversalityfamily}
Fix a family of complex manifolds $X\to B$ over a finite simplicial complex $B$, a relative divisor $D\subseteq X\to B$, and a set $A\subseteq D$ whose map to $B$ is proper with finite fibers.
After `trimming' $X$ to remove a closed subset contained in $D\setminus A$ and subdividing $B$, there exists a Banach space of $D$-deformations of $X\to B$, a generic point of which is $D$-regular.

More generally, suppose $D\subseteq X\times_BV\to V$ is a relative divisor over a subcomplex $V\subseteq B$, and suppose $U\subseteq B$ is a constructible open set (i.e.\ the complement of a subcomplex) contained in $V$.
Then after trimming $X$ over $V$, there exists a Banach space of $D$-deformations of $X\to B$ vanishing over the complement of $U$, a generic point of which is $D$-regular over $U$.
\end{lemma}

\begin{proof}
The argument is similar to Lemma \ref{generictransversality}.

We first argue that there exists a finite collection of subcomplexes $M_i\subseteq B$ and charts
\begin{equation}\label{divisorcoveringcharts}
\begin{tikzcd}
(\DD,0)\times\DD^{n-1}\times M_i\ar[d]\ar[r,hook,"\phi_i"]&(X,D)\ar[d]\\
M_i\ar[r,hook]&B
\end{tikzcd}
\end{equation}
with the property that $\phi_i^{-1}(A)\subseteq 0\times\DD_{1/2}^{n-1}\times M_i$ (where $\DD_{1/2}\subseteq\DD$ denotes the subdisk of radius $\frac 12$) and whose interiors $0\times\DD_{1/2}^{n-1}\times M_i^\circ$ jointly cover $A$ (note that these charts may overlap arbitrarily); here $\DD\subseteq\CC$ denotes the closed unit disk.
Such a collection of charts is illustrated as follows:
\begin{equation}
\tikzset{ccc/.initial=\textwidth/15}
\begin{tikzpicture}[x=\pgfkeysvalueof{/tikz/ccc},y=\pgfkeysvalueof{/tikz/ccc},baseline={([yshift=-.8ex]current bounding box.center)}]
\draw(0,1)rectangle(10,5);\node[anchor=west]at(10.5,3){$X$};
\draw[thick](0,0)to(10,0);\node[anchor=west]at(10.5,0){$B$};
\draw(1,2)to[out=0,in=180](6,3)to[out=180,in=-25](5,3.5)to[out=155,in=0](4,4)to[out=0,in=180](9.5,4.8);\node[anchor=east]at(1,2){$A$};
\newcommand\rectanglefromnodes[2]{\draw(#1)to(#1|-#2);\draw(#1-|#2)to(#2);\draw[ultra thick](#2)to(#1|-#2);\draw[ultra thick](#1-|#2)to(#1);}
\foreach\x/\y/\xx/\yy in{.5/1.5/2/2.5,1.8/1.8/4/3,3.5/2.3/5.4/3.1,5/2.7/7/3.7,5.2/3.2/4.5/3.95,3.5/3.4/4.8/4.5,4.65/3.85/7.5/4.8,6.5/4.1/9.7/4.9}
{
	\coordinate(A)at(\x,\y);
	\coordinate(B)at(\xx,\yy);
	\rectanglefromnodes{A}{B};
}
\end{tikzpicture}
\end{equation}
Note how the bold horizontal boundary of each chart $0\times(\DD^n\setminus\DD_{1/2}^{n-1})\times M_i$ is indeed disjoint from $A$.

To construct such a system of charts \eqref{divisorcoveringcharts}, argue as follows.
By compactness of $A$ (note that $B$ is compact and $A\to B$ is proper), it suffices to produce a chart covering a neighborhood of any given point of $A$.
Given a point $a\in A$ over $b\in B$, we may fix a chart $(\DD,0)\times\DD^{n-1}\subseteq(X_b,D_b)$ in which $a=(0,0)$ and which contains no other point of $A$ (note $A\to B$ has finite fibers).
We may then extend this to a chart $(\DD,0)\times\DD^{n-1}\times M$ on $(X,D)$ over a neighborhood $M$ of $b\in B$.
Properness of $A\to B$ implies that the inverse image of $A$ in these coordinates will be contained in $0\times(D^2_{1/2})^{n-1}\times M$, after possibly shrinking $M$.
Now this $M$ is an open subset of $B$, rather than a subcomplex, but this may be rectified by subdividing $X$.

Now given any collection of charts $\phi_i:(\DD,0)\times\DD^{n-1}\times M_i\hookrightarrow(X,D)$ as above, we may construct $D$-deformations of (a trimming of) $X\to B$ as follows.
Consider any collection of sections of $T_{X/B}$ over each punctured chart $(\DD\setminus 0)\times\DD^{n-1}\times M_i$ which vanish over (the inverse image of) $\partial M_i=M_i\setminus M_i^\circ$ (the topological boundary, as a subset of $B$).
The (naive) sum of all such sections may (despite having no control on how the charts intersect each other) be exponentiated as in Definition \ref{divisordeform} to give a regluing of $X\setminus\bigcup_i0\times(\partial(\DD^{n-1}))\times M_i\to B$ (note that the locus we remove from $X$ is indeed disjoint from $A$).

Finally, we should specify an appropriate Banach space of $D$-deformations.
On each chart $\DD\times\DD^{n-1}\times M$, we consider the space of simplex-wise analytic sections of $T_{X/B}$ over $(\DD\setminus 0)\times\DD^{n-1}\times M$ which vanish over $\partial M$ (and over $B'$, if present).
We consider the $L^2$-norm on this space which integrates over a small neighborhood inside the complexification $(\DD\setminus 0)\times\DD^{n-1}\times M_\CC$, with exponential weight near the puncture $0\times\DD^{n-1}\times M_\CC$ so that poles of arbitrary orders there are allowed.
Now we take our Banach space of $D$-deformations to be the direct sum of all these spaces.
Now the same argument used in the proof of Lemma \ref{generictransversality} shows that generic elements of this Banach space are $D$-regular.
\end{proof}

\subsection{Enough divisors}

To get any mileage out of generic transversality as formulated in Lemmas \ref{generictransversality} and \ref{generictransversalityfamily}, we need a sufficiently rich collection of divisors.
Given a single 1-cycle $z$ in a single threefold $X$, it is trivial to observe that, after replacing $X$ with a small neighborhood of $z$, there exists a divisor $D\subseteq X$ controlling $z$ (namely, $D$ is a union of transverse disks at a finite collection of smooth points on $z$).
To apply Lemma \ref{generictransversalityfamily}, we will need to know this result for families of cycles in families of threefolds, which will involve a nontrivial inductive argument to keep the chosen divisors disjoint:

\begin{proposition}[Enough divisors]\label{divisorcovering}
Let $X\to B$ be a family of complex threefolds over a finite simplicial complex, and let $K\subseteq\cZ(X/B)$ be a compact analytic set for which the map $K\to B$ is injective.
After possibly removing a closed subset from $X$ disjoint from $K$ (call this `trimming $X$ near $K$'), there exists a finite collection of open subsets $U_i\subseteq B$ and disjoint relative divisors $D_i\subseteq X\times_BU_i\to U_i$ such that every $z\in K$ is $D_i$-controlled for some $i$.
\end{proposition}

We emphasize that, in this statement, it is utterly irrelevant to ask that the $U_i$ cover $B$, rather the goal is to `cover' (or, more precisely, control) all the cycles in $K$.

\begin{proof}
First, let us discuss how to construct (germs of) relative divisors $D\subseteq X\to B$ locally near a given point $x\in X$.
Suppose $x$ lies over the interior of a simplex $\sigma\subseteq B$.
Define $D_\sigma\subseteq X_\sigma$ as the transverse zero set $D_\sigma=f_\sigma^{-1}(0)$ of a (germ of) holomorphic map $f_\sigma:(X_\sigma,x)\to(\CC,0)$ defined near $x$.
To explain the term `holomorphic' for $\pi_\sigma$, recall that $X_\sigma\to\sigma$ is the restriction of a given family $X_\sigma^\CC\to\sigma^\CC\cong\CC^{\dim\sigma}$ over the complexification, so it makes sense to require that $f_\sigma$ be the restriction to $X_\sigma$ of a (necessarily unique) holomorphic function on $X_\sigma^\CC$.
For $D_\sigma$ to be a \emph{relative} divisor, we need it to be submersive over $\sigma$, which in terms of $f_\sigma$ is the condition that $df_\sigma|T_{X/B}$ is surjective.
To extend $D_\sigma$ to a neighborhood of $x$ in the total space $X$, it suffices to extend $f_\sigma$ to a (simplexwise) holomorphic map $f$ (note that surjectivity of $df|T_{X/B}$ is an open condition), which can be done using induction on simplices and Lemma \ref{functionextension}.
Note that there are plenty of analytic functions on a simplex vanishing on its boundary, which will be important below when we want to `choose divisors generically'.
We note that the resulting germ of relative divisor $D\subseteq X\to B$ can be promoted to a true (not germ) relative divisor over an open neighborhood of the image of $x$ in $B$, by removing a suitable closed subset of $X$.

Given the local existence of relative divisors, compactness of $K$ immediately produces a finite collection of relative divisors $D_i\subseteq X\times_BU_i\to U_i$ ($U_i\subseteq B$ open) such that every $z\in K$ is $D_i$-controlled for some $i$.
These divisors, however, need not be disjoint.
Rather than ensuring the divisors are disjoint ($D_i\cap D_j=\varnothing$ for all $i\ne j$), note that it is enough to ensure the seemingly weaker condition that $D_i\cap D_j\cap z=\varnothing$ for all $z\in K$ and $i\ne j$, as then $\bigcup_{i<j}D_i\cap D_j\subseteq X$ is closed and disjoint from $K$ so we can simply remove it.
To produce divisors with this property, we use an inductive argument, the key being that intersections $D\cap D'\cap z$ generically happen over a codimension two ($\dim(\cU/\cZ)-\codim D-\codim D'=-2$) subset of $K$.

To make the inductive argument work, we consider the following more general problem.
In addition to the data of $X\to B$ and $K\subseteq\cZ(X/B)$, fix a relative singular divisor $D^\prev\subseteq X\to B$ whose intersection with every cycle in $K$ is discrete.
By `relative singular divisor', we mean a (not necessarily disjoint or transverse) union of compact subsets of relative divisors over open subsets of the base.
We then ask for a finite collection of relative divisors $D_i\subseteq X\times_BU_i\to U_i$ ($U_i\subseteq B$ open) which together control all $z\in K$ and which are disjoint from each other and from $D^\prev$.
Our original problem is the special case $D^\prev=\varnothing$.

Now we claim that this more general problem always has a solution.
We argue by induction on $\dim K$, the case $K=\varnothing$ being trivial.
Suppose given an instance of the more general problem with some given $(D^\prev,K)$.
Given any choice of relative divisors $D_i$ (not necessarily mutually disjoint or disjoint from $D^\prev$) controlling all $z\in K$, we consider the problem of the pair
\begin{multline}
(D^{\prev\prime},K')=\Biggl(D^\prev\cup\bigcup_iD_i,\\\pi_K\biggl(\Bigl(\cU(X/B)\times_{\cZ(X/B)}K\Bigr)\cap\bigcup_i\Bigl(D_i\cap\bigl(D^\prev\cup\bigcup_{j\ne i}D_j\bigr)\Bigr)\biggr)\Biggr).
\end{multline}
Choosing the $D_i$ generically ensures that $K'$ has real codimension at least two inside $K$, and hence that this new problem has a solution by the induction hypothesis.
We may thus choose relative divisors $D_i'$ which solve $(D^{\prev\prime},K')$, that is they are mutually disjoint, disjoint from $D^\prev$ and from every $D_i$, and they control all $z\in K'$, hence all $z$ in a neighborhood of this set (Lemma \ref{controlledopen}).
Now to solve our original problem $(D^\prev,K)$, we take these $D_i'$ along with the restrictions of the $D_i$ to $U_i\setminus\Nbd K'$.
The intersections of these divisors with each other and with $D^\prev$ will be disjoint from $\cU(X/B)\times_{\cZ(X/B)}K$ by definition of $K'$ hence can simply be removed from $X$.
\end{proof}

\subsection{Semi-regularity}\label{semiregularsec}

Transversality (in its various forms discussed in Section \ref{secregularity}) is a notion for maps $u:C\to X$, and it is much less clear what a reasonable notion of transversality for 1-cycles $\sum_im_iC_i$ in $X$ would be.
The following notion is rather blunt but will serve our purpose (it is not intended to be related to semi-regularity in the sense of Bloch \cite{blochsemiregularity}).

\begin{definition}[Semi-regular]\label{semiregularsmooth}
For a family of complex manifolds $X\to B$ over a complex (or real) manifold $B$, we will call a 1-cycle $z=\sum_im_iC_i\in\cZ(X/B)$ \emph{semi-regular} when the (necessarily simple!) map $\bigsqcup_i\tilde C_i\to X$ is regular (in the family $X\to B$) in the sense of Definition \ref{regulardef}.
We denote by $\cZ^\semiregular\subseteq\cZ$ the locus of semi-regular cycles.
\end{definition}

Semi-regularity is not preserved by, but does descend along, pullbacks.
Namely, suppose $X'\to B'$ is a pullback of $X\to B$, fix $z'\in\cZ(X'/B')$ and denote by $z\in\cZ(X/B)$ its image.
If $z'$ is semi-regular then $z$ is semi-regular.
If $z$ is semi-regular and $B'\to B$ is a submersion, then $z'$ is semi-regular.
The condition that $B'\to B$ be a submersion is necessary: the deformation theory of $z$ may be unobstructed, yet become obstructed when we restrict from varying in $B$ to varying in $B'$.

\begin{definition}[Semi-regular]\label{semiregularsimplex}
For a family of complex manifolds $X\to B$ over a simplicial complex $B$, we will call a 1-cycle $z=\cZ(X/B)$ \emph{semi-regular} when it is semi-regular as a point of $\cZ(X_{\sigma^\circ}/\sigma^\circ)$ where $\sigma\subseteq B$ is the smallest simplex containing the image of $z$ (that is, the unique simplex whose interior contains the image of $z$).
Note that here we are considering semi-regularity in the family $X_{\sigma^\circ}\to\sigma^\circ$ over the real simplex $\sigma^\circ$, not over its complexification (which would be a weaker condition).
\end{definition}

We use this `stratum-wise' definition of semi-regularity since it is compatible with pullback: we have $\cZ(X'/B')^\semiregular=\cZ(X/B)^\semiregular\times_BB'$ for $X'\to B'$ the pullback of $X\to B$ along a map of simplicial complexes $B'\to B$.

\subsection{Interior semi-regularity}\label{interiorsemiregulargrothendieck}

Now we would like define a variant of the Grothendieck group of 1-cycles $H^*_c(\cZ/\Cpx_3)$ using only semi-regular cycles (Definitions \ref{semiregularsmooth} and \ref{semiregularsimplex} above).
For this variant to have enumerative significance (or, more specifically, so that it admits a pushforward map to $H^*_c(\cZ/\Cpx_3)$), we must restrict to \emph{interior} semi-regular cycles (which just means cycles in the interior $\cZ^{\semiregular\circ}$ of the locus of semi-regular cycles $\cZ^\semiregular\subseteq\cZ$).
Now semi-regularity (Definition \ref{semiregularsimplex}) is compatible with pullback and local in the sense of Definition \ref{grothgpdefmodifiedeasy}, but formation of the interior is not compatible with pullback: for a family of threefolds $X\to B$ over a finite simplicial complex $B$ and a map $B'\to B$, we have an (open) inclusion
\begin{equation}\label{intsemiregpullback}
\cZ(X'/B')^{\semiregular\circ}\supseteq\cZ(X/B)^{\semiregular\circ}\times_BB'
\end{equation}
which need not be an equality (at least, we see no reason why it would always be an equality).
This means that Definition \ref{grothgpdefmodifiedeasy} is not sufficient to define a Grothendieck group of interior semi-regular 1-cycles $H^*_c(\cZ^{\semiregular\circ}/\Cpx_3)$, so it requires some generalization which we now explain.

To construct a sensible Grothendieck group of interior semi-regular 1-cycles, we turn to Lemma \ref{bivariantdoublecomplex}, which decomposed the bivariant group $H^*_c(\cZ(X/B)/B)$ into a combination of the bivariant groups $H^*_c(\cZ(X\tighttimes_B\sigma/\sigma)/\sigma)$ obtained by restricting $X\to B$ to simplices $\sigma\subseteq B$.
Now we turn around the logic and take the statement of Lemma \ref{bivariantdoublecomplex} as a \emph{definition}.
This allows for more flexibility in the sort of subsets of $\cZ$ that we can consider (compared with the setup of Definition \ref{grothgpdefmodifiedeasy}).

\begin{definition}[Modified Grothendieck group of 1-cycles $H^*_c(\cZ_\alpha/\Cpx_3)$]\label{grothgpdefmodifiedhard}
Let $\cZ_\alpha\subseteq\cZ$ be the specification of a locally closed subset $\cZ_\alpha(X/\sigma)\subseteq\cZ(X/\sigma)$ for every family of threefolds $X\to\sigma$ over a \emph{simplex}, satisfying the following properties:
\begin{itemize}
\item For every map of simplices $\sigma'\to\sigma$, the pair of locally closed subsets $\cZ_\alpha(X/\sigma)\times_\sigma\sigma'$ and $\cZ_\alpha(X\times_\sigma\sigma'/\sigma')$ (both subsets of $\cZ(X\times_\sigma\sigma'/\sigma')$) is good (in the sense of the discussion below Definition \ref{grothgpdefmodifiedeasy}), and they are equal when $\sigma'\to\sigma$ is surjective.
\item$\cZ_\alpha\subseteq\cZ$ is local, in the sense that for any open set $X^-\subseteq X\to B$, we have $\cZ_\alpha(X^-/B)=\cZ(X^-/B)\cap\cZ_\alpha(X/B)$.
\end{itemize}
\emph{Warning:} although $\cZ_\alpha(X/B)$ may be defined for finite simplicial complexes $B$, we \emph{completely ignore} these spaces and instead consider only the spaces $\cZ_\alpha(X\times_B\sigma/\sigma)$ for individual simplices $\sigma$!

Now we consult Lemma \ref{bivariantdoublecomplex} and \emph{define}
\begin{equation}\label{grothgpdefmodifiedhardeqn}
H^*_c(\cZ_\alpha(X/B)/B):=H^*_c\Bigl(\cZ(X/B),\textstyle\bigoplus\limits_{\sigma\subseteq B}\ZZ_{\cZ_\alpha(X\tighttimes_B\sigma/\sigma)}[\dim\sigma]\Bigr).
\end{equation}
where the direct sum of sheaves on the right is equipped with the differential given by the sum (with the usual orientation signs) over codimension one inclusions $\sigma'\subseteq\sigma$ of the maps $\ZZ_{\cZ_\alpha(X\tighttimes_B\sigma/\sigma)}\to\ZZ_{\cZ_\alpha(X\tighttimes_B\sigma'/\sigma')}$, arising from the fact that the pair of locally closed subsets $\cZ_\alpha(X\tighttimes_B\sigma/\sigma)$ and $\cZ_\alpha(X\tighttimes_B\sigma'/\sigma')$ of $\cZ(X/B)$ is good (which is implied by our hypothesis on $\cZ_\alpha$).
\emph{Warning:} The notation in \eqref{grothgpdefmodifiedhardeqn} is very abusive: the left hand side is \emph{not} the bivariant $H^*_c$ group of a map $\cZ_\alpha(X/B)\to B$.

We may now define $H^*_c(\cZ_\alpha/\Cpx_3):=\dircolim_{X\to B}H^*_c(\cZ_\alpha(X/B)/B)$ as usual.
We should note that this directed colimit is filtered, by following the proof of Lemma \ref{familyfiltered}, which uses the hypothesis that $\cZ_\alpha\subseteq\cZ$ is local and that we have equality $\cZ_\alpha(X/\sigma)\times_\sigma\sigma'=\cZ_\alpha(X\times_\sigma\sigma'/\sigma')$ for $\sigma'\to\sigma$ surjective, in the last step.
\end{definition}

Definition \ref{grothgpdefmodifiedhard} applies to the locus of interior semi-regular cycles $\cZ^{\semiregular\circ}\subseteq\cZ$, yielding a group $H^*_c(\cZ^{\semiregular\circ}/\Cpx_3)$.
More generally, given any system of locally closed subsets $\cZ_\alpha\subseteq\cZ$ as in Definition \ref{grothgpdefmodifiedeasy}, we may consider the subset $\cZ_\alpha^{\semiregular\circ}=(\cZ_\alpha)^{\semiregular\circ}\subseteq\cZ_\alpha$ (where interior is taken relative $\cZ_\alpha$, not relative $\cZ$!) and apply Definition \ref{grothgpdefmodifiedhard} to form a group $H^*_c(\cZ_\alpha^{\semiregular\circ}/\Cpx_3)$.
Note that $\cZ_\alpha\cap\cZ^{\semiregular\circ}$ may be a proper subset of $\cZ_\alpha^{\semiregular\circ}$ when $\cZ_\alpha\subseteq\cZ$ is not an open subset.

\subsection{Graph trick}\label{graphtricksec}

We are now working towards using generic transversality (Proposition \ref{generictransversalityfamily}) to show that the map $H^*_c(\cZ^{\semiregular\circ}/\Cpx_3)\to H^*_c(\cZ/\Cpx_3)$ is an isomorphism (Theorem \ref{grothendiecktransverse} below).
However, there appears, at first glance, to be a gap between the applicability of Proposition \ref{generictransversalityfamily} and the situation of classes in $H^*_c(\cZ/\Cpx_3)$.
Namely, any class in $H^*_c(\cZ/\Cpx_3)$ may be represented by a family $X\to B$ over a finite simplicial complex $B$ and a class in $H^*_K(\cZ(X/B)/B)$ for some compact $K\subseteq\cZ(X/B)$.
But to perturb the family $X\to B$ to achieve transversality near $K$ using Proposition \ref{generictransversalityfamily}, we require a collection of relative divisors controlling all cycles in $K$, which may be produced using Proposition \ref{divisorcovering} \emph{only under the additional hypothesis that $K\to B$ is injective}.
To bridge this gap, we introduce the `graph trick', which converts a class in $H^*_K(\cZ(X/B)/B)$ for some compact $K\subseteq\cZ(X/B)$ into a class in $H^*_{K'}(\cZ(X'/B')/B')$, representing the same element in $H^*_c(\cZ/\Cpx_3)$, for which $K'\to B'$ is injective.

Before introducing the graph trick, we need to discuss \emph{stabilization}.
The stabilization of a family $X\to B$ should be the product $X\times\RR\to B\times\RR$, and the stabilization of a class $\gamma\in H^*_c(\cZ(X/B)/B))$ should be its pushforward to $H^*_c(\cZ(X\tighttimes\RR/B\tighttimes\RR)/B\tighttimes\RR))$ under the $\times 0$ map $B\to B\times\RR$; more generally we may consider multiplication with $\RR^k$ for any $k<\infty$.
However, we wish to stay within the setting of simplicial complexes, so we will replace $B\times\RR$ with some similar operation for simplicial complexes.

To stabilize a simplicial complex $B$, we consider its join with a point.
The join $T\star T'$ of simplicial complexes $T$ and $T'$ has vertices $V(T\star T')=V(T)\sqcup V(T')$, and a subset of these vertices spans a simplex iff its intersections with $V(T)$ and $V(T')$ both span simplices (or are empty).
Now the join $B\star *$ contains (inside its geometric realization) a copy of $B\times[0,1)$ via the map $(b,t)\mapsto((1-t)b,t)$.
Iterating, there is an open embedding from $B\times[0,1)^k$ into $B\star *\star\cdots\star *$ for the join with $k$ points.
We will really care only about this open subset $B\times[0,1)^k$.

To stabilize a family $X\to B$, we would like to pull it back to $B\times[0,1)^k$ under the evident projection to $B$.
To make sense of this over the join $B\star *\star\cdots\star *$, we consider the map $(b,t_1,\ldots,t_k)\mapsto(1-t_1-\cdots-t_k)^{-1}b$, which makes sense on complexifications away from the complexified simplex $(*\star\cdots\star *)_\CC$ (a family over the complement of this simplex is also a family over the entire $B\star *\star\cdots\star *$).
This defines the stabilized family $X\times[0,1)^k\to B\times[0,1)^k\subseteq B\star *\star\cdots\star *$.

Now there is a pushforward map $H^*_c(\cZ(X/B)/B)\to H^*_c(\cZ(X\tighttimes[0,1)^k/B\tighttimes[0,1)^k)/B\tighttimes[0,1)^k)$ induced by the inclusion of simplicial complexes $B\hookrightarrow B\star *\star\cdots\star *$.

\begin{lemma}[Graph trick]\label{graphtrick}
Every class in $H^*_c(\cZ/\Cpx_3)$ (resp.\ $\ker(H^*_c(\cZ^{\semiregular\circ}/\Cpx_3)\to H^*_c(\cZ/\Cpx_3))$) is in the image of $H^*_K(\cZ(X/B)/B)$ (resp.\ $\ker(H^*_K(\cZ^{\semiregular\circ}(X/B)/B)\to H^*_K(\cZ(X/B)/B))$) for some family of threefolds $X\to B$ over a finite simplicial complex $B$ and some compact subanalytic \cite{bierstonemilman} set $K\subseteq\cZ(X/B)$ whose map to $B$ is injective.

The same holds with $\cZ_\alpha$ in place of $\cZ$, for any locally closed subset $\cZ_\alpha\subseteq\cZ$ which is compatible with pullback and local, in the sense of Definition \ref{grothgpdefmodifiedeasy}.
\end{lemma}

\begin{proof}
If we remove the requirement that $K\to B$ be injective, then this follows from elementary properties of directed colimits.
It thus suffices to begin with an element of $H^*_K(\cZ(X/B)/B)$ (resp.\ $\ker(H^*_K(\cZ^{\semiregular\circ}(X/B)/B)\to H^*_K(\cZ(X/B)/B))$) and modify it to ensure that $K\to B$ is injective, keeping its image in $H^*_c(\cZ/\Cpx_3)$ (resp.\ $\ker(H^*_c(\cZ^{\semiregular\circ}/\Cpx_3)\to H^*_c(\cZ/\Cpx_3))$) the same.

Consider the stabilization $X\times[0,1)^k\to B\times[0,1)^k$.
For any class $\gamma\in H^*_K(\cZ(X/B)/B)$, we may consider its pushforward
\begin{equation}
\gamma\times y\in H^*_{K\times y}(\cZ(X\tighttimes[0,1)^k/B\tighttimes[0,1)^k)/B\tighttimes[0,1)^k)
\end{equation}
along the $\times y$ inclusion $B\hookrightarrow B\times[0,1)^k$ for any $y\in[0,1)^k$.
The classes $\gamma$ and $\gamma\times\{0\}^k$ represent the same element of $H^*_c(\cZ/\Cpx_3)$.
On the other hand, the classes $\gamma\times y$ all agree with each other in $H^*_c$.
We will take $y=\{\frac 12\}^k$.

Now consider the map $\cZ(X/B)\times[0,1)^k\to\cZ(X/B)\times\RR^k$ given by $T_f:(z,y)\mapsto(z,y+f(z))$ for some function $f:\cZ(X/B)\to(-\frac 12,\frac 12)^k$.
Note that the image of $K\times\{\frac 12\}^k$ under this map remains a compact subset of $(0,1)^k$, so we can make sense of the pushforward
\begin{equation}
(T_f)_*(\gamma\times\{{\textstyle\frac 12}\}^k)\in H^*_{T_f(K\times\{\frac 12\}^k)}(\cZ(X\tighttimes[0,1)^k/B\tighttimes[0,1)^k)/B\tighttimes[0,1)^k).
\end{equation}
This pushforward coincides with $\gamma\times\{{\textstyle\frac 12}\}^k$ in $H^*_c$, by a homotopy argument, so it represents the same class as $\gamma$ in $H^*_c(\cZ/\Cpx_3)$.
On the other hand, it is supported over the graph of $\frac 12+f|_K$, which maps injectively to $[0,1)^k$ provided we take $f:\cZ(X/B)\to(-\frac 12,\frac 12)^k$ to restrict to an embedding over $K$ (we can also make $f$ subanalytic so that this graph is subanalytic).

Now if instead $\gamma\in\ker(H^*_K(\cZ^{\semiregular\circ}(X/B)/B)\to H^*_K(\cZ(X/B)/B))$, we may proceed in the same way and see that 
\begin{multline}
(T_f)_*(\gamma\times\{{\textstyle\frac 12}\}^k)\in\ker\Bigl(H^*_{T_f(K\times\{\frac 12\}^k)}(\cZ^{\semiregular\circ}(X\tighttimes[0,1)^k/B\tighttimes[0,1)^k)/B\tighttimes[0,1)^k)\\
H^*_{T_f(K\times\{\frac 12\}^k)}(\cZ(X\tighttimes[0,1)^k/B\tighttimes[0,1)^k)/B\tighttimes[0,1)^k)\Bigr)
\end{multline}
represents the same element in $\ker(H^*_c(\cZ^{\semiregular\circ}/\Cpx_3)\to H^*_c(\cZ/\Cpx_3))$ as $\gamma$.

The same reasoning applies to any $\cZ_\alpha\subseteq\cZ$.
\end{proof}

It may be worthwhile to formulate the graph trick as the statement that the map from a suitably defined Grothendieck group $H^*_{c,\inj}(\cZ_\alpha/\Cpx_3)$ (a directed colimit of $H^*_K(\cZ_\alpha(X/B)/B)$ over families $X\to B$ and subsets $K\subseteq\cZ(X/B)$ injecting into $B$) to $H^*_c(\cZ_\alpha/\Cpx_3)$ is an isomorphism, for any $\cZ_\alpha\subseteq\cZ$ as in Definition \ref{grothgpdefmodifiedhard}.
We will not pursue this here.

\subsection{Generic transversality and the Grothendieck group}

We now use generic transversality to argue that the Grothendieck group of interior semi-regular 1-cycles $H^*_c(\cZ^{\semiregular\circ}/\Cpx_3)$ coincides with that of all 1-cycles $H^*_c(\cZ/\Cpx_3)$.
The argument will be somewhat technical, but the rough idea is as follows.
According to the graph trick (Lemma \ref{graphtrick}), we may represent classes in the Grothendieck group of 1-cycles by bivariant classes $\gamma\in H_K^*(\cZ(X/B)/B)$ supported over compact subsets $K\subseteq\cZ(X/B)\to B$ mapping injectively to $B$.
Since $K\to B$ is injective, we can control $K$ with divisors (Proposition \ref{divisorcovering}), which allows us to apply generic transversality (Proposition \ref{generictransversalityfamily}) to construct a perturbation of the family $X\to B$ making a neighborhood of $K$ semi-regular.
We may then `push' our class into the interior semi-regular locus.

\begin{theorem}\label{grothendiecktransverse}
The map $H^*_c(\cZ^{\semiregular\circ}/\Cpx_3)\to H^*_c(\cZ/\Cpx_3)$ is an isomorphism.
The same holds with $\cZ_\alpha$ in place of $\cZ$, for any any locally closed subset $\cZ_\alpha\subseteq\cZ$ which is compatible with pullback and local, in the sense of Definition \ref{grothgpdefmodifiedeasy}.
\end{theorem}

\begin{proof}[Proof of surjectivity]
Fix a class in $H^*_c(\cZ/\Cpx_3)$, and let us show it is in the image of $H^*_c(\cZ^{\semiregular\circ}/\Cpx_3)$.
Represent this class in $H^*_c(\cZ/\Cpx_3)$ by a class $\gamma\in H^*_K(\cZ(X/B)/B)$ for some family of threefolds $X\to B$ over a finite simplicial complex $B$ and some compact $K\subseteq\cZ(X/B)$.
According to the graph trick (Lemma \ref{graphtrick}), we may assume wlog that $K$ is subanalytic and that its map to $B$ is injective.

Now let us argue that we can take $B$ to be a homology manifold (dualizing sheaf is locally $\cong\ZZ[d]$ for some $d\in\ZZ$) over a neighborhood of the image of $K$ (this step is probably unnecessary, but will make things easier later).
To achieve this, first choose an embedding $B\hookrightarrow\hat B$ into a simplicial complex $\hat B$ which is a homology manifold over a neighborhood of the image of $B$.
For example, we could take $\hat B$ to be a triangulation of $[-1,2]^{V(B)}$ which contains (some subdivision of) $B\subseteq\RR^{V(B)}$ as a subcomplex (note that subdividing $B$ is not a problem by Lemma \ref{subdivisiongrothendieck}).
Now there is a local retraction $\hat B\to B$ (on complexifications), and pulling back $X\to B$ along it defines a family $\hat X\to\hat B$ whose restriction to $B$ is $X\to B$.
Now we may push forward our class in $H^*_K(\cZ(X/B)/B)$ to $H^*_K(\cZ(\hat X/\hat B)/\hat B)$ to achieve the desired result.

Now we apply `enough divisors' (Proposition \ref{divisorcovering}, which applies since $K$ is subanalytic and $K\to B$ is injective) to trim $X$ away from $K$ (which leaves the image of $\gamma\in H^*_c(\cZ(X/B)/B)$ in the Grothendieck group unchanged) and thereafter fix disjoint relative divisors $D_i\subseteq X\times_BW_i\to W_i$ for open sets $W_i\subseteq B$ which together control $K$.

Now we wish to apply generic transversality (Lemma \ref{generictransversalityfamily}, or rather its generalization to multiple divisors) to the family of controlling divisors $D_i\subseteq X\times_BW_i\to W_i$.
If $V_i\subseteq W_i$ denotes the largest subcomplex of $B$ contained in $W_i$, and $U_i\subseteq V_i$ denotes the largest constructible open set (i.e.\ the complement of a subcomplex) contained in $V_i$, then Lemma \ref{generictransversalityfamily} would guarantee that a generic deformation of $X\to B$ using $D_i$ over $U_i$ is $\bigsqcup_i(D_i\times_{W_i}U_i)$-regular.
To ensure that $\bigsqcup_i(D_i\times_{W_i}U_i)$ continues to control $K$ (at the moment, we only know that $\bigsqcup_iD_i$ controls $K$), it suffices to replace $B$ with a sufficiently fine subdivision thereof (which is ok by Lemma \ref{subdivisiongrothendieck}).

Now we actually apply Lemma \ref{generictransversalityfamily} not to $X\to B$ itself but to the stabilization $X\times\RR_{\geq 0}\to B\times\RR_{\geq 0}$ (see Section \ref{graphtricksec}) and the constructible open sets $U_i\times\RR_{>0}$.
This produces for us a collection $\Phi$ of simplex-wise analytic sections $\varphi_i:B\times\RR_{\geq 0}\to\cJ_{D_i\times\RR_{\geq 0}}(X\tighttimes\RR_{\geq 0}/B\tighttimes\RR_{\geq 0})=\cJ_{D_i}(X/B)\times\RR_{\geq 0}$ which are nonzero only on $U_i\times\RR_{>0}$, for which the resulting reglued family $(X\times\RR_{\geq 0})_\Phi\to B\times\RR_{\geq 0}$ is $(\bigsqcup_i(D_i\times_{W_i}U_i)\times\RR_{>0})$-regular.

Now we wish to `push' our class in $H^*_K(\cZ(X/B)/B)$ into the interior semi-regular locus of (the relative cycle space of) the reguled family $(X\times\RR_{\geq 0})_\Phi\to B\times\RR_{\geq 0}$.
This is in some sense trivial, yet when written out seems complicated.
First choose a local retraction $r:\cZ((X\tighttimes\RR_{\geq 0})_\Phi/(B\tighttimes\RR_{\geq 0}))\to\cZ(X/B)$ defined in a neighborhood of the compact set $K\subseteq\cZ(X/B)$.
\begin{equation}
\tikzset{ccc/.initial=\textwidth/15}
\begin{tikzpicture}[x=\pgfkeysvalueof{/tikz/ccc},y=\pgfkeysvalueof{/tikz/ccc},baseline={([yshift=-.8ex]current bounding box.center)}]
\draw[thick](0,0)--(10,0);\node[anchor=east]at(-0.2,0){$B\times\RR_{\geq 0}$};\draw[fill](0,0)circle(0.1);
\draw(0,1)--(0,5);\node[anchor=west]at(6.5,3.5){$\cZ((X\tighttimes\RR_{\geq 0})_\Phi/(B\tighttimes\RR_{\geq 0}))$};
\draw(0,1)to[out=0,in=180](3,1.3)to[out=0,in=180](5,0.5)to[out=0,in=180](8,2)to[out=0,in=180](10,1);
\draw(0,5)to[out=0,in=180](2,4.5)to[out=0,in=180](6,6)to[out=0,in=180](7,5)to[out=0,in=180](10,5);
\draw[ultra thick](0,2)--(0,4);\node[anchor=east]at(-0.1,3){$K$};
\draw[dashed](0,1.5)to[out=0,in=270](2,3)to[out=90,in=0](0,4.5);\node[anchor=west]at(2,2.5){domain of $r$};
\begin{scope}
\clip(0,1.5)to[out=0,in=270](2,3)to[out=90,in=0](0,4.5)--(0,1.5);
\filldraw[opacity=.15](0,2)--(0,4)to[out=0,in=180](.5,4.1)to[out=0,in=180](1,3.8)to[out=0,in=180](1.5,4.2)to[out=0,in=180](2,4)--(2,2)to[out=180,in=0](1.5,1.8)to[out=180,in=0](1,1.9)to[out=180,in=0](.5,2.1)to[out=180,in=0](0,2);
\end{scope}
\draw(1,3)--(2.5,3.5);\node[anchor=west]at(2.5,3.5){$r^{-1}(K)$};
\end{tikzpicture}
\end{equation}
More precisely, we suppose that $r$ is defined on an open subset of $\cZ((X\tighttimes\RR_{\geq 0})_\Phi/(B\tighttimes\RR_{\geq 0}))$ containing $K$, and we note that by shrinking said open set, we may assume wlog that the inverse image $r^{-1}(K)$ is proper over $[0,\varepsilon]$ for some $\varepsilon>0$.
We now consider the following diagram of spaces
\begin{equation}
\begin{tikzcd}
\cZ(X/B)\ar[d]\ar[r]&\cZ((X\tighttimes\RR_{\geq 0})_\Phi/(B\tighttimes\RR_{\geq 0}))\ar[r,"r\times p_{\RR_{\geq 0}}"]\ar[d]&\cZ(X/B)\times\RR_{\geq 0}\ar[d]\\
B\ar[r,"\times 0"]&B\times\RR_{\geq 0}\ar[r,equals]&B\times\RR_{\geq 0}
\end{tikzcd}
\end{equation}
and the induced diagram on bivariant $H^*_c$ groups
\begin{equation}\label{grothendiecktransversemovingclassdiagram}
\begin{tikzcd}[column sep=tiny]
H^*_K(\cZ(X/B)/B)\ar[r]\ar[rd]&H^*_{K\tighttimes[0,\varepsilon]}((\cZ(X/B)\tighttimes\RR_{\geq 0})/(B\tighttimes\RR_{\geq 0}))\ar[d,"{(r\times p_{\RR_{\geq 0}})^*}"]\\
{}&H^*_{r^{-1}(K)\tighttimes p_{\RR_{\geq 0}}^{-1}([0,\varepsilon])}(\cZ((X\tighttimes\RR_{\geq 0})_\Phi/(B\tighttimes\RR_{\geq 0}))/(B\tighttimes\RR_{\geq 0}))
\end{tikzcd}
\end{equation}
in which the horizontal and diagonal arrows are \emph{pushforward} along $(\times 0):B\to B\times\RR_{\geq 0}$ and in which the vertical arrow is \emph{pullback} along $r\times p_{\RR_{\geq 0}}$.
This pullback requires identifying the sheaves $\omega_{B\times\RR_{\geq 0}}=\omega_B\otimes\omega_{\RR_{\geq 0}}$ and $(r\times p_{\RR_{\geq 0}})^*\omega_{B\times\RR_{\geq 0}}=r^*\omega_B\otimes\omega_{\RR_{\geq 0}}$ on $\cZ((X\tighttimes\RR_{\geq 0})_\Phi/(B\tighttimes\RR_{\geq 0}))$ over a neighborhood of $r^{-1}(K)\tighttimes p_{\RR_{\geq 0}}^{-1}([0,\varepsilon])$: such an identification comes from the fact that $B$ is a homology manifold (over a neighborhood of the image of $K$) and from shrinking $\varepsilon>0$.
Now observe that the top horizontal map may be described as the pushforward along $(\times t):B\to B\times\RR_{\geq 0}$ for \emph{any} $t\in[0,\varepsilon]$!
We conclude that the class of $\gamma\in H^*_K(\cZ(X/B)/B)$ in the Grothendieck group of 1-cycles coincides with the class of
\begin{equation}
(r\tighttimes p_{\RR_{\geq 0}})^*(\gamma\times t)\in H^*_{r^{-1}(K){\cap}p_{\RR_{\geq 0}}^{-1}(t)}(\cZ((X\tighttimes\RR_{\geq 0})_\Phi/(B\tighttimes\RR_{\geq 0}))/(B\tighttimes\RR_{\geq 0})).
\end{equation}
Now we claim that the support $r^{-1}(K){\cap}p_{\RR_{\geq 0}}^{-1}(t)\subseteq\cZ((X\tighttimes\RR_{\geq 0})_\Phi/(B\tighttimes\RR_{\geq 0}))$ of this class lies within the interior semi-regular locus.
This locus lies within an arbitrarily small neighborhood of $K$, hence is $(\bigsqcup_i(D_i\times_{W_i}U_i)\times\RR_{>0})$-controlled (since being controlled is an open condition on 1-cycles).
Now the locus of $(\bigsqcup_i(D_i\times_{W_i}U_i)\times\RR_{>0})$-controlled 1-cycles is semi-regular and open, hence interior semi-regular.
We have thus shown that our given arbitrary class of $H^*_c(\cZ/\Cpx_3)$ lies in the image of $H^*_c(\cZ^{\semiregular\circ}/\Cpx_3)\to H^*_c(\cZ/\Cpx_3)$, as desired.
\end{proof}

\begin{proof}[Proof of injectivity]
To prove injectivity of the map $H^*_c(\cZ^{\semiregular\circ}/\Cpx_3)\to H^*_c(\cZ/\Cpx_3)$, we just need a relative version of the argument used for surjectivity.
Represent an arbitrary class in $\ker(H^*_c(\cZ^{\semiregular\circ}/\Cpx_3)\to H^*_c(\cZ/\Cpx_3))$ by a class in $\ker(H^*_K(\cZ^{\semiregular\circ}(X/B)/B)\to H^*_K(\cZ(X/B)/B))$ for some family of threefolds $X\to B$ over a finite simplicial complex $B$ and some compact $K\subseteq\cZ(X/B)$.
Our goal is to show this class maps to zero in $H^*_c(\cZ^{\semiregular\circ}/\Cpx_3)$.
According to the graph trick (Lemma \ref{graphtrick}), we may assume that $K$ is subanalytic and that its map to $B$ is injective.

As in the proof of surjectivity, we will need to subdivide $B$ at various points in the argument.
This presents an additional difficulty in the present context, as the notion of interior semi-regularity is not preserved under subdivision (compare Section \ref{interiorsemiregulargrothendieck}).
By replacing the standard map (on geometric realizations) $bB\to B$ by a generic perturbation thereof (say, send each vertex of $bB$ to any interior point of the corresponding simplex of $B$, not necessarily its barycenter, and extend linearly), we can ensure that every semi-regular 1-cycle over $B$ remains semi-regular over $bB$ (the difference being that semi-regularity over $bB$ involves directions in $bB$ which are tangent to the simplex of $bB$ we are in, which may be a proper subspace of the tangent space of the simplex of $B$ we are in).
Now the \emph{interior} semi-regular locus over $bB$ may be \emph{strictly larger} than that over $B$, but it does not matter: there is still a pushforward map from $H^*_c(\cZ^{\semiregular\circ}(X/B)/B)$ to $H^*_c(\cZ^{\semiregular\circ}(bX/bB)/bB)$, which is all we need.
This pushforward map is compatible with the map from both the domain and codomain to $H^*_c(\cZ^{\semiregular\circ}/\Cpx_3)$, by the usual concordance argument from Lemma \ref{subdivisiongrothendieck}.
This means we are allowed to replace $B$ with a very fine subdivision thereof.

Now we continue to follow the argument used for surjectivity.
For the same reasons as before, we may assume wlog that $B$ is a homology manifold over a neighborhood of the image of $K$.
We may apply `enough divisors' to produce disjoint relative divisors $D_i\subseteq X\times_BW_i\to W_i$ which together control $K$.
We may also apply generic transversality (after subdividing $B$) to produce a collection $\Phi$ of analytic sections $\varphi_i$ nonzero on $U_i\times\RR_{>0}$ for which the reglued family $(X\times\RR_{\geq 0})_\Phi\to B\times\RR_{\geq 0}$ is $(\bigsqcup_i(D_i\times_{W_i}U_i)\times\RR_{>0})$-regular.

Now it remains to push our situation in $\cZ(X/B)\to B$ into the interior semi-regular locus of (the relative cycle space of) the reguled family $(X\times\RR_{\geq 0})_\Phi\to B\times\RR_{\geq 0}$.
We draw the same diagram \eqref{grothendiecktransversemovingclassdiagram}, and we conclude that our class $\gamma\in\ker(H^*_K(\cZ^{\semiregular\circ}(X/B)/B)\to H^*_K(\cZ(X/B)/B))$ is equivalent in $H^*_c(\cZ^{\semiregular\circ}/\Cpx_3)$ to its image
\begin{multline}
(r\tighttimes p_{\RR_{\geq 0}})^*(\gamma\times t)\\\in\ker\Bigl(H^*_{r^{-1}(K){\cap}p_{\RR_{\geq 0}}^{-1}(t)}(\cZ^{\semiregular\circ}((X\tighttimes\RR_{\geq 0})_\Phi/(B\tighttimes\RR_{\geq 0}))/(B\tighttimes\RR_{\geq 0}))\\\to H^*_{r^{-1}(K){\cap}p_{\RR_{\geq 0}}^{-1}(t)}(\cZ((X\tighttimes\RR_{\geq 0})_\Phi/(B\tighttimes\RR_{\geq 0}))/(B\tighttimes\RR_{\geq 0}))\Bigr).
\end{multline}
But this map (whose kernel we are taking) is now an isomorphism since, as argued above, the support $r^{-1}(K){\cap}p_{\RR_{\geq 0}}^{-1}(t)$ lies in the interior semi-regular locus.
\end{proof}

\section{Generation by local curves}

We now study the Grothendieck group of interior semi-regular semi-Fano 1-cycles $H^*_c(\cZ_\semiFano^{\semiregular\circ}/\Cpx_3)$ from Section \ref{interiorsemiregulargrothendieck}.
Interior semi-regularity along with semi-Fano together imply that the dimension of this space of cycles is bounded above by its virtual dimension (Lemma \ref{semiregulardimension}).
This bound implies $H^*_c(\cZ_\semiFano^{\semiregular\circ}/\Cpx_3)$ is supported in virtual dimension $\geq 0$ and in virtual dimension zero is generated by Poincaré duals of smooth points.
We then use the multiplicity filtration to analyze these generators and prove Theorem \ref{mainsurjectivity}, which is the `generation' part of Theorem \ref{maincalculation}.
Generic transversality in the form of Theorem \ref{grothendiecktransverse} is crucial (but is applied as a black box).

\subsection{Semi-charts}\label{semichart}

Around each point $z\in\cZ(X)$ is a \emph{semi-chart} defined as follows.

\begin{definition}[Semi-chart]\label{semichartdef}
Let $z=\sum_im_iC_i\in\cZ(X)$ be a point.
Denote by $\tilde C_i\to C_i$ the normalization of $C_i$, so each $\tilde C_i$ is a compact smooth curve.
We consider all local deformations of $\tilde C=\bigsqcup_i\tilde C_i\to X$ (including deformations of the complex structure on the domain), which defines for us a germ of complex analytic space $(S_z,z)$ (note that this is precisely the deformation theory considered for the notion of semi-regularity Definition \ref{semiregularsmooth}).
There is now a map $(S_z,z)\to(\cZ(X),z)$ sending $\tilde C'=\bigsqcup_i\tilde C_i'\to X$ the cycle $\sum_im_iC_i'$.
We call this the \emph{semi-chart} at $z$.
\end{definition}

The prefix `semi-' is meant to reflect the fact that semi-charts need not be (germs of) open embeddings, since they does not take into account the possibility of the topology changing (as in $xy=t$ or $y^2=x(x-t)(x+t)$ near $t=0$) or of curves with multiplicities breaking apart (as in $y^2=tx$ near $t=0$).
By `open embedding' we just mean `an isomorphism onto an open subset' (topological isomorphism is enough).

It turns out that the semi-charts are open embeddings not into $\cZ$ but rather into the strata of a certain stratification of $\cZ$, as we now explain.
Recall the multiplicity map $\cZ\to\bM$ from Definition \ref{multfilt} and that it is (analytically) constructible and upper semi-continuous (sub-level sets $\cZ(X)_{\leq\bm}$ are open).
Its level sets $\cZ(X)_{=\bm}$ thus form a stratification of $\cZ(X)$ by locally closed analytic subsets.
Now let $\tilde\chi:\cZ(X)\to\ZZ$ denote the `Euler characteristic of the normalization' function sending $z=\sum_im_iC_i$ to $\sum_i\chi(\tilde C_i)$, which is analytically constructible since the universal family $\cU(X)\subseteq X\times\cZ(X)$ is proper over $\cZ(X)$.

\begin{lemma}[Semi-chart stratification]\label{semichartopen}
The function $\tilde\chi:\cZ(X)\to\ZZ$ is upper semi-continuous on every multiplicity stratum $\cZ(X)_{=\bm}$ (meaning sub-level sets $\cZ(X)_{=\bm,\tilde\chi\leq a}$ are open).
The semi-chart associated to any point of $\cZ(X)$ is an open embedding into the locally closed stratum $\cZ(X)_{=\bm,\tilde\chi=a}$ containing it.
\end{lemma}

\begin{proof}
The stratum $\cZ(X)_{=\bm}$ is, near a point $z=\sum_im_iC_i$, the set of cycles $\sum_im_iC_i'$ for $C_i'$ close to $C_i$.
The germ $(\cZ(X)_\bm,z)$ is thus equivalent to a product of germs $(\cZ(X),C_i)$.
The function $\tilde\chi$ on $(\cZ(X)_\bm,z)$ is the sum of functions $\tilde\chi$ on each of the factors $(\cZ(X),C_i)$.
The semi-chart of $z$ (which stays within $\cZ(X)_\bm$) is also the product of the semi-charts of each $C_i$.
It thus suffices to prove the result for each of these factors $(\cZ(X),C_i)$ individually.

So, let us consider the behavior of $\tilde\chi$ in a neighborhood of a 1-cycle consisting of a single (irreducible) curve $C$ (taken with multiplicity one).
A nearby curve $C'$ is, locally near smooth points of $C$, simply a nearby smooth curve, hence may be (non-canonically) identified with $C$ (as smooth manifolds) with nearby complex structure.
Near a singular point of $C$ (necessarily isolated), choose a ball $B$ around it so that $\tilde C\cap B$ is a disjoint union of disks.
Now a disjoint union of disks is the unique filling of a disjoint union of circles of maximal Euler characteristic, so we see that $\chi(\tilde C')\leq\chi(\tilde C)$, with equality iff $\tilde C'\cap B$ is also a disjoint union of disks, in which case we see that $\tilde C'\to X$ is a small perturbation of $\tilde C\to X$.
Thus we have shown that $\tilde\chi$ is upper semi-continuous and that the semi-chart at $C$ is an open embedding into the level set $\tilde\chi^{-1}(\tilde\chi(C))$.
\end{proof}

The above discussion applies without change in the relative setting of $\cZ(X/B)$.

\subsection{Vanishing in negative virtual dimension}

The virtue of interior semi-regularity in combination with semi-Fano is that it implies that the space of 1-cycles is bounded above by its virtual dimension:

\begin{lemma}\label{semiregulardimension}
We have $\dim_z\cZ_\semiFano^{\semiregular\circ}(X/B)\leq\dim B+2k$ for any point $z$ of chern number $k$, for any real analytic manifold $B$.
\end{lemma}

\begin{proof}
In view of the semi-chart stratification (Lemma \ref{semichartopen}, which remains valid over $B$), it suffices to bound the dimension of the semi-chart at every point $z\in\cZ(X/B)_\semiFano^{\semiregular\circ}$.
When $z=\sum_im_iC_i$ is semi-regular, its semi-chart has dimension $\dim B+2\sum_ik_i$ where $k_i=\langle c_1(T_{X/B}),C_i\rangle$.
When $z$ is semi-Fano (that is $k_i\geq 0$ for all $i$), this is bounded above by $\dim B+2\sum_im_ik_i=\dim B+2k$ for $k=\langle c_1(T_{X/B}),z\rangle$.
\end{proof}

This result immediately implies vanishing in negative virtual dimension:

\begin{proposition}\label{vdimpositive}
The group $H^*_c(\cZ_\semiFano^{\semiregular\circ}/\Cpx_3)$ is supported in virtual dimension $\geq 0$.
The same holds for any locally closed $\cZ_\alpha\subseteq\cZ_\semiFano$ in place of $\cZ_\semiFano$ (compatible with pullback and local in the sense of Definition \ref{grothgpdefmodifiedeasy}).
\end{proposition}

\begin{proof}
It suffices to prove that $H^*_c(\cZ_\semiFano^{\semiregular\circ}(X/B)/B)$ is supported in virtual dimension $\leq 0$ for every family of complex threefolds $X\to B$ over a finite simplicial complex $B$.
This group is a finite iterated extension of the groups $H^{*+\dim\sigma}_c(\cZ_\semiFano^{\semiregular\circ}(X\tighttimes_B\sigma/\sigma))$ for simplices $\sigma\subseteq B$.
Now each of these groups $H^{*+\dim\sigma}_c(\cZ_\semiFano^{\semiregular\circ}(X\tighttimes_B\sigma/\sigma))$ is supported in cohomological degree $i$ at most twice the chern number $2k$ by Lemma \ref{semiregulardimension}, which gives virtual dimension $2k-i\geq 0$.

The same applies to $\cZ_\alpha\subseteq\cZ_\semiFano$ (the dimension bound on $\cZ_\semiFano$ evidently passes to any subset $\cZ_\alpha\subseteq\cZ_\semiFano$).
\end{proof}

\subsection{Geometric local curve elements}\label{geometricgeneration}

It is natural to continue the reasoning in Proposition \ref{vdimpositive} and try to understand the classes in virtual dimension zero.
They will correspond to the Poincaré duals of smooth points in the space of semi-Fano and interior semi-regular 1-cycles where the dimension upper bound in Lemma \ref{semiregulardimension} is an equality.

\begin{definition}[Topological type]\label{toptypedef}
The \emph{topological type} of a 1-cycle $z=\sum_im_iC_i\in\cZ(X)$ is the multi-set of tuples $(g_i,m_i,k_i)$ consisting of the genus $g_i$ of $\tilde C_i$, the multiplicity $m_i\geq 1$, and the chern number $k_i=\langle c_1(T_{X/B}),C_i\rangle$; this is constant over any semi-chart.
\end{definition}

\begin{definition}[Geometric local curve element]\label{geolocalelts}
Let $X\to\sigma$ be a family of threefolds over a simplex.
Let $z=\sum_im_iC_i\in\cZ_\semiFano(X/\sigma)$ be a 1-cycle lying over the interior $\sigma^\circ$ for which $\bigsqcup_iC_i\subseteq X$ is smooth and unobstructed (including variations in the base) and whose semi-chart is an open embedding into $\cZ_\semiFano(X/\sigma)$.
The Poincaré dual of the point $z$ in this semi-chart thus defines an element of $H^*_c(\cZ_\semiFano^{\semiregular\circ}/\Cpx_3)$.
This element has cohomological degree $\sum_ik_i$ and chern number $\sum_im_ik_i$, hence has virtual dimension $\sum_i(m_i-1)k_i$.
When this virtual dimension is zero (that is, when $(m_i-1)k_i=0$ for all $i$), we call this Poincaré dual a \emph{geometric local curve element} (there is a sign ambiguity since we have not discussed orientations, but we will not worry about it).
Note that the topological type of a geometric local curve element is, by definition, semi-Fano ($k_i\geq 0$ for all $i$) and \emph{non-deficient} ($(m_i-1)k_i=0$ for all $i$).

More generally, the same definition applies to any locally closed $\cZ_\alpha\subseteq\cZ_\semiFano$ as in Definition \ref{grothgpdefmodifiedeasy} which is `closed under semi-charts' (the semi-chart of any $z\in\cZ_\alpha$ is contained in $\cZ_\alpha$).
\end{definition}

A product of geometric local curve elements is evidently again a geometric local curve element.

\begin{lemma}\label{geolocaleltsexist}
Geometric local curve elements exist in all (semi-Fano, non-deficient) topological types.
\end{lemma}

\begin{proof}
It suffices to deal with connected topological types $(g,m,k)$, since we can then get all topological types by taking products.

Begin with a local curve $E\to C$ where $H^1(C,E)=0$ (for example, we can take $E$ to be the direct sum of generic line bundles of degree $g-1$ and $g-1+k$).
If $m=1$, then we note that $\cZ(E,m)=\cZ(E,1)=H^0(C,E)$ is unobstructed, hence interior semi-regular, and the Poincaré dual of any point is the desired geometric local curve element.

If $m>1$ (so $k=0$), then $\cZ(E,m)_{=(m)}=H^0(C,E)$ remains unobstructed, but $\cZ(E,m)_{\leq(m)}$ may be more complicated and perhaps not semi-regular.
So, we will use generic transversality to find a suitable transverse perturbation of it.
Let $D\subseteq E$ be the fiber over some point of $C$.
According to Proposition \ref{generictransversality}, there exists a $D$-deformation $E_\varphi$ of a neighborhood of the zero section $C\subseteq E$ which is $D$-regular (all $D$-controlled simple maps are regular, hence all $D$-controlled 1-cycles are semi-regular, hence interior semi-regular since being $D$-controlled is an open condition).
Now the family of sections $\cZ(E,1)=H^0(C,E)$ is cut out transversally, hence for sufficiently small deformations $\varphi$, there exists a smooth unobstructed curve $C'\subseteq E_\varphi$ nearby the zero section $C\subseteq E$.
Now the cycle $mC'\in\cZ^{\semiregular\circ}(E_\varphi,m)$ is smooth and semi-regular.
Its semi-chart is an open embedding over a dense open set for dimension reasons (since $k=0$, Lemma \ref{semiregulardimension} says that the dimension of $\cZ^{\semiregular\circ}(E_\varphi,m)$ does not exceed the dimension of this semi-chart), so we can take the Poincaré dual of a generic point of it to obtain the desired geometric local curve element.

The basic point of this argument, in the harder case $m>1$, was just to show that there exists \emph{some} threefold $X$ containing an \emph{interior} semi-regular smooth cycle $mC'$ of chern number zero (and any given genus).
We used the soft but inexplicit Proposition \ref{generictransversality} to produce an example of such.
It is certainly conceivable one could come up with an explicit construction (though we do not know how).
\end{proof}

\begin{remark}\label{geolocaleltsnottopological}
Geometric local curve elements in $H^*_c(\cZ_\semiFano^{\semiregular\circ}/\Cpx_3)$ \emph{do not} depend only on their topological type.
The variation of the geometric local curve elements in a given topological type is studied in Ionel--Parker \cite[Section 7]{ionelparker} and Bai--Swaminathan \cite{baiswaminathan}.
These references suggest that one might be able to define a \emph{canonical} (though not unique) geometric local curve element of a given (semi-Fano, non-deficient) topological type by considering a super-rigid complex structure (if such were to exist, compare Wendl \cite{wendl}).
We will satisfy ourselves with proving the basic `invariance modulo lower multiplicity' result in Proposition \ref{reducedtoptype} below.
\end{remark}

Now let us strengthen Proposition \ref{vdimpositive} in the evident way:

\begin{proposition}\label{vdimpositivebetter}
The group $H^*_c(\cZ_\semiFano^{\semiregular\circ}/\Cpx_3)$ is supported in virtual dimension $\geq 0$, and in virtual dimension zero it is generated by geometric local curve elements.
The same holds for any locally closed $\cZ_\alpha\subseteq\cZ_\semiFano$ as in Definition \ref{geolocalelts} in place of $\cZ_\semiFano$.
\end{proposition}

\begin{proof}
It suffices to prove the assertion for $H^*_c(\cZ_\semiFano^{\semiregular\circ}(X/B)/B)$ for every family of complex threefolds $X\to B$ over a finite simplicial complex $B$.
This group is a finite iterated extension of the groups $H^{*+\dim\sigma}_c(\cZ_\semiFano^{\semiregular\circ}(X\tighttimes_B\sigma/\sigma))$ for simplices $\sigma\subseteq B$, so it suffices to show the result for these.
Now each of these groups $H^{*+\dim\sigma}_c(\cZ_\semiFano^{\semiregular\circ}(X\tighttimes_B\sigma/\sigma))$ is supported in cohomological degree $i$ at most twice the chern number $2k$ by Lemma \ref{semiregulardimension}, which gives virtual dimension $2k-i\geq 0$.
Virtual dimension zero classes are thus generated by Poincaré duals of smooth points of $\cZ_\semiFano^{\semiregular\circ}(X\tighttimes_B\sigma/\sigma)$ where the dimension bound in Lemma \ref{semiregulardimension} is an equality.
According to the semi-chart stratification (Lemma \ref{semichartopen}), it is enough to consider Poincaré duals of smooth points whose semi-chart is an open embedding.
By the definition of regularity (Definition \ref{regulardef}), such a semi-chart contains smooth curves.
The Poincaré dual of such a point is now a geometric local curve element.

The same reasoning applies to $\cZ_\alpha\subseteq\cZ_\semiFano$ as in Definition \ref{geolocalelts} (note that such $\cZ_\alpha$ is, by hypothesis, a union of semi-chart strata).
\end{proof}

\subsection{Multiplicity filtration}

The defect of Proposition \ref{vdimpositivebetter} is that the generating set it produces (geometric local curve elements) is quite redundant.
We now use the multiplicity filtration to argue that we actually only need one geometric local curve element of every topological type.
The basic point is that geometric local curve elements depend only on their topological type, when viewed modulo `lower multiplicity' 1-cycles (compare Remark \ref{geolocaleltsnottopological}).
Recall the multiplicity filtration from Definition \ref{multfilt}.

\begin{proposition}\label{reducedtoptype}
Geometric local curve elements in $H^*_c((\cZ_{\semiFano,=\bm})^{\semiregular\circ}/\Cpx_3)$ depend only on their topological type.
\end{proposition}

Before embarking on the proof, we emphasize the key difference between $\cZ_{\semiFano,=\bm}$ and $\cZ_{\semiFano,\leq\bm}$.
For any smooth cycle $z=\sum_im_iC_i\in\cZ_{\semiFano,=\bm}$ (meaning $\bigsqcup_iC_i\subseteq X$ is smooth), the semi-chart of $z$ is an open embedding into $\cZ_{\semiFano,=\bm}$.
That is, nearby points $z'\in\cZ_{\semiFano,=\bm}$ all have the form $\sum_im_iC_i'$ for smooth curves $C_i'$ near $C_i$.
In particular, if $z$ is semi-regular then it is interior semi-regular.
By contrast, the space $\cZ_{\semiFano,\leq\bm}$ could \emph{a priori} contain cycles close to $z$ where each $C_i$ splits apart in quite a complicated fashion.
This is the essential difference between $\cZ_{\semiFano,=\bm}$ and $\cZ_{\semiFano,\leq\bm}$ which is responsible for the geometric local curve elements being unique (in their topological type) in $H^*_c((\cZ_{\semiFano,=\bm})^{\semiregular\circ}/\Cpx_3)$ but not in $H^*_c((\cZ_{\semiFano,\leq\bm})^{\semiregular\circ}/\Cpx_3)$.

\begin{proof}
The proof has a few different steps, but all are quite `soft'.

Given a family $X\to B$ over a real analytic manifold $B$ and a smooth cycle $z=\sum_im_iC_i\in\cZ(X/B)$ which is semi-regular, we may produce a geometric local curve element in $H^*_c((\cZ_{\semiFano,=\bm})^{\semiregular\circ}/\Cpx_3)$ by choosing a simplex $\sigma\to B$ (of the same dimension) whose interior contains the fiber containing $z$ (evidently every geometric curve element is of this form).
The resulting geometric local curve element is independent of $\sigma$ by a concordance argument.
Moving $z$ within the class of semi-regular smooth cycles also leaves this geometric local curve element invariant.

Now there is another operation we may consider: cutting down the family $X\to B$ to a submanifold $B'\subseteq B$ containing the fiber containing $z$, such that $z$ remains semi-regular over $B'$.
To see that this also leaves the associated geometric local curve element invariant, consider a simplex $\sigma'\to B'$ (with the fiber containing $z$ in its interior) and extend it to a simplex $\sigma\to B$ where $\sigma'\subseteq\sigma$ is a face (arbitrary codimension).
Now $\cZ(X/B)_{\semiFano,=\bm}$ is, near $z$, a smooth manifold mapping to $B$ transversely to $B'$.
It follows that the pushforward of the Poincaré dual of $z\in\cZ(X'/B')_{\semiFano,=\bm}$ in $H^*_c(\cZ_{\semiFano,=\bm}^{\semiregular\circ}(X\times_B\sigma'/\sigma'))$ is the Poincaré dual of a nearby point over $\sigma^\circ$ in $H^*_c(\cZ_{\semiFano,=\bm}^{\semiregular\circ}(X\times_B\sigma/\sigma))$.

Now let us show that the geometric local curve element (constructed as above) actually depends on just $z$ and the single fiber $X_0$ containing it (that is, the family into which we include $X_0$ to make $z$ semi-regular does not matter).
First, note that for every $(X_0,z)$, there is some family into which $X_0$ can be included to make $z$ semi-regular (choose a divisor controlling $z$ and apply Lemma \ref{dregularsingle}).
Now suppose $X^1\to B^1$ and $X^2\to B^2$ are two families into which $X_0$ includes (as the fibers over basepoints $0\in B^1$ and $0\in B^2$) and in which $z$ is semi-regular.
We might consider trying to relate them by constructing a family $X^{12}\to B^1\times B^2$ whose restrictions to $B_1\times 0$ and $0\times B_2$ are the families $X^1$ and $X^2$ (which would imply, using the result of the previous paragraph, that the geometric local curve elements resulting from $X^1\to B^1$ and $X^2\to B^2$ are the same).
Instead of doing this, we will construct a third family $X^3\to B^3$ which is related to each of the two families $X^1\to B^1$ and $X^2\to B^2$ in this way.
To construct $X^3\to B^3$, we apply again Lemma \ref{dregularsingle} to $X_0$, with respect to some divisor $D_0\subseteq X_0$ controlling $C$.
But now we observe that $D_0$ extends (just by extension of local holomorphic functions) to relative divisors $D^1\subseteq X^1$ and $D^2\subseteq X^2$, which may be used to extend the deformation of $X$ given by $X^3\to B^3$ to deformations $X^{13}\to B^1\times B^3$ and $X^{23}\to B^2\times B^3$.
This gives the desired relation between the geometric local curve elements associated to $X^1\to B^1$ and $X^2\to B^2$.

We have now associated a geometric local curve element in $H^*_c((\cZ_{\semiFano,=\bm})^{\semiregular\circ}/\Cpx_3)$ to every smooth cycle $z=\sum_im_iC_i$ inside a threefold $X_0$ (and we have shown that every geometric local curve element arises in this way).
Finally, observe that as we deform $(X_0,z)$, the associated geometric local curve element is unchanged (since we can apply Lemma \ref{dregularsingle} to thicken a family as in the previous paragraph).
This is enough, since a pair consisting of a smooth cycle $z=\sum_im_iC_i$ inside a threefold $X_0$, regarded as a germ around the smooth curve $\bigsqcup_iC_i\subseteq X_0$, is classified up to deformation by its topological type (deform to the normal cone and apply Remark \ref{localdeform}).
\end{proof}

\begin{lemma}\label{reducedtoptypelinind}
Geometric local curve elements in $H^*_c((\cZ_{\semiFano,=\bm})^{\semiregular\circ}/\Cpx_3)$ (one for each topological type) are linearly independent.
\end{lemma}

\begin{proof}
We do not use this result logically, so we just sketch the argument.

Form the `interior smooth semi-regular Grothendieck group' $H^*_c((\cZ_\alpha)^{\semiregular\smooth\circ}/\Cpx_3)$ following Definition \ref{grothgpdefmodifiedhard} but using instead the interior of the locus of 1-cycles which are both semi-regular and have smooth support (meaning for $z=\sum_im_iC_i$ that $\bigsqcup_iC_i\subseteq X$ is a smooth curve).
Now we claim that the map $H^*_c((\cZ_\alpha)^{\semiregular\smooth\circ}/\Cpx_3)\to H^*_c((\cZ_\alpha)^{\semiregular\circ}/\Cpx_3)$ (for $\cZ_\alpha\subseteq\cZ_\semiFano$ locally closed) is bijective in virtual dimension $\leq 0$ and surjective in virtual dimension $1$.
The reason is that it fits into a long exact sequence whose third term is an iterated extension of $H^{*+\dim\sigma}((\cZ_\alpha^{\semiregular\circ}\setminus\cZ_\alpha^{\semiregular\smooth\circ})(X/\sigma))$, and this group is supported in virtual dimension $\geq 2$ by a dimension count similar to Lemma \ref{semiregulardimension}, since non-smoothness is a real codimension $\geq 2$ phenomenon given regularity.

It therefore suffices to show the result for $H^*_c((\cZ_{\semiFano,=\bm})^{\semiregular\smooth\circ}/\Cpx_3)$.
But now this group is graded by topological type, since the space of interior smooth cycles of fixed multiplicity is a disjoint union of open pieces corresponding to topological types.
It is therefore enough to show that each geometric local curve element is nonzero, and this follows from a discussion of orientability of the space of interior smooth semi-regular cycles (of fixed multiplicity).
\end{proof}

\begin{theorem}\label{filtrationgenerationcor}
The group $H^*_c(\cZ_\semiFano/\Cpx_3)$ is supported in virtual dimension $\geq 0$, and in virtual dimension zero it is generated by (the images of) any set of elements of $H^*_c(\cZ_{\semiFano,\leq\bm}/\Cpx_3)$ (various $\bm\in\bM$) which project to all geometric local curve elements in $H^*_c(\cZ_{\semiFano,=\bm}/\Cpx_3)$ (one of each topological type; compare Proposition \ref{reducedtoptype}).
\end{theorem}

\begin{proof}
In fact, we show the same statement for $H^*_c(\cZ_{\semiFano,S}/\Cpx_3)$, for any downward closed subset $S\subseteq\bM$.
That is, we show this group is supported in virtual dimension $\geq 0$ and in virtual dimension zero is generated by any set of elements of $H^*_c(\cZ_{\semiFano,\leq\bm}/\Cpx_3)$ (various $\bm\in S$) which project to all geometric local curve elements in $H^*_c(\cZ_{\semiFano,=\bm}/\Cpx_3)$ (one of each topological type).

By a directed colimit argument, it suffices to treat the case that $S$ is finite.
Choose a maximal element $\bm\in S$, and suppose the result is true for $S\setminus\{\bm\}$ (by induction on $\abs S$).
In view of the long exact sequence (a special case of \eqref{modifiedgrothles})
\begin{equation}
\cdots\to H^*_c(\cZ_{\semiFano,S\setminus\{\bm\}}/\Cpx_3)\to H^*_c(\cZ_{\semiFano,S}/\Cpx_3)\to H^*_c(\cZ_{\semiFano,=\bm}/\Cpx_3)\to\cdots
\end{equation}
it suffices to know that $H^*_c(\cZ_{\semiFano,=\bm}/\Cpx_3)$ is supported in virtual dimension $\geq 0$ and in virtual dimension zero is generated by geometric local curve elements.
This follows from Proposition \ref{vdimpositivebetter} applied to $H^*_c((\cZ_{\semiFano,=\bm})^{\semiregular\circ}/\Cpx_3)$ and Theorem \ref{grothendiecktransverse} applied to the map $H^*_c((\cZ_{\semiFano,=\bm})^{\semiregular\circ}/\Cpx_3)\to H^*_c(\cZ_{\semiFano,=\bm}/\Cpx_3)$ (which says it is an isomorphism).
\end{proof}

\begin{theorem}\label{mainsurjectivity}
The group $H^*_c(\cZ_\semiFano/\Cpx_3)$ is supported in virtual dimension $\geq 0$, and in virtual dimension zero it is:
\begin{itemize}
\item generated as a $\QQ$-algebra by the equivariant local curve elements $x_{g,m,k}$ with $g\geq 0$, $m\geq 1$, $k\geq 0$, and $(m-1)k=0$, after passing to $\QQ$ coefficients.
\item generated as a $\ZZ$-algebra by any choice of geometric local curve elements $y_{g,m,k}$ with $g\geq 0$, $m\geq 1$, $k\geq 0$, and $(m-1)k=0$.
\end{itemize}
\end{theorem}

\begin{proof}
According to Theorem \ref{filtrationgenerationcor}, we just need to show our claimed generators reduce to all geometric local curve elements in $H^*_c(\cZ_{\semiFano,=\bm}/\Cpx_3)$.
The reduction of a geometric local curve element in $H^*_c(\cZ_{\semiFano,\leq(m)}/\Cpx_3)$ is a geometric local curve element in $H^*_c(\cZ_{\semiFano,=(m)}/\Cpx_3)$ by definition.
Lemma \ref{eqlocalcurveassocgraded} tells us that the equivariant local curve element $x_{g,m,k}\in H^*_c(\cZ_{\semiFano,\leq(m)}/\Cpx_3)$ reduces to a geometric local curve element in $H^*_c(\cZ_{\semiFano,=(m)}/\Cpx_3)$.
Now reduction is compatible with products, so we get the geometric local curve elements of disconnected topological types as well.
\end{proof}

\begin{remark}\label{continue}
Theorem \ref{mainsurjectivity} is very close to Theorem \ref{maincalculation}: it remains only to prove that there are no relations among the local curve elements (geometric or equivariant, it is the same) of non-deficient ($(m_i-1)k_i=0$) semi-Fano ($k_i\geq 0$) topological type.
It seems possible one could show this by continuing the present line of reasoning.
It would suffice to prove Lemma \ref{reducedtoptypelinind} (whose proof we only sketched) and to prove that the connecting map
\begin{equation}
H^*_c(\cZ_{\semiFano,=\bm}/\Cpx_3)\xrightarrow{+1}H^*_c((\cZ_{\semiFano,<\bm}/\Cpx_3)
\end{equation}
vanishes in virtual dimension $1$ mapping to virtual dimension $0$.
According to the proof of Lemma \ref{reducedtoptypelinind}, elements of virtual dimension $1$ in $H^*_c(\cZ_{\semiFano,=\bm}/\Cpx_3)$ can be described by loops $z_t$ of smooth semi-regular cycles of multiplicity $\bm$.
The connecting map counts how many multiplicity $<\bm$ cycles fly off this family $z_t$ as $t$ varies around the circle.
It seems plausible that the present setting is sufficiently complex analytic to imply that cycles of multiplicity $<\bm$ approaching a cycle of multiplicity $\bm$ is a real codimension two phenomenon, hence is avoided in a generic path/loop.
We will not attempt to make this reasoning precise here.
\end{remark}

\section{Bi-algebra structure and constraints}\label{bialgebraconstraints}

We now show that the ring homomorphism $\ZZ[x_{g,m,k}]\to H^*_c(\cZ_\semiFano/\Cpx_3;\QQ)$ is injective (Theorem \ref{maininjectivity}).
To do this, we use the bi-algebra structure and division operations to produce an element of the kernel (if said kernel is nonzero) of a certain controlled form, which we can then show is detected by some Gromov--Witten invariant hence must be nonzero in $H^*_c(\cZ_\semiFano/\Cpx_3;\QQ)$.
The arguments involved in the first step are somewhat reminiscent of the Milnor--Moore Theorem \cite{milnormoore} (of which many generalizations already exist in the literature).

Consider the free polynomial ring $R=\ZZ[x_{g,m,k}]$ on formal variables $x_{g,m,k}$ indexed by integers $g\geq 0$, $m\geq 0$, and $k\geq 0$, satisfying $(m-1)k=0$, modulo the relation that $x_{g,0,k}=1$.
Equip $R$ with the co-unit and co-multiplication maps given by
\begin{align}
\eta:R&\to\ZZ&\Delta:R&\to R\otimes R\\
x_{g,m,k}&\mapsto 0\quad\text{for }m>0&x_{g,m,k}&\mapsto\sum_{\begin{smallmatrix}a+b=m\\a,b\geq 0\end{smallmatrix}}x_{g,a,k}\otimes x_{g,b,k}
\end{align}
on generators and extended to be algebra maps.
This makes $R$ into a commutative and co-commutative bi-algebra.
Let $\rho_1:R\to R$ be the identity, and let $\rho_d:R\to R$ for $d>1$ be given on generators by
\begin{equation}
\rho_d(x_{g,m,k})=\begin{cases}x_{g,m/d,k}&m\text{ divisible by }d\text{ and }k=0,\\0&\text{otherwise.}\end{cases}
\end{equation}
and extended multiplicatively.
This $\rho_d$ is a map of bi-algebras (commutes with $\Delta$ and $\eta$) by inspection.

The \emph{weight} of a monomial in the variables $x_{g,m,k}$ is a function $w:\ZZ_{\geq 0}\times\ZZ_{\geq 0}\to\ZZ_{\geq 0}$ defined by $w(ab)=w(a)+w(b)$ and $w(x_{g,m,k})=m\cdot 1_{g,k}$.
Given an arbitrary element $a\in R$, we denote by $a_w$ its weight $w$ part.
A tensor product of monomials $a\otimes b$ has a bi-weight $(w(a),w(b))$ and a total weight $w(a)+w(b)$.
The coproduct $\Delta$ preserves (total) weight.

\begin{lemma}[Weight splitting]\label{weightsplitting}
Let $A\subseteq R$ be a subgroup with the property that $\Delta(A)\subseteq(A\otimes R)+(R\otimes A)$.
Let $w$ be any nonzero weight.
If $A$ has an element with nonzero weight $w$ part, then $A$ has an element with nonzero weight $m\cdot 1_{g,k}$ part for some $(g,k)\in\supp w$.
\end{lemma}

\begin{proof}
The idea is to use $\Delta$ to `split' the weight $w$ until its support becomes a singleton.
We have $\Delta(a)_w=\Delta(a_w)$.
If the support of $w$ is not a singleton, then write $w=w_1+w_2$ for nonzero $w_1$ and $w_2$ of disjoint support.
Since the supports of $w_1$ and $w_2$ are disjoint, the result of applying $\Delta$ to a monomial of weight $w$ will have a \emph{unique} monomial of bi-weight $(w_1,w_2)$.
In particular, $a_w\ne 0$ implies that $\Delta(a)_{w_1,w_2}\ne 0$.
Since $\Delta(a)\in(A\otimes R)+(R\otimes A)$, we conclude that $A$ must have an element with nonzero weight $w_1$ part or weight $w_2$ part.
Now we replace $w$ with whichever of $w_1$ or $w_2$ it is and repeat until the support of $w$ is a singleton.
\end{proof}

\begin{lemma}[Weight purifying]\label{weightpurifying}
Let $A\subseteq R$ be a subgroup with the property that $\Delta(A)\subseteq(A\otimes R)+(R\otimes A)$.
If $A$ has an element with nonzero weight $m\cdot 1_{g,k}$ part, then $A$ has an element containing the single variable monomial $x_{g,m',k}$ for some $m'\leq m$.
\end{lemma}

\begin{proof}
The argument is similar to `weight splitting' Lemma \ref{weightsplitting}.
Let $w=m\cdot 1_{g,k}$.
Let $a\in A$ have nonzero weight $w$ part.
Among the weight $w$ monomials appearing in $a$, consider the factor $x_{g,m',k}$ with $m'$ the largest possible.
Now consider the monomials in $\Delta(a)_w$ of the form $x_{g,m',k}\otimes-$.
How can a given monomial in $a_w$ contribute such a monomial to $\Delta(a)_w$?
The factors $x_{g,m'',k}$ with $m''<m'$ must go completely on the right.
Of the factors of $x_{g,m',k}$, exactly one must go completely on the left, and the rest must go completely on the right.
Thus the monomials in $\Delta(a)_w$ of the form $x_{g,m',k}\otimes-$ are in bijection with the monomials in $a_w$ with at least one $x_{g,m',k}$ factor, and the effect of $\Delta$ is to multiply their coefficient by the number of such factors.
In particular, $\Delta(a)_w$ contains monomials of the form $x_{g,m',k}\otimes-$.
Appealing to $\Delta(a)\in(A\otimes R)+(R\otimes A)$, we have `split' $w$ unless $m'=m$, in which case we have proven the desired result.
\end{proof}

\begin{lemma}[Weight dividing]\label{weightdividing}
Let $A\subseteq R$ be a subgroup with the property that $\rho_d(A)\subseteq A$.
If $A$ has an element containing the single variable monomial $x_{g,m,0}$, then $A$ has an element containing the single variable monomial $x_{g,1,0}$ and no single variable monomials $x_{g,m',0}$ for $m'>1$.
\end{lemma}

\begin{proof}
Take an element of $A$ containing single variable monomials $x_{g,m,0}$ for various $m$, and apply $\rho_d$ to it where $d$ is the maximum $m$ among them.
\end{proof}

Order the pairs $(g,k)\in\ZZ_{\geq 0}\times\ZZ_{\geq 0}$ lexicographically, namely $(g,k)<(g',k')$ when either $g<g'$ or $g=g'$ and $k<k'$.
This is evidently a well-ordering.

\begin{proposition}\label{deltalemma}
Let $A\subseteq R$ be a subgroup with the property that $\Delta(A)\subseteq(A\otimes R)+(R\otimes A)$ and $\rho_d(A)\subseteq A$.
If $A\ne 0$, then there exists an element of $A$ containing the single variable monomial $x_{g,1,k}$ and no single variable monomials $x_{g,m,k}$ with $m>1$ or $x_{g',m',k'}$ with $(g',k')<(g,k)$.
\end{proposition}

\begin{proof}
Consider the set of all weights of all monomials appearing in elements of $A$.
Each such weight has a maximum pair $(g,k)$ in its support.
Fix $(g,k)$ to be the minimum such pair.
Let $a\in A$ have a monomial whose weight has $(g,k)$ as the maximum element in its support.
By Lemma \ref{weightsplitting}, there exists $a\in A$ with a monomial of weight $m\cdot 1_{g,k}$.
By Lemma \ref{weightpurifying}, there exists $a\in A$ with a single variable monomial $x_{g,m,k}$ (possibly different $m$).
By Lemma \ref{weightdividing}, there exists $a\in A$ with a single variable monomial $x_{g,1,k}$ and no single variable monomials $x_{g,m,k}$ with $m>1$.
Finally, $a$ has no single variable monomials $x_{g',m',k'}$ with $(g',k')<(g,k)$ by choice of $(g,k)$.
\end{proof}

\begin{theorem}\label{likemilnormoore}
Let $f:R\to R'$ be any morphism of bi-algebras compatible with bi-algebra endomorphisms $\rho_d$ (any such on $R'$).
Suppose that for every $g\geq 0$ and $k\geq 0$, there exists a group homomorphism $\varphi_{g,k}:R'\to\QQ$ sending $f(x_{g,1,k})$ to $1$, sending $f(x_{g',m',k'})$ to $0$ for $(g',k')>(g,k)$, and sending $f$ of any monomial of degree $>1$ to zero.
In this case the map $f$ is injective.
\end{theorem}

\begin{proof}
We have $\Delta(\ker f)\subseteq\ker(f\otimes f)$ since $f$ is compatible with co-multiplication $\Delta$, and we have $\rho_d(\ker f)\subseteq\ker f$ since $f$ is compatible with $\rho_d$.
Now $\ker(f\otimes f)=(\ker f)\otimes R+R\otimes(\ker f)$ if we base change to $\QQ$ (note that $R$ is torsion free, so it suffices to show injectivity after rationalizing).
We may thus apply Proposition \ref{deltalemma} (which works just as well over $\QQ$ as over $\ZZ$) to conclude that if $\ker f\ne 0$, then it has an element containing the single variable monomial $x_{g,1,k}$ and no single variable monomials $x_{g,m,k}$ with $m>1$ or $x_{g',m',k'}$ with $(g',k')<(g,k)$.
Such an element is evidently sent to something nonzero by $\varphi_{g,k}\circ f$, contradicting the fact that it lies in $\ker f$.
We thus have $\ker f=0$.
\end{proof}

\begin{theorem}\label{maininjectivity}
The ring homomorphism $R\to H^*_c(\cZ_\semiFano/\Cpx_3;\QQ)$ which sends the formal variable $x_{g,m,k}\in R$ to the equivariant local curve element $x_{g,m,k}\in H^*_c(\cZ_\semiFano/\Cpx_3;\QQ)$ is a bi-algebra morphism compatible with $\rho_d$ and is injective.
\end{theorem}

\begin{proof}
This ring homomorphism is a bi-algebra homomorphism by Lemma \ref{coproductlocalcurve} and the fact that $H^*_c(\cZ_\semiFano/\Cpx_3)$ vanishes in negative virtual dimension (note the virtual dimension of $x_{g,m,k,\ell}$ is $-2\ell$) by Theorem \ref{grothendiecktransverse} and Proposition \ref{vdimpositive}.
It is also with the operations $\rho_d$ on both sides by Lemma \ref{rholocalcurve} and again the fact that $H^*_c(\cZ_\semiFano/\Cpx_3)$ vanishes in negative virtual dimension (we have $\rho_d(x_{g,m,k})=x_{g,m/d,k,(1-1/d)mk}$, whose virtual dimension is negative except when $d=1$ or $mk=0$).

Now to show injectivity, we apply Theorem \ref{likemilnormoore}, according to which it suffices to define, for every $g,k\geq 0$, a group homomorphism $H^*_c(\cZ_\semiFano/\Cpx_3)\to\QQ$ sending $x_{g,1,k}$ to $1$, sending $x_{g',m',k'}$ to $0$ for $(g',k')>(g,k)$, and sending any monomial in equivariant local curve elements of degree $>1$ to zero.
Such a homomorphism is provided by Lemma \ref{gwdetect} below.
\end{proof}

\begin{lemma}\label{gwdetect}
Let $g,k\geq 0$.
There exists a group homomorphism $\GW_{g,k}:H^*_c(\cZ/\Cpx_3)\to\QQ$ such that $\GW_{g,k}(x_{g,1,k})=1$, $\GW_{g,k}(x_{g',m',k'})=0$ for $(g',k')>(g,k)$, and $\GW_{g,k}$ evaluates to zero on any monomial in equivariant local curve elements of degree $>1$.
\end{lemma}

\begin{proof}
Let $\GW_{g,k}$ integrate over the virtual fundamental class of the moduli space $(\cM_g')_{c_1=k}$ of non-constant stable maps from connected nodal curves of arithmetic genus $g$ representing a homology class with chern number $k$.
Since the image of a connected space is connected, $\GW_{g,k}$ annihilates monomials in equivariant local curve elements of degree $>1$.

We have $\GW_{g,k}(x_{g',m',k'})=0$ if $g'>g$, since there are no non-constant maps from a nodal curve of arithmetic genus $g$ to a curve of genus $g'>g$.
In the case $g=g'$, the map would have to have degree $d=k/k'<1$, hence cannot exist.

Finally, let us calculate $\GW_{g,k}(x_{g,1,k})=1$.
According to Lemma \ref{eqlocalcurveassocgraded}, the local curve element $x_{g,1,k}$ is the Poincaré dual of a point in the transversely cut out moduli space $H^0(C,E)$ of sections of a local curve $E\to C$ with $H^1(C,E)=0$.
It thus pairs against the fundamental class (which is the virtual fundamental class) of $H^0(C,E)$ to give $1$.
\end{proof}

\appendix

\section{Virtual fundamental classes}\label{vfc}

We give a brief exposition of the theory of the intrinsic normal cone, perfect obstruction theories, and virtual fundamental classes as pioneered by Behrend--Fantechi \cite{intrinsicnormalcone}.
Naturally, we give particular emphasis to the properties of this theory which we need for this paper (our need is, in fact, confined to Section \ref{vfcsubsec}).
References include Behrend--Fantechi \cite{intrinsicnormalcone}, Manolache \cite{manolache}, Qu \cite{quvirtualpullback}, Khan \cite{khanvirtual,khanrydhvirtual}, Déglise--Jin--Khan \cite{deglisejinkhan}, and Porta--Yu \cite{portayu}.

A morphism of complex analytic spaces is called a \emph{closed embedding} when it is the inclusion of a closed analytic subspace, i.e.\ what is usually called a `closed immersion' (a term which we will avoid since it conflicts with the meaning of the term `immersion' in differential topoogy).
If $X\to Y$ is a closed embedding, then $\Bl_XY=\Proj_Y\bigoplus_{r\geq 0}I_X^r$ denotes the blow-up of $Y$ along $X$.
For any cartesian diagram of complex analytic spaces
\begin{equation}
\begin{tikzcd}
X'\ar[d,hook]\ar[r]&X\ar[d,hook]\\
Y'\ar[r]&Y
\end{tikzcd}
\end{equation}
in which the vertical arrows are closed embeddings, there is a (functorial) closed embedding of blow-ups $\Bl_{X'}Y'\to\Bl_XY\times_YY'$, corresponding to the surjection of graded rings $\bigoplus_{r\geq 0}I_X^r\otimes_{\cO_Y}\cO_{Y'}\twoheadrightarrow\bigoplus_{r\geq 0}I_{X'}^r$.

A \emph{complex analytic stack} shall mean a stack (i.e.\ sheaf of groupoids) on the site of complex analytic spaces which admits a submersive atlas (hence has representable diagonal).
A morphism of complex analytic stacks is called \emph{Deligne--Mumford} when its diagonal is \emph{unramified} (a morphism of complex analytic spaces is called unramified when it is, locally on the source, a closed embedding).
The diagonal of any morphism of complex analytic spaces is a locally closed embedding, hence every morphism of complex analytic spaces is Deligne--Mumford.

\begin{definition}[Deformation to the normal cone]
Given $X\hookrightarrow Y$ a closed embedding of complex analytic spaces, one associates the following spaces \cite[Chapter 5]{fulton}.
\begin{align}
M_{X/Y}&=\Bl_{X\times 0}(Y\times\AA^1)\\
M_{X/Y}^\circ&=M_{X/Y}\setminus\Bl_{X\times 0}(Y\times 0)\\
C_{X/Y}&=M^\circ_{X/Y}\times_{\AA^1}0
\end{align}
Note that $\Bl_{X\times 0}(Y\times 0)\to\Bl_{X\times 0}(Y\times\AA^1)$ is a closed embedding; explicitly $M_{X/Y}^\circ=\Spec_Y\cO_Y[t,I_Xt^{-1}]$.
The object $C_{X/Y}$ is called the \emph{normal cone} of $X\hookrightarrow Y$, and the space $M_{X/Y}^\circ$ is called the \emph{deformation to the normal cone} (it maps to $\AA^1$ with fiber $Y$ over $\AA^1-0$ and fiber $C_{X/Y}$ over $0$).
Given $X\to Y$ a Deligne--Mumford morphism of complex analytic stacks, the above objects are defined by descent \cite{intrinsicnormalcone,kreschartin,kreschrational,kimkreschpantev}\cite[\S 2.2]{manolache}\cite[\S 1.1]{quvirtualpullback}.
\end{definition}

\begin{remark}\label{swdualrmk}
The normal cone $C_{X/Y}$ is an algebro-geometric analogue of the relative Spanier--Whitehead dual of $X$ over $Y$.
\end{remark}

\begin{lemma}[{\cite[Theorem 2.31]{manolache}\cite[Proposition 1.2]{quvirtualpullback}}]
For any cartesian square
\begin{equation}
\begin{tikzcd}
X'\ar[d]\ar[r]&X\ar[d]\\
Y'\ar[r]&Y
\end{tikzcd}
\end{equation}
whose vertical arrows are Deligne--Mumford, there is an induced closed embedding $M_{X'/Y'}^\circ\to M_{X/Y}^\circ\times_YY'$ over $\AA^1$, hence also a closed embedding $C_{X'/Y'}\to C_{X/Y}\times_YY'$.
\end{lemma}

\begin{proof}
For a closed embeddings $X\hookrightarrow Y$, this corresponds to the surjection of rings $\cO_Y[t,I_Xt^{-1}]\otimes_{\cO_Y}\cO_{Y'}\twoheadrightarrow\cO_{Y'}[t,I_{X'}t^{-1}]$.
Now apply descent.
\end{proof}

\begin{lemma}\label{deformationproduct}
For maps $X\to Y$ and $X'\to Y'$, there is a canonical isomorphism
\begin{equation}
M_{X\times X'/Y\times Y'}^\circ\xrightarrow\sim M_{X/Y}^\circ\times_{\AA^1}M_{X'/Y'}^\circ
\end{equation}
which, in particular, specializes over $0$ to an isomorphism $C_{X\times X'/Y\times Y'}\xrightarrow\sim C_{X/Y}\times C_{X'/Y'}$.
\end{lemma}

\begin{proof}
For closed embeddings $X\hookrightarrow Y$ and $X'\hookrightarrow Y'$, this map corresponds to the map of graded rings 
\begin{equation}
\cO_{Y\times Y'}[t,I_X\cO_{Y'}t^{-1},t',\cO_YI_{X'}(t')^{-1}]/(t-t')\to\cO_{Y\times Y'}[t,(I_X\cO_{Y'}+\cO_YI_{X'})t^{-1}],
\end{equation}
which is an isomorphism by inspection (to see this, it is helpful to choose vector space splitings of the inclusions $I^r\supseteq I^{r+1}$).
Now apply descent.
\end{proof}

Recall the derived category $D(X)$ of (unbounded) complexes of sheaves of $\ZZ$- or $\QQ$-modules on a complex analytic stack $X$ (references include \cite{laszloolsson}).
At the level of $\infty$- or dg-categories, the category $D(X)$ associated to a stack $X$ is the limit of $D(U)$ over all complex analytic spaces $U$ with a map to $X$ (equivalently, all those with a submersive map to $X$).
For a map of complex analytic stacks $f:X\to Y$, we have pairs of adjoint functors $(f^*,f_*)$ and $(f_!,f^!)$ (always derived) between $D(X)$ and $D(Y)$.
Note that we need the exceptional pushforward/pullback operations for non-separated morphisms.
These satisfy proper base change, in the sense that for a cartesian square
\begin{equation}
\begin{tikzcd}
X'\ar[d,"f'"]\ar[r,"\beta"]&X\ar[d,"f"]\\
Y'\ar[r,"\alpha"]&Y
\end{tikzcd}
\end{equation}
there is a canonical isomorphism $f'_!\beta^*=\alpha^*f_!$ (equivalently $\beta_*f^{\prime!}=f^!\alpha_*$).
There is also a canonical natural transformation $f_!\to f_*$, which is an isomorphism for proper representable morphisms $f$.
More generally, it is an isomorphism for proper Deligne--Mumford morphisms $f$ provided we are using $\QQ$-coefficients (to see the need for $\QQ$-coefficients, note that for $f:*/G\to *$ with $G$ finite, the natural transformation $f_!\to f_*$ is the natural map $V_G\to V^G$ from co-invariants to invariants given by `sum over $G$-orbits' for $G$-representations $V$).

Recall the bivariant group $H_*^\relinfty(X/Y)=H^*(X,\pi^!\underline\ZZ_Y)=a_*\pi_*\pi^!a^*\ZZ$ associated to a morphism $\pi:X\to Y$ (where $a:Y\to *$) in Definition \ref{bivariantmaindef}.
Whereas Definition \ref{bivariantmaindef} was about `nice' maps of topological spaces, we consider here the same definition but for (possibly non-separated) morphisms of complex analytic stacks.

\begin{definition}\label{specializationgeneral}
Fix a map $f:X\to Y$ and a map $\pi:D\to Y\times\AA^1$ which is an isomorphism over $Y\times(\AA^1-0)$ and whose fiber over $Y\times 0$ is identified (over $Y$) with $X$.
Associated to this data is a canonical \emph{specialization element} in $H_0^\relinfty(X/Y)$ defined as follows (note that this definition is somewhat different from its analogue in the context of algebraic cycles \cite{manolache,quvirtualpullback,khanvirtual}).
In fact, there is an associated natural transformation of functors $1\to f_*f^!$.
Let $(\AA^1-0)^\sim$ denote the universal cover, and consider the following diagram.
\begin{equation}
\begin{tikzcd}
X\ar[r,"i"]\ar[d,"f"']&D\ar[d,"\pi"]&Y\times(\AA^1-0)^\sim\ar[l,"j"']\ar[d,equal]\\
Y\ar[r,"i"]&Y\times\AA^1&\ar[l,"j"']Y\times(\AA^1-0)^\sim
\end{tikzcd}
\end{equation}
There is an induced diagram of derived categories, which commutes by proper base change.
\begin{equation}
\begin{tikzcd}
D(X)\ar[r,"i_*"]\ar[d,leftarrow,"f^!"']&D(D)\ar[d,leftarrow,"\pi^!"]&D(Y\times(\AA^1-0)^\sim)\ar[l,"j_*"']\ar[d,equal]\\
D(Y)\ar[r,"i_*"]&D(Y\times\AA^1)&\ar[l,"j_*"']D(Y\times(\AA^1-0)^\sim)
\end{tikzcd}
\end{equation}
Now we combine this on the bottom with the natural transformation $j_*p^*\to i_*$ where $p:Y\times(\AA^1-0)^\sim\to Y$, and we post-compose with the pushforward $D(D)\to D(Y)$.
This produces the desired a natural transformation $1\to f_*f^!$ by contractibility of $(\AA^1-0)^\sim$.
\end{definition}

\begin{definition}
The \emph{specialization element} in $H_0^\relinfty(C_{X/Y}/Y)$ refers to that associated to the family $M_{X/Y}^\circ\to Y\times\AA^1$ by Definition \ref{specializationgeneral}.
\end{definition}

\begin{remark}
Continuing Remark \ref{swdualrmk}, the specialization element in $H_*^\relinfty(C_{X/Y}/Y)$ is analogous to the Spanier--Whitehead dual (relative $Y$) of the map $X\to Y$.
\end{remark}

\begin{lemma}\label{vfcpullback}
For a cartesian square
\begin{equation}
\begin{tikzcd}
X'\ar[d]\ar[r]&X\ar[d]\\
Y'\ar[r]&Y
\end{tikzcd}
\end{equation}
whose vertical arrows are Deligne--Mumford, the specialization elements in $H_*^\relinfty(C_{X/Y}/Y)$ and $H_*^\relinfty(C_{X'/Y'}/Y')$ have the same image in $H_*^\relinfty(C_{X/Y}\times_YY'/Y')$.
\end{lemma}

\begin{proof}
Diagram chase and compatibility of the proper base change isomorphism with composition of pullback squares.
\end{proof}

\begin{lemma}\label{vfcproduct}
For maps $X\to Y$ and $X'\to Y'$, the specialization element of $H_*^\relinfty(C_{X\times X'/Y\times Y'}/Y\tighttimes Y')=H_*^\relinfty(C_{X/Y}\times C_{X'/Y'}/Y\tighttimes Y')$ is the image of the tensor product of the specialization elements of $H_*^\relinfty(C_{X/Y}/X)$ and $H_*^\relinfty(C_{X'/Y'}/Y')$ under the Künneth map.\qed
\end{lemma}

Now we recall the notion of a \emph{perfect obstruction theory} from Behrend--Fantechi \cite{intrinsicnormalcone}, which allows to transfer the specialization element in $H_0^\relinfty(C_{X/Y}/Y)$ to a class in $H_d^\relinfty(X/Y)$ when the `virtual relative tangent bundle' $T_{X/Y}^\vir$ has rank $d$.

\begin{definition}\label{twotermtotalspace}
Given a two-term complex of vector bundles (aka perfect complex of amplitude $[-1\;0]$) $E^\bullet=[E^{-1}\to E^0]$ over a complex analytic stack $X$ (by which we mean an object in the sheafified 2-category of such), we may form the `total space' $E^\bullet$ given locally in the case $X$ is a space by the groupoid $E^{-1}\times_XE^0\arrowstack{r,r}E^0$ \cite[\S 2]{intrinsicnormalcone}.
\end{definition}

\begin{definition}\label{potnondef}
A \emph{perfect obstruction theory} on a Deligne--Mumford morphism $X\to Y$ is an amplitude $[0\;1]$ perfect complex $T_{X/Y}^\vir$ on $X$ together with a closed embedding $C_{X/Y}\to T_{X/Y}^\vir[1]$ (i.e.\ into the `total space' of $T_{X/Y}^\vir[1]$).
\end{definition}

\begin{remark}
Definition \ref{potnondef} is non-standard and is made only for the sake of simplicity in the present discussion.
A perfect obstruction theory (in the standard meaning as introduced by Behrend--Fantechi \cite{intrinsicnormalcone}) on a Deligne--Mumford morphism $X\to Y$ is an amplitude $[0\;1]$ perfect complex $T_{X/Y}^\vir$ on $X$ together with a linear map $T_{X/Y}[1]\to T_{X/Y}^\vir[1]$ which is a closed embedding on total spaces, where $T_{X/Y}[1]$ denotes the `total space of the dual of $(\tau^{\geq-1}\LL_{X/Y})[-1]$'.
There is a canonical closed embedding $C_{X/Y}\to T_{X/Y}[1]$ (see \cite{intrinsicnormalcone,manolache}), so a perfect obstruction theory in the sense of \cite{intrinsicnormalcone} induces one in the sense of Definition \ref{potnondef}.
\end{remark}

Given (the total space of) a two-term complex $\pi:E^\bullet\to X$, the functor $\pi_*\pi^!$ is shift by $\dim E^\bullet=\dim E^0-\dim E^{-1}$.
In particular, for a morphism $X\to Y$ this identifies $H_{*+\dim E^\bullet}^\relinfty(E^\bullet/Y)$ with $H_*^\relinfty(X/Y)$.

\begin{definition}\label{vfcdef}
The \emph{relative virtual fundamental class} $[X/Y]^\vir\in H_{\dim T_{X/Y}^\vir}^\relinfty(X/Y)$ associated to a perfect obstruction theory $C_{X/Y}\to T_{X/Y}^\vir[1]$ is the image of the canonical degree zero element in $H_*^\relinfty(C_{X/Y}/Y)$ under the proper pushforward map to $H_*^\relinfty(T_{X/Y}^\vir[1]/Y)$, followed by the identification of this group with $H_{*+\dim T_{X/Y}^\vir}^\relinfty(X/Y)$ from just above.
\end{definition}

A perfect obstruction theory on $X\to Y$ determines one on $X'=X\times_YY'\to Y'$ for any morphism $Y'\to Y$, namely we take $T_{X'/Y'}^\vir$ to be the pullback of $T_{X/Y}^\vir$, and we consider the composition of the closed embedding $C_{X'/Y'}\to C_{X/Y}\times_YY'$ with the pullback to $Y'$ of the map $C_{X/Y}\to T_{X/Y}^\vir[1]$.

\begin{lemma}\label{vfcpullbackpot}
Fix a cartesian diagram
\begin{equation}
\begin{tikzcd}
X'\ar[d]\ar[r]&X\ar[d]\\
Y'\ar[r]&Y
\end{tikzcd}
\end{equation}
whose vertical arrows are Deligne--Mumford, and fix a perfect obstruction theory $C_{X/Y}\to T_{X/Y}^\vir[1]$ with pullback $C_{X'/Y'}\to C_{X\times Y}\times_YY'\to T_{X/Y}^\vir[1]\times_YY'=T_{X'/Y'}^\vir[1]$.
The pullback map $H_*^\relinfty(X/Y)\to H_*^\relinfty(X'/Y')$ sends $[X/Y]^\vir$ to $[X'/Y']^\vir$.
\end{lemma}

\begin{proof}
Immediate from Lemma \ref{vfcpullback}.
\end{proof}

Given perfect obstruction theories on $X\to Y$ and $X'\to Y'$, their product is a perfect obstruction theory on $X\times X'\to Y\times Y'$ (recall the isomorphism $C_{X\times X'/Y\times Y'}\xrightarrow\sim C_{X/Y}\times C_{X'/Y'}$ of Lemma \ref{deformationproduct}).

\begin{lemma}\label{vfcproductpot}
For Deligne--Mumford morphisms $X\to Y$ and $X'\to Y'$ with perfect obstruction theories, the virtual fundamental class $[X\times X'/Y\times Y']^\vir$ is the image of $[X/Y]^\vir\otimes[X'/Y']^\vir$ under the Künneth map $H_*^\relinfty(X/Y)\otimes H_*^\relinfty(X'/Y')\to H_*^\relinfty(X\times X'/Y\times Y')$.
\end{lemma}

\begin{proof}
Immediate from Lemma \ref{vfcproduct}.
\end{proof}

Lemmas \ref{vfcpullbackpot} and \ref{vfcproductpot} together imply the compatibility of virtual fundamental classes with \emph{fiber} products.

\bibliographystyle{amsplain}
\bibliography{mnopcy}

\end{document}